%% file: USD2-Infinitesimal-extensions.tex
\documentclass[oneside,a4paper]{amsart}

\input{preamble}

\begin{document}

\title[Infinitesimal extensions]{Unstable synthetic deformations II: \\ Infinitesimal extensions}
\author{William Balderrama}
\author{Piotr Pstr\k{a}gowski}

\begin{abstract}
This paper is the second in a series devoted to the study of unstable synthetic deformations through the lens of Malcev theories: certain $\infty$-categorical algebraic theories $\pcat$ with well-behaved $\infty$-categories $\Model_\pcat$ of models. In this paper, we show that Malcev theories and their models admit a well-behaved \emph{deformation theory}, generalizing the classical deformation theory of rings and modules.

As our main example, we prove that the Postnikov tower of a Malcev theory $\pcat$ is a tower of square-zero extensions, and that all of this structure is preserved by passage to $\infty$-categories of models. This allows us to control the difference between the $\infty$-categories $\Model_{\h_{n+r}\pcat}$ and $\Model_{\h_n\pcat}$ for $r \leq n$, and forms the basis of a ``cofibre of $\tau$'' formalism in our approach to unstable synthetic homotopy theory. As an application, we derive from this a variety of new Blanc--Dwyer--Goerss style decompositions of moduli spaces of lifts along the tower $\Model_\pcat\to\cdots\to\Model_{\h\pcat}$.
\end{abstract}

\maketitle 

\tableofcontents

\section{Introduction}

This is the second in a series of papers that develops a theory of \emph{unstable synthetic deformations}: deformations of unstable homotopy theories which categorify spectral sequences and obstruction theories. We refer the reader to \cite[Section 1.2]{usd1} for a broad overview of this project. In \cite{usd1}, we defined and studied an $\infty$-categorical and infinitary generalization of the \emph{Malcev theories} of classical universal algebra \cite{smith1976malcev,borceuxbourn2004malcev}, refining an earlier such generalization given in \cite{balderrama2021deformations}, and developed the higher universal algebra of their $\infty$-categories of \emph{models}.

Just as classical theories and models generalize rings and modules, $\infty$-categorical theories and models generalize connective ring spectra and connective modules: associated to any connective ring spectrum $A$ is a Malcev theory $\cfrees(A)$ satisfying $\Model_{\cfrees(A)}\simeq\LMod_A^{\geq 0}$. A key difference is that in higher universal algebra, this generalization is sufficiently robust as to encode \emph{deformations} of both nonconnective and nonstable homotopy theories, including certain filtered deformations of stable $\infty$-categories \cite[Section 9.2]{usd1} and deformations which encode generalizations of the classical Blanc--Dwyer--Goerss \cite{realization_space_of_a_pi_algebra} theory of moduli problems in homotopy theory, previously explored in \cite{pstrkagowski2023moduli,balderrama2021deformations}.

This paper pushes the analogy between theories and connective ring spectra even further. Our goal is to show that Malcev theories and their models admit a categorified \emph{deformation theory}, generalizing the classical deformation theory of rings and modules. This theory forms the underpinning of an unstable ``cofibre of $\tau$'' formalism which is central to the study of synthetic deformations. This formalism has been an important ingredient in many recent breakthroughs in the knowledge of the stable homotopy groups of spheres \cite{burklund2024classical, lin2024last, burklund2025adams}, and the formalism presented here is built with analogous applications in the unstable world in mind. 

The basic idea is as follows. A Malcev theory $\pcat$, as an $\infty$-category, admits a \emph{Postnikov tower}
\[
\pcat\to\cdots\to\h_{n+1}\pcat\to\h_n\pcat\to\cdots\to\h\pcat,
\]
which is again a tower of Malcev theories. After passing to $\infty$-categories of models, this induces what we call the \emph{spiral tower}, or \emph{Goerss--Hopkins tower}, of $\pcat$:
\[
\Model_\pcat\to\cdots\to\Model_{\h_{n+1}\pcat}\to\Model_{\h_n\pcat}\to\cdots\to\Model_{\h\pcat}.
\]
We show that the spiral tower of $\pcat$ can be completed to cartesian \emph{spiral squares} 
\begin{equation}\label{eq:basicspiralsquare}\begin{tikzcd}
\Model_{\h_{n+1}\pcat}\ar[r]\ar[d] &\Model_{\h\pcat}\ar[d]\\
\Model_{\h_n\pcat}\ar[r,"k_!"]&\Model_{\kinv_{n,1}^{n+1}\pcat}
\end{tikzcd},\end{equation}
induced by corresponding categorical \emph{Postnikov squares} of $\pcat$. The theory $\kinv_{n,1}^{n+1}\pcat$ which serves as the target of the corresponding $k$-invariant is in a precise sense \emph{$\integers$-linear} over the homotopy category $\h\pcat$, and these squares realize the spiral tower of $\pcat$ as a tower of \emph{linear extensions} (or \emph{square-zero extensions}) with essentially algebraic layers. The spiral tower of $\pcat$ therefore realizes the $\infty$-category $\Model_\pcat$ as an infinitesimal deformation of $\Model_{\h\pcat}$ in the same way that the Postnikov tower realizes a connective ring spectrum $A$ as an infinitesimal deformation of its $0$th homotopy ring $\pi_0 A\simeq A_{\leq 0}$.

As we will explain, these squares and variants both recover the classical deformation theory of rings and modules, and lead to an extremely flexible framework for constructing decompositions of moduli spaces and mapping spaces in homotopy theory. Our goal for the rest of the introduction is to describe this and our main theorems in more detail.

\subsection{The spiral tower of a Malcev theory}\label{sssec:spiraltower}

The Postnikov squares of a Malcev theory are a special case of a more general categorical construction, inspired by classical work of Dwyer--Kan--Smith \cite{dwyerkansmith1986obstruction} on Postnikov towers of simplicial categories and more recent work of Harpaz--Nuiten--Prasma \cite{harpaznuitenprasma2020kinvariants} on Postnikov towers of $\infty$-categories. 
Any $\infty$-category $\ccat$ admits a tower of homotopy $n$-categories 
\[
\ccat\to\cdots\to\h_{n+1}\ccat\to\h_n\ccat\to\cdots\to\h\ccat.
\]
In \S\ref{sec:modifications}, we study the following generalization of this construction. Let $\spaces_0\subset\spaces$ be a subcategory closed under finite products and let $F\colon \spaces_0\to\spaces$ be a functor which preserves finite products. 

\begin{defn}[Informal]
\label{definition:introduction_modification}
Let $\ccat$ be an $\infty$-category with mapping spaces in $\spaces_0$. The \emph{modification of $\ccat$ along $F$} is a new $\infty$-category $\ccat_F$ with, informally, the same objects as $\ccat$, but mapping spaces given by the formula 
\[
\map_{\ccat_F}(a,b) := F(\map_\ccat(a,b)).
\]
Composition is induced by composition in $\ccat$, using the fact that $F$ preserves products.
\end{defn}
The definition above is informal, and we formalize it in \cref{def:modification} using the theory of enriched $\infty$-categories developed by Gepner--Haugseng \cite{gepnerhaugseng2015enriched}: in short, $\ccat$ is an $\infty$-category enriched in $\spaces_0$, and $\ccat_F$ is its change of enrichment along $F$. In practice, the details of the construction are rarely needed, and in \cref{prop:modificationproperties} we give a more precise formulation of the above informal description which suffices for most purposes.

\begin{ex}
\label{ex:modifytruncate}
If $\tau_{<n} \colon \spaces \rightarrow \spaces$ is the Postnikov $(n-1)$-truncation, then $\ccat_{\tau_{<n}}\simeq\h_n\ccat$. That is, the homotopy $n$-category is an example of a modification in the sense of \cref{definition:introduction_modification}. 
\end{ex}

Classical Postnikov theory asserts that if $X$ is a space, then the truncation $X_{\leq n} \to X_{<n}$ is classified by a $k$-invariant: roughly, a class in $\Hrm^{n+1}(X_{<n};\pi_n X)$, where the cohomology groups are possibly with local coefficients. We revisit and extend this classical theory in \S\ref{sec:postnikovsquares}, showing that the truncation $X_{<n+r} \to X_{<n}$ is classified by linear data for any $1 \leq r < n$, and when $r = n$ when certain Whitehead products vanish in the homotopy groups of $X$. The classical theory is obtained as the special case when $r=1$. Modifying a Malcev theory along these constructions leads to the following.

\begin{theorem}[\S\ref{ssec:categoricalpostnikovsquares}]
\label{introthm:catpost}
Let $\pcat$ be a Malcev theory and fix integers $1\leq r \leq n < \infty$. Then the truncation
\[
\tau_{(n+r,n)}\colon \h_{n+r}\pcat\to\h_n\pcat
\]
is a linear extension of $\infty$-categories in the following sense: there is a spectrum object
\[
\kinv_{n,r}^\bullet\pcat \in \Mod_{\thesphere_{<r}}(\Sp(\malcevtheories_{/\h_r\pcat})),
\]
in $\infty$-categories over $\h_r\pcat$, linear over the truncation $\thesphere_{< r}$ of the sphere spectrum, together with cartesian Postnikov squares
\begin{center}\begin{tikzcd}
\h_{n+r}\pcat\ar[r,"\tau_{(n+r,r)}"]\ar[d,"\tau_{(n+r,n)}"'] &\h_r\pcat\ar[d,"0"]\\
\h_n\pcat\ar[r,"k"]&\kinv_{n,r}^{n+1}\pcat
\end{tikzcd}.\end{center}
\end{theorem}

Here, the map $k$ can be thought of as a suitable cohomology class classifying $\h_{n+r} \pcat$ as an $\infty$-category over $\h_{n} \pcat$, see \cref{ex:univtoda,rmk:bwcohomology}.

\begin{ex}
\label{ex:cnringpostnikov}
Let $A$ be a connective ring spectrum. If $1 \leq r \leq n < \infty$, then it is known that $A_{<n+r}$ is a square-zero extension of $A_{<n}$ by the $A_{<r}$-bimodule $A_{[n,n+r)} = \tau_{<n+r}\tau_{\geq n}A$. For $r = 1$, this goes back in various forms to work of Kriz, Basterra, Lazarev, and Dugger--Shipley \cite{basterra1999andre, lazarev2001homotopy, duggershipley2006postnikov}, with the general case shown by Lurie in \cite[{Corollary 7.4.1.27}]{higher_algebra}\footnote{More precisely, Lurie shows that an $n$-small extension of ring spectra is square-zero. As $A_{<n+r} \rightarrow A_{<n}$ is $n$-small for any $1 \leq r \leq n$, the result follows.}. 

The Postnikov squares of a Malcev theory of \cref{introthm:catpost} are a nonabelian generalization of these results. In particular, under the embedding $\cfrees(\bs)$ of connective ring spectra into single-sorted Malcev theories, we have
\begin{center}\begin{tikzcd}
\h_{n+r}\cfrees(A)\ar[r]\ar[d] &\h_r\cfrees(A)\ar[d]\\
\h_n\cfrees(A)\ar[r]&\kinv_{n,r}^{n+1}\cfrees(A)
\end{tikzcd}
$\qquad = \qquad$
\begin{tikzcd}
\cfrees(A_{<n+r})\ar[r]\ar[d] &\cfrees(A_{<r})\ar[d]\\
\cfrees(A_{<n})\ar[r]&\cfrees(A_{<r}\oplus\Sigma A_{[n,n+r)}).
\end{tikzcd}\end{center}
\end{ex}

\begin{ex}\label{ex:univtoda}
For $n = r = 1$, \cref{introthm:catpost} provides a cartesian square of the form
\begin{center}
\begin{tikzcd}
\h_2\pcat\ar[r]\ar[d]&\h\pcat\ar[d,"0"]\\
\h\pcat\ar[r,"k"]&\kinv_{1,1}^2\pcat
\end{tikzcd}.
\end{center}
The functor $k$ is a section to the projection $\kinv_{1,1}^2\pcat\to\h\pcat$, and so is classified by an element of the \emph{Baues--Wirsching cohomology} \cite{baueswirsching1985cohomology}
\[
\kappa \in \Hrm^3(\h\pcat;\Pi_1\pcat)
\]
of $\h\pcat$ with coefficients in the natural system $\Pi_1\pcat\colon \Tw(\h\pcat)\to\Ab$ of abelian groups given by
\[
(\Pi_1\pcat)(\phi\colon \nu_1 P \to \nu_1 Q) = \pi_1(\map_\pcat(P,Q),\phi).
\]
The class $\kappa$ is a \emph{universal (ternary) Toda bracket} for the theory $\pcat$. We refer the reader to \cite{bauesdreckmann1989cohomology,bauesjibladze2002classification, sagave2008universal} for further discussion.
\end{ex}

\cref{introthm:catpost} shows that there is a good Postnikov theory of Malcev theories. However we are ultimately interested not in Malcev theories themselves but their $\infty$-categories of models, and in general a cartesian square of Malcev theories need not be preserved by passage to $\infty$-categories of models. The following highly non-formal result shows that these issues cannot occur in the Postnikov case:

\begin{theorem}
\label{introthm:spiralsquares}
All of the structure on the categorical Postnikov tower of a Malcev theory $\pcat$ is preserved by passage to $\infty$-categories of models. In particular, for all $1 \leq r \leq n < \infty$, the Postnikov square of $\pcat$ induces a cartesian \emph{spiral square}
\begin{equation}\label{eq:spiralsquare}\begin{tikzcd}
\Model_{\h_{n+r}\pcat}\ar[r,"\tau_{(n+r,r)!}"]\ar[d,"\tau_{(n+r,n)!}"']&\Model_{\h_r\pcat}\ar[d,"0_!"]\\
\Model_{\h_n\pcat}\ar[r,"k_!"]&\Model_{\kinv_{n,r}^{n+1}\pcat}
\end{tikzcd}\end{equation}
realizing $\Model_{\h_{n+r}\pcat}$ as linear extension of $\Model_{\h_n\pcat}$ by an object
\[
\Model_{\kinv_{n,r}^{n+\bullet}\pcat} \in \Mod_{\thesphere_{<r}}((\catinfty)_{/\Model_{\h_r\pcat}}),
\]
itself with $\thesphere_{<r}$-module structure inherited from that on $\kinv_{n,r}^\bullet\pcat$.
\end{theorem}
The statement given above is informal and intended to give the reader an idea of the kind of results they can find in the current work. We make it more precise later in the introduction: the spiral squares of a Malcev theory $\pcat$ are a special case of a more general theory of infinitesimal extensions of Malcev theories which we will describe in \S\ref{sssec:infinitesimal}, and we specify what we consider to be ``all of the structure'' of the Postnikov tower of $\pcat$ in \S\ref{sssec:spiralsystems}. 

Before doing so, to further motivate our results, we give some examples and applications.

\begin{ex}
Let $A$ be a connective ring spectrum. Then \cref{introthm:spiralsquares} combined with \cref{ex:cnringpostnikov} asserts that the comparison map
\[
\LMod_A^{\geq 0}\to\lim_{n\to\infty}\LMod_{A_{<n}}^{\geq 0}
\]
is an equivalence, and that the squares
\begin{center}\begin{tikzcd}
\LMod_{A_{<n+r}}^{\geq 0}\ar[r]\ar[d,"A_{<n}\otimes_{A_{<n+r}}(\bs)"']&\LMod_{A_{< r}}^{\geq 0}\ar[d]\\
\LMod_{A_{<n}}^{\geq 0}\ar[r]&\LMod_{A_{<r}\oplus\Sigma A_{[n,n+r)}}^{\geq 0}
\end{tikzcd}\end{center}
are cartesian. This gives an inductive construction of the $\infty$-category of connective $A$-modules, and recovers a special case of a theorem of Lurie \cite[Theorem 16.2.0.2]{lurie_spectral_algebraic_geometry}. In fact, our results are strong enough to recover the general case, see \cref{ex:nilextension} below. 
\end{ex}

\begin{ex}
Let $R$ be a discrete commutative ring. Fix $k\geq 2$, and let $\pcat\subset\Alg_{\bfE_k}(\Mod_R^{\geq 0})$ be the full subcategory spanned by those $\bfE_k$-algebras over $R$ which are free on a set. Then there are equivalences
\[
\Model_\pcat\simeq \Alg_{\bfE_k}(\Mod_R^{\geq 0}),\qquad \Model_{\h\pcat}\simeq \CAlg_R^\Delta,
\]
where $\CAlg_R^\Delta$ is the $\infty$-category of animated commutative $R$-algebras. Thus the spiral tower
\[
\Model_\pcat\to\cdots\to\Model_{\h\pcat}\qquad = \qquad \Alg_{\bfE_k}(\Mod_R^{\geq 0}) \to \cdots \to \CAlg_R^\Delta
\]
interpolates between spectral algebra (based on $\bfE_k$-rings) and derived algebra (based on animated rings), and \cref{introthm:spiralsquares} gives a precise sense in which connective spectral algebra can be regarded as an infinitesimal extension of derived algebra.
\end{ex}

\subsection{Applications to moduli problems}\label{ssec:applications}

The spiral tower of a Malcev theory leads to an extremely flexible theory of decompositions of moduli spaces in homotopy theory. At the most basic level, this takes the following form.

\begin{theorem}[\ref{thm:modulidecomposition}]
\label{introthm:modulidecomposition}
Let $\pcat$ be a Malcev theory. Fix $\Lambda \in \Model_{\h_r\pcat}$, and let
\[
\calM(\Lambda) \subset \Model_\pcat^\core
\]
be the full subspace spanned by those $X \in \Model_\pcat$ for which $\tau_{!} X \simeq \Lambda$. Then there is a decomposition
\[
\calM(\Lambda)\simeq\lim_{r\leq n\to\infty}\calM_n(\Lambda),
\]
with layers fitting into cartesian squares
\begin{equation*}\begin{tikzcd}
\calM_{n+r}(\Lambda)\ar[r]\ar[d]&B\!\Aut(\Lambda)\ar[d,"0_!"]\\
\calM_n(\Lambda)\ar[r,"k_!"]&\map_{\h_r\pcat/\Lambda}(\Lambda,B^{n+2}_\Lambda\Lambda_{S^n})_{\h\!\Aut(\Lambda)}.
\end{tikzcd}\end{equation*}
Here, $\Lambda_{S^n} \in \Mod_{\thesphere_{<r}}^{\geq 0}((\Model_{\h_r\pcat})_{/\Lambda})$ is defined to be compatible with geometric realizations and satisfy $(\nu_r P)_{S^n}\simeq \tau_{<r}((\nu P)^{S^n})$ for $P \in \pcat$, with $\thesphere_{<r}$-module structure uniquely extending the $\bfE_n$-cogroup structure on $S^n$. 
\end{theorem}

If $r=1$, then $\Lambda $ can be often identified with an object of classical algebra, and $X$ with its \emph{realization}; that is, a homotopy-theoretic object whose algebraic invariant is equivalent to $\Lambda$. In this way, \cref{introthm:modulidecomposition} gives a decomposition of the moduli space of realizations of $\Lambda$. We illustrate this with some examples, after which we discuss how the above result significantly improves on the state of the art. 

\begin{ex}
Taking $r=1$, the $\infty$-category $\Model_{\h\pcat}$ is equivalent to the underlying $\infty$-category of Quillen's model category \cite{quillen1967homotopical} of simplicial set-valued models of $\h\pcat$ \cite[Proposition 2.2.2]{balderrama2021deformations}, and the mapping space $\map_{\h\pcat/\Lambda}(\Lambda,B^{n+2}_\Lambda\Lambda_{S^n})$ encodes the classical \emph{Quillen cohomology} of $\Lambda$ with coefficients in $\Lambda_{S^n} \in \Mod_\integers^{\geq 0}((\Model_{\h\pcat})_{/\Lambda})$: 
\[
\pi_0 \map_{\h\pcat/\Lambda}(\Lambda,B^{n+2}_\Lambda\Lambda_{S^n})_{\h\!\Aut(\Lambda)} = \Hrm^{n+2}_{\h\pcat/\Lambda}(\Lambda,\Lambda_{S^n})_{/\pi_0 \!\Aut(\Lambda)}.
\]
In this context, \cref{introthm:modulidecomposition} provides an obstruction theory to realizing a point of $\calM(\Lambda)$, i.e.\ a model of $\pcat$ lifting $\Lambda$, with obstructions living in these Quillen cohomology groups, and higher algebraic analogues for $r > 1$.
\end{ex}

\begin{ex}
Let $A$ be a connective ring spectrum. Taking $\pcat = \cfrees(A)$ to be the theory of connective $A$-modules, \cref{introthm:modulidecomposition} provides for $M \in \LMod_{A_{<n}}^{\geq 0}$ an obstruction class
\[
k_!(M) \in \pi_0 \map_{A_{<n}}(M,\Sigma^{n+2}A_{[n,n+r)} \otimes_{A_{<n}}M) = \Ext^{n+2}_{A_{<n}}(M,A_{[n,n+r)}\otimes_{A_{<n}}M)
\]
to finding a module $\widetilde{M} \in \LMod_{A_{<n+r}}^{\geq 0}$ which lifts $M$ in the sense that $A_{<n}\otimes_{A_{<n+r}} \widetilde{M} = M$.
\end{ex}

A particularly important case is when $\pcat$ is a \emph{loop theory}; that is, a Malcev theory which admits tensors by $S^1$ \cite[\S7]{usd1}. If $\Lambda \in \Model_{\h\pcat}^\heartsuit$ is a discrete model, then $\calM_n(\Lambda)$ is the space of \emph{potential $n$-stages} of $\Lambda$ in the sense of \cite{pstrkagowski2023moduli,balderrama2021deformations}, generalizing the classical definition of Blanc--Dwyer--Goerss \cite[Definition 9.1]{realization_space_of_a_pi_algebra}, and $\calM(\Lambda)$ is the space of loop models $X \in \Model_\pcat^\Omega$ for which $\pi_0 X \simeq \Lambda$. Thus \cref{introthm:modulidecomposition} for $r=1$ specializes to a variation of the decompositions of the \emph{realization space} of $\Lambda$ considered there.

We observe that \cref{introthm:modulidecomposition} generalizes the previous work of \cite{pstrkagowski2023moduli, balderrama2021deformations} discussed above in three main ways:
\begin{enumerate}
\item One need not assume that $\Lambda \in \Model_{\h\pcat}$ is discrete;
\item One need not assume that $\pcat$ is a loop theory;
\item It classifies $\calM_{n+r}(\Lambda)\to\calM_n(\Lambda)$ in terms of linear information for $1\leq r \leq n$.
\end{enumerate}
Informally, the development in \cite{pstrkagowski2023moduli, balderrama2021deformations} applies to ``classical objects'' (i.e.\ loop models), whereas the present work applies to more general ``synthetic objects'' (i.e.\ arbitrary models). The further extension to $r > 1$ is new, and is an essential piece of structure in the theory of unstable synthetic deformations: it corresponds to additional stable structure present on the unstable analogue of the cofibre $C(\tau^r)$, see the discussion at the end of this subsection.

We emphasize that removing the assumptions (1), (2) is not formal, and requires a completely different approach to the relatively hands-on constructions of \cite{pstrkagowski2023moduli, balderrama2021deformations}. Informally, while the previously known constructions were \emph{internal} to the $\infty$-category of models, and produced suitable cartesian squares there, our strategy in the current paper is to work \emph{externally}; that is, to produce the required Postnikov squares at the level of $\infty$-categories. 

The statements about moduli spaces are then produced by mere restriction. The cartesian squares of \cref{introthm:spiralsquares} induce cartesian squares of underlying $\infty$-groupoids, and \cref{introthm:modulidecomposition} follows by restricting to path components living over $\Lambda$. More precisely,
\[
\calM_1(\Lambda) = B\!\Aut(\Lambda)\qquad\text{and}\qquad \calM_n(\Lambda)= \calM_1(\Lambda)\times_{\Model_{\h\pcat}}\Model_{\h_n\pcat},
\]
and the result follows from a diagram chase which ultimately identifies
\[
\calM_1(\Lambda)\times_{\Model_{\h\pcat}}\Model_{\kinv_{n,r}^{n+1}\pcat}\simeq \map_{\h\pcat/\Lambda}(\Lambda,B^{n+2}_\Lambda\Lambda_{S^n})_{\h\!\Aut(\Lambda)}.
\]

The external approach via the general categorical deformation theory of \cref{introthm:spiralsquares} extends this further:
\begin{enumerate}[resume]
\item It gives a decomposition of \emph{$\infty$-categories}, rather than of mere spaces;
\item It constructs a decomposition in a manner which is natural in $\pcat$.
\end{enumerate}
In particular, by not restricting to underlying $\infty$-groupoids, the spiral squares of $\pcat$ allow one to access information about $\Model_\pcat$ involving non-invertible morphisms between objects, such as to derive similar decompositions of realization spaces for \emph{diagrams} of models (\cref{thm:liftdiagrams}). As a simpler example, by analyzing mapping spaces in the spiral squares of $\pcat$, we obtain the following: 

\begin{theorem}[\ref{cor:mappingspacedecomposition}, \ref{thm:linextspaceoflifts}]
\label{introthm:decomposemaps}
Fix $X,Y \in \Model_\pcat$. Then there is a decomposition
\[
\map_\pcat(Y,X)\simeq\lim_{n\to\infty}\map_{\h_n\pcat}(\tau_{n!}Y,\tau_{n!}X)
\]
of mapping spaces, with layers fitting into cartesian squares
\begin{center}\begin{tikzcd}
\map_{\h_{n+r}\pcat}(\tau_{(n+r)!}Y,\tau_{(n+r)!}X)\ar[r]\ar[d]&\map_{\h_r\pcat}(\tau_{r!}Y,\tau_{r!}X)\ar[d,"0"]\\
\map_{\h_n\pcat}(\tau_{n!}Y,\tau_{n!}X)\ar[r,"k_!"]&\map_{\h_r\pcat}(\tau_{r!}Y,B^{n+1}_{\tau_{r!}X}\tau_{r!}X_{S^n})
\end{tikzcd}\end{center}
for $1\leq r \leq n < \infty$.
\end{theorem}

In practical terms, this limit decomposition gives rise to an unstable spectral sequence calculating the homotopy groups of the mapping space $\map_\pcat(X, Y)$, while the cartesian squares allow one to identify its second page in algebraic terms and to handle differentials occuring on the fringe of the spectral sequence. In the very special case where $\pcat$ is a loop theory, $X$ and $Y$ are loop models, and $r=1$, versions of \cref{introthm:decomposemaps} were considered in \cite[Sections 4.4, 5.3]{balderrama2021deformations}. As shown there, these decompositions recover a variety of classical spectral sequences, including universal coefficient and Goerss--Hopkins style spectral sequences. 

In synthetic terms, \cref{introthm:decomposemaps} plays the role of the $\tau$-Bockstein tower, extending it to the unstable setting. In more detail, given $X,Y\in \Model_\pcat$ and a map $f\colon \tau_{n!}Y \to \tau_{n!}X$, one may define the \emph{synthetic homotopy groups} of $\map_\pcat(Y,X)$ as follows :
\begin{enumerate}
    \item If $w\leq n$ and $s=0$, or $w+r\leq n$ and $s\geq 0$, then
\[
\pi_{s,w}(\map_\pcat(Y,X)/(\tau^r),f) = \pi_s\left(\{\tau_{w!}f\}\underset{\map_{\h_w\pcat}(\tau_{w!}Y,\tau_{w!}X)}{\times} \map_{\h_{w+r}\pcat}(\tau_{(w+r)!}Y,\tau_{(w+r)!}X)\right);
\]
    \item If $r\leq \min(n,w)$ and $s \in \integers$, then
\[
\pi_{s,w}(\map_\pcat(Y,X)/(\tau^r),f) = \pi_0 \map_{\h_r\pcat/\tau_{r!}X}(\tau_{r!}Y,B_{\tau_{r!}X}^{w+s}\tau_{r!}X_{S^w}).
\]
\end{enumerate}
These two definitions agree on their overlap. Thus, for example,
\[
f \in \pi_0 \map_{\h_n\pcat}(\tau_{n!}Y,\tau_{n!}X) = \pi_{0,0}(\map_\pcat(Y,X)/(\tau^n),f),
\]
and restricting \cref{introthm:decomposemaps} to its fibre over $\tau_{(n,r)!}f$ leads to an obstruction
\[
k_!(f) \in \pi_0 \map_{\h_r\pcat/\tau_{r!}X}(\tau_{r!}Y,B^{n+1}_{\tau_{r!}X}\tau_{r!}X_{S^n}) = \pi_{-1,n}(\map_\pcat(Y,X)/(\tau^r),f)
\]
to finding a lift of $f$ to $\map_{\h_{n+r}\pcat}(\tau_{(n+r)!}Y,\tau_{(n+r)!}X)$. With more work, one can extend this obstruction theory to an \emph{unstable spectral sequence}, in an extended sense with differential information extending to negative stems in a range, such as seen for instance in work of Bousfield \cite{bousfield1989homotopy}. Details will appear elsewhere.

We emphasize the startling fact that the unstable analogue of the cofibre $C(\tau^r)$ admits \emph{negative} homotopy groups in weights $w\geq r$. and that these groups can contain useful information. This appears to be an essential and surprising feature of unstable synthetic homotopy theory. It uses in an essential way the existence of the spiral squares of \cref{introthm:spiralsquares} for $r \geq 1$.

\subsection{Infinitesimal extensions of Malcev theories}\label{sssec:infinitesimal}

We obtain the spiral squares of a Malcev theory appearing in \cref{introthm:spiralsquares} as a special case of the following more general clutching theorem:

\begin{theorem}[\ref{thm:animatesquaregeneral}]
\label{introthm:nilsquare}
Consider a cartesian square of Malcev theories and homomorphisms, and associated square of $\infty$-categories of models:
\begin{center}\begin{tikzcd}
\pcat'\ar[r]\ar[d,"f'"]&\pcat\ar[d,"f"]\\
\qcat'\ar[r,"g"]&\qcat
\end{tikzcd}
$\qquad\leadsto\qquad$
\begin{tikzcd}
\Model_{\pcat'}\ar[r]\ar[d,"f'_!"]&\Model_{\pcat}\ar[d,"f_!"]\\
\Model_{\qcat'}\ar[r,"g_!"]&\Model_\qcat
\end{tikzcd}.\end{center}
Suppose that $f$ is full and $f'$ is essentially surjective, and that \emph{at least one} of the following conditions is satisfied:
\begin{enumerate}
    \item $f_!$ reflects connectivity, i.e.\ a morphism $\alpha$ in $\Model_\pcat$ is $n$-connective if and only if $f_!\alpha$ is;
    \item $g$ is full; 
    \item $\pcat$, $\qcat$, and $\qcat'$ are additive.
\end{enumerate}
Then the square of $\infty$-categories of models is again cartesian.
\end{theorem}

Given a cartesian square of Malcev theories as in \cref{introthm:nilsquare} with $f$ full and $f'$ essentially surjective, it is not difficult to show that the induced functor
\[
\Model_\pcat \to \Model_{\pcat'}\times_{\Model_{\qcat'}}\Model_\qcat
\]
is fully faithful and admits a right adjoint (see \cref{prop:weakpb}). To prove \cref{introthm:nilsquare} one must then prove that this right adjoint is conservative, which we do by a careful analysis of connectivity properties of homomorphisms and derived functors between Malcev theories. Our analysis makes essential use of the additional conditions listed;  we do not know to what extent they may be omitted, but they are satisfied in all examples of interest.

\begin{ex}
The zero-section $0\colon \h_r\pcat \to \kinv_{n,r}^{n+1} \pcat$ appearing in \cref{introthm:catpost} is full, and its derived functor reflects connectivity as it admits a retraction. Thus, the fact that the spiral squares of $\pcat$ given in \cref{introthm:spiralsquares} are cartesian is a special case of \cref{introthm:nilsquare}.
\end{ex}

\begin{ex}\label{ex:nilextension}
Consider a cartesian square of connective ring spectra and associated square of $\infty$-categories of connective modules:
\begin{center}\begin{tikzcd}
A'\ar[r]\ar[d]&A\ar[d,"\phi"]\\
B'\ar[r]&B
\end{tikzcd}
$\qquad\leadsto\qquad$
\begin{tikzcd}
\LMod_{A'}^{\geq 0}\ar[r]\ar[d]&\LMod_A^{\geq 0}\ar[d,"B\otimes_A(\bs)"]\\
\LMod_{B'}^{\geq 0}\ar[r]&\LMod_B^{\geq 0}
\end{tikzcd}.\end{center}
In this case, \cref{introthm:nilsquare} implies that this square of $\infty$-categories is cartesian provided that $\phi$ induces a surjection $\pi_0 A \to \pi_0 B$, recovering \cite[Theorem 16.2.0.2]{lurie_spectral_algebraic_geometry}. In this sense, our result can be thought as a nonabelian generalization of the result of Lurie. 
\end{ex}

An important class of cartesian squares of connective ring spectra fitting into \cref{ex:nilextension} are those that encode \emph{square-zero extensions}. The classical theory of square-zero extensions admits a nonabelian generalization to \emph{linear extensions} of theories, studied in the setting of classical Lawvere theories by Jibladze--Pirashvili \cite{jibladzepirashvili1991cohomology, jibladzepirashvili2005linear}. The decompositions of moduli spaces described in \S\ref{ssec:applications} are derived from a more general deformation theory for linear extensions of Malcev theories, which we develop in \S\ref{sec:modulispaces} and describe now. 

To set our terminology, we recall that Harpaz--Nuiten--Prasma prove in \cite{harpaznuitenprasma2018abstract}that
\[
\Sp((\catinfty)_{/\ccat})\simeq\Fun(\Tw(\ccat),\Sp)
\]
for any $\infty$-category $\ccat$, where $\Tw(\ccat)$ is the \emph{twisted arrow category} of $\ccat$. We will refer to functors $D\colon \Tw(\ccat)\to\spectra$ as \emph{natural systems} of spectra on $\ccat$.

\begin{defn}[\ref{def:catsqz}]
\label{def:introlinext}
Let $D$ be a natural system of $R$-modules on an $\infty$-category $\ccat$ and $p\colon \bcat\to\ccat$ be an $\infty$-category over $\ccat$. A functor $f\colon \ecat\to\bcat$ is said to be a \emph{linear extension of $\bcat$ by $D$} if we have specified a cartesian square of $\infty$-categories over $\ccat$ of the form
\begin{center}\begin{tikzcd}
\ecat\ar[r,"t"]\ar[d,"f"]&\ccat\ar[d,"0"]\ar[dr,equals]\\
\bcat\ar[r,"k"]\ar[rr,bend right,"p"]&\dcat^2\ar[r,"q"]&\ccat
\end{tikzcd}.\end{center}
\end{defn}

The above notion strictly generalizes the square-zero extensions of connective ring spectra:

\begin{ex}
Let $A$ be a connective $\bfE_1$-ring spectrum. By \cite[{Theorem 7.3.4.18}]{higher_algebra}, spectrum objects in $\Alg_{\bfE_1}(\Sp)_{/A}$ are equivalent to \emph{$A$-bimodules}. A connective $A$-bimodule may be encoded by a product-preserving functor $\cfrees(A\otimes_\thesphere A^\op) \to \Sp^{\geq 0}$, from which we may define the following natural system $\tilde{I}$ on $\cfrees(A)$:
\begin{center}\begin{tikzcd}[column sep=small]
\Tw(\cfrees(A))\ar[r]&\cfrees(A)^\op\times\cfrees(A)\ar[r,"\simeq"]&(\cfrees(A)\times\cfrees(A^\op))^\op\ar[r,"\otimes_\thesphere"]&\cfrees(A\otimes_\thesphere A^\op)\ar[r,"I"]&\Sp^{\geq 0}
\end{tikzcd}.\end{center}
It follows from full and faithfulness of $\calL_0(\bs)\colon \Alg_{\bfE_1}(\spectra) \to \malcevtheories_{/\calL_0(\thesphere)}$ \cite[Remark 9.1.4]{usd1} and \cref{prop:linearextensionsaremalcev} that linear extensions of the Malcev theory $\cfrees(A)$ by $\tilde{I}$ correspond exactly to square-zero extensions of the $\bfE_1$-ring $A$ by $I$.
\end{ex}

\begin{ex}
In the above language, \cref{introthm:catpost} asserts exactly that if $\pcat$ is a Malcev theory, then for all $1\leq r \leq n$ the truncation
\[
\tau_{(n+r,n)}\colon \h_{n+r}\pcat\to\h_n\pcat
\]
is a linear extension by a certain natural system $\Tw(\h_r\pcat)\to\Mod_{\thesphere_{<r}}(\Sp)$. Taking $r = 1$ for simplicity, this natural system may be concretely described by
\[
(\phi\colon \nu_1 Q \to \nu_1 P) \mapsto \Sigma^n \pi_n(\map_\pcat(Q,P),\phi),
\]
where here we identify $\pi_{n}$ with its associated Eilenberg-MacLane spectrum. Note that this group is automatically abelian even when $n = 1$ as a consequence of the Malcev condition. 
\end{ex}

Both of the above examples have the additional property that the given cartesian square is a diagram of Malcev theories and theory homomorphisms. This is not true for a general linear extension, and comes down to the following additional property. 

\begin{defn}[\ref{def:cartesiansystems}]
Let $\pcat$ be a Malcev theory. A \emph{cartesian system} of connective spectra on $\pcat$ is a functor
\[
D\colon \Tw(\pcat)\to\Sp^{\geq 0}
\]
with the property that for all $P \in \pcat$, the composite
\begin{equation}\label{eq:nupd}
(\nu P)_D\colon \pcat_{/P}^\op \to \Tw(\pcat) \to \Sp^{\geq 0} \to \spaces
\end{equation}
defines a model of the theory $\pcat_{/P}$; that is, sends coproducts to products. 
\end{defn}

In \cref{prop:linearextensionsaremalcev}, we show that this condition ensures that every linear extension of $\pcat$ by $D$ is again a Malcev theory. Given a cartesian system $D$ on a Malcev theory $\pcat$, one may construct for each $\Lambda \in \Model_\pcat$ an object $\Lambda_D \in \Sp^{\geq 0}((\Model_\pcat)_{/\Lambda})$, extending (\ref{eq:nupd}) when $\Lambda = \nu P$ is representable. 

In these terms, \cref{introthm:spiralsquares} is a special case of the following more general result, which describes explicitly how a linear extension of Malcev theories induces a linear extension between $\infty$-categories of models: 

\begin{theorem}[\ref{thm:linearextension}]
Let $f\colon \qcat\to\pcat$ be a linear extension of Malcev theories by a cartesian system $D$ of connective spectra. Then $f_!\colon \Model_\qcat\to\Model_\pcat$ is a linear extensions of $\infty$-categories by a natural system of the form
$
D_!(\phi\colon \Gamma\to\Lambda) = \map_{/\Lambda}(\Gamma,B^\bullet_\Lambda\Lambda_D).
$
\end{theorem}

As we explain in \S\ref{ssec:linearextensions}, there is a generic obstruction theory for lifting objects along a linear extension of $\infty$-categories, see \cref{thm:spaceoflifts}. The decomposition of moduli spaces of \cref{introthm:modulidecomposition} is derived as a consequence of this more general statement, using the linear extension of $\infty$-categories of models associated to $\h_{n+r}\pcat\to\h_n\pcat$.

\subsection{Elementary modifiers and spiral systems}\label{sssec:spiralsystems}

In \cref{introthm:spiralsquares}, we asserted that \emph{all} of the structure on the categorical Postnikov of a Malcev theory $\pcat$ is preserved by passage to $\infty$-categories of models. We now describe what precisely we mean by ``all'', by introducing a novel formalization of the concept of a Postnikov tower which we refer to as a \emph{spiral system} and which captures not only the tower itself but also the various relevant cartesian squares as well as their naturality. 

Recall that the mapping spaces of a Malcev theory lie in the full subcategory $\ukanspaces\subset\spaces$ of \emph{supersimple spaces} \cite[Definition 2.2.9]{usd1}, a natural refinement of the class of spaces with vanishing Whitehead products. 

\begin{defn}[\ref{def:elementarymodifiers}]
Write $\Sph\subset\spaces$ for the full subcategory generated by $S^1$ under products, wedges, and smash products (with respect to any basepoint). The class of \emph{elementary modifiers} is the full subcategory
\[
\elementarymodifiers\subset\End(\ukanspaces)
\]
of endomorphisms of supersimple spaces generated under geometric realizations by modifiers of the form $\map(T,\bs)$ for $T \in \Sph$. 
\end{defn}

Thus, for example, the $n$-iterated free loop space $L^{n}(-) \simeq \map(S^1 \times \cdots \times S^1, -)$ is an elementary modifier. More importantly, in \S\ref{ssec:elementarypostnikov} we show that all functors appearing in the Postnikov theory of supersimple spaces may be constructed as elementary modifiers, including Postnikov truncations and the targets of functorial $k$-invariants.

The class of elementary modifiers should be thought as a natural enlargement of the class of ``Postnikov truncation-like'' functors which is convenient as it is essentially built inductively from the functor $\map(S^{1}, -)$. If $X$ is a supersimple space, then the evaluation $F \mapsto F(X)$ defines a functor 
\[
\elementarymodifiers \rightarrow \spaces 
\]
which can be thought of as encoding the Postnikov tower of $X$, together with all of its extra structure. This motivates the following definition. 

\begin{defn}[\ref{definition:spiral_system}]
A \emph{convergent spiral system} in an $\infty$-category $\dcat$ is a functor
\[
H\colon\elementarymodifiers\to\dcat
\]
which preserves levelwise pullbacks along levelwise effective epimorphisms and which satisfies $H(\id)\simeq \varprojlim H(\tau_{<n})$.
\end{defn}

A spiral system $H$ is a highly structured refinement of the tower
\[
H(\id)\to\cdots\to H(\tau_{\leq n}) \to H(\tau_{<n}) \to \cdots H(\tau_{\leq 0})
\]
which, in addition to the tower itself, also encodes the data of 
\begin{enumerate}
\item Cartesian squares which present $H(\tau_{<n+r}) \to H(\tau_{<n})$ as a square-zero, or linear, extension by a specified object of $\Mod_{\thesphere_{<r}}(\dcat_{/H(\tau_{<r})})$ for each $1 \leq r \leq n$;
\item Various compatibilities between these as $n$ and $r$ vary. 
\end{enumerate}
The promised precise reformulation of \cref{introthm:spiralsquares} is as follows: 

\begin{theorem}[\ref{thm:modelspiral}]
\label{introthm:spiralsystemmodels}
Let $\pcat$ be a Malcev theory. Then the functor 
\[
\elementarymodifiers \to \catinfty,\qquad F \mapsto \Model_{\pcat_F}
\]
is a convergent spiral system of $\infty$-categories. 
\end{theorem}

In other words, the above association is convergent in the sense that 
\[
\Model_\pcat\simeq\lim_{n\to\infty}\Model_{\h_n\pcat}, 
\]
and it takes certain cartesian squares of elementary modifiers to cartesian squares of $\infty$-categories. To prove the latter, we apply the criterion of \cref{introthm:nilsquare} together with the following highly non-formal result, which can be thought of as a version of the homology Whitehead theorem for models of Malcev theories:

\begin{theorem}[\ref{thm:homologywhitehead}]
\label{introthm:homologywhitehead}
Let $\pcat$ be a Malcev theory. For any natural transformation $\alpha\colon F \to G$ between elementary modifiers, the induced derived functor
\[
\pi_{\alpha!}\colon \Model_{\pcat_F}\to\Model_{\pcat_G}
\]
reflects connectivity: a morphism $f$ in $\Model_{\pcat_F}$ is $n$-connective if and only if $\pi_{\alpha!}f$ is $n$-connective.
\end{theorem}

Convergence is then proved \cref{thm:modelspiral}. To prove \cref{introthm:homologywhitehead}, we show that the class of homomorphisms $f\colon \pcat\to\qcat$ for which $f_!$ reflects connectivity has good closure properties: roughly, if $f$ is full then any pullback of $f$ will again reflect connectivity (see \cref{prop:thickeningcartesian}). Combined with the Postnikov tower of $\pcat$, this allows us to reduce proving that $\tau\colon\pcat\to\h\pcat$ preserves connectivity to proving that the zero-section $0\colon \h\pcat\to\kinv_{n,1}^{n+1}\pcat$ preserves connectivity, which is clear on the account of being a section. 

The general case then follows from a detailed investigation of the structure of $\tau_{\leq 0}F$ for an elementary modifier $F$. The arguments here are somewhat delicate, and take the entirety of \S\ref{subsection:path_components_of_elementary_modifiers}. In short, we make use of the existence of particularly nice cell structures on the spaces in $\Sph$ which trivialize upon mapping into any supersimple space.

To see that the above result can be thought as a nonabelian version of the classical homology Whitehead theorem, observe \cref{introthm:homologywhitehead} implies in particular that the derived truncation
\[
\tau_!\colon \Model_\pcat\to\Model_{\h\pcat}
\]
reflects connectivity, and so is conservative. If $\pcat$ is the theory of grouplike $\mathbf{E}_{2}$-algebras in spaces, then $\h\pcat$ is the theory of abelian groups, and under the Boardman-Vogt-May equivalence
\[
\Model_{\pcat} \simeq \Alg_{\mathbf{E}_{2}}^{\mathrm{grp}}(\spaces) \simeq \spaces_{\ast}^{\geq 2} 
\]
the derived truncation can be identified with a shift of reduced integral chains
\[
\Sigma^{-2}\widetilde{C}_\bullet(\bs;\integers)\colon \spaces_\ast^{\geq 2} \to \Mod_\integers^{\geq 0}. 
\]
The fact that $\widetilde{C}_\bullet(\bs;\integers)$ reflects connectivity is a classical result of Whitehead \cite{whitehead1949combinatorial}. The statement of \cref{introthm:homologywhitehead} generalizes this classical result to a general principle about Malcev theories.

\section{Modifications of \texorpdfstring{$\infty$}{infty}-categories}\label{sec:modifications}

Suppose that $F \colon \spaces \to \spaces$ is an endofunctor of spaces which preserves finite products. Given an $\infty$-category $\ccat$, we can construct a new $\infty$-category $\ccat_{F}$ informally as follows: the objects of $\ccat_F$ are the same as the objects of $\ccat$, but the mapping spaces of $\ccat_F$ are given by
\[
\map_{\ccat_F}(a,b) \colonequals F(\map_\ccat(a,b)),
\]
with composition determined by the commutative diagram 
\[
\begin{tikzcd}[column sep=3mm]
	{\map_{\ccat_{F}}(b, c) \times \map_{\ccat_F}(a,b) } & {F(\map_{\ccat}(b, c)) \times F(\map_{\ccat}(b, a))} & {F(\map_{\ccat}(b, c) \times \map_{\ccat}(a, b)) } \\
	 {\map_{\ccat_{F}}(a, c)} && {F(\map_{\ccat}(a, c)) },
	\arrow["\simeq", no head, from=1-1, to=1-2]
	\arrow[from=1-1, to=2-1, dashed]
	\arrow["\simeq", no head, from=1-2, to=1-3]
	\arrow[from=1-3, to=2-3]
	\arrow["\simeq", no head, from=2-1, to=2-3]
\end{tikzcd}
\]
where the right vertical arrow is induced by composition in $\ccat$. We refer to $\ccat_{F}$ as the \emph{modification of $\ccat$ along $F$}. 

\begin{ex}
Modifying along the $(n-1)$-th Postnikov truncation
\[
\tau_{<n} \colon \spaces\rightarrow\spaces
\]
has the effect of passing to homotopy $n$-categories; that is, we have 
\[
\ccat_{\tau_{< n}} \simeq \h_{n}\ccat.
\]
\end{ex}

Following the above example, for an arbitrary modifier $F$ the assignment $\ccat \mapsto \ccat_F$ can be thought of as a generalized form of passage to the homotopy category. The goal of this section is to carefully formalize this construction and study its properties. 

In \cite{gepnerhaugseng2015enriched}, Gepner--Haugseng develop a theory of \emph{enriched $\infty$-categories}. This theory is self-consistent: $\infty$-categories are equivalent to $\infty$-categories enriched in $\spaces$ with its cartesian monoidal structure. Given a product-preserving functor $F\colon \spaces\to\spaces$, the modification $\ccat_F$ can be formally defined as the $\spaces$-enriched $\infty$-category obtained by changing the enrichment of $\ccat$ along $F$. We review what we need from the theory of enriched $\infty$-categories in \S\ref{ssec:enrichedcategories}, allowing us to make this definition in \cref{def:modification}.

In practice, the specifics of the construction are rarely needed. We give a formalization of the above informal description of $\ccat_F$ in \cref{prop:modificationproperties} which is sufficient for most purposes. After these basic properties are set up, in \S\ref{ssec:modificationlimits} we study how the construction $\ccat \mapsto \ccat_F$ behaves with respect to limits and colimits in $F$, and in \S\ref{ssec:modificationmalcev} we begin our study of modifications of Malcev theories. In particular, in \S\ref{ssec:elemmodifiers} we introduce a class of modifiers, the \emph{elementary modifiers}, which will play an important role in our study of the Postnikov and spiral towers of a Malcev theory.

\begin{remark}[{Previous work}] 
A form of the construction $\ccat \mapsto \ccat_{F}$ was first formalized by Dwyer--Kan--Smith in the setting of simplicial categories with a fixed set of objects \cite{dwyerkansmith1986obstruction}, where it was used to construct the Postnikov tower of a simplicial category, including $k$-invariants. 

More recently, Harpaz--Nuiten--Prasma \cite{harpaznuitenprasma2020kinvariants} used the theory of enriched $\infty$-category theory, as we do, to construct and study the Postnikov tower of an $(\infty,n)$-category.
\end{remark}

\subsection{Enriched \texorpdfstring{$\infty$}{infty}-category theory} 
\label{ssec:enrichedcategories}

We begin by recalling some of the relevant definitions for enriched $\infty$-categories, just enough as needed to define the change of enrichment functors that we will need.

\begin{recollection}[{\cite[Section 4]{gepnerhaugseng2015enriched}}]
Let $\vcat$ be a monoidal $\infty$-category.
\begin{enumerate}
\item Associated to every space $\sfO$ is a nonsymmetric $\sfO\times \sfO$-colored $\infty$-operad $\calO_\sfO$. Informally, an algebra $A\in \Alg_{\calO_\sfO}(\vcat)$ consists of a functor
\[
A \colon \sfO \times \sfO \rightarrow \vcat,
\]
together with composition morphisms
\[
A(b,c) \otimes A(a,b) \rightarrow A(a,c)
\]
in $\vcat$ which are suitably unital and associative.
\item The assignment $\sfO \mapsto \Alg_{\calO_\sfO}$ is contravariantly functorial in $\sfO$. Its associated cartesian unstraightening
\[
\Alg_\Cat(\calV)\rightarrow \spaces
\]
is the $\infty$-category of \emph{categorical $\vcat$-algebras}. Informally, a categorical $\vcat$-algebra is a pair $\sfC = (\Ob(\sfC),\map_\sfC)$ consisting of a space $\Ob(\sfC)$ of objects and a $\calO_{\Ob(\sfC)}$-algebra $\map_\sfC$.
\end{enumerate}
\end{recollection}

\begin{ex}\label{ex:categoricalspcalgebras}
If $\vcat = \spaces$ with its cartesian monoidal structure, then:
\begin{enumerate}
\item Given a space $\sfO$, the $\infty$-category $\Alg_{\calO_\sfO}(\spaces)$ of $\calO_\sfO$-algebras is equivalent to the $\infty$-category of Segal spaces
\[
\sfC_\bullet\colon \Delta^\op\rightarrow\spaces
\]
together with an equivalence $\sfC_0 \simeq \sfO$ \cite[Corollary 4.4.4]{gepnerhaugseng2015enriched}.
\item The $\infty$-category $\Alg_\Cat(\spaces)$ of categorical $\spaces$-algebras is equivalent to the $\infty$-category of Segal spaces.
\end{enumerate}

By \cite{ayalafrancis2018flagged}, Segal spaces $\sfC_\bullet$ are equivalent to \emph{flagged $\infty$-categories}: $\infty$-categories $\ccat$ equipped with an essentially surjective map $\sfO\rightarrow\ccat$ with $\sfO$ an $\infty$-groupoid. Under this equivalence, $\ccat$ is the Rezk completion of $\sfC_\bullet$ and $\sfO = \sfC_0$. In particular, $\calO_\sfO$-algebras are equivalent to $\infty$-categories flagged by $\sfO$.
\end{ex}

As this example shows, categorical $\calV$-algebras are \emph{not} the same as $\calV$-enriched $\infty$-categories, but may be thought of as $\calV$-enriched $\infty$-categories with a fixed space of objects. To obtain $\calV$-enriched $\infty$-categories, one must impose an additional completeness condition.

\begin{recollection}
\label{rec:enrichedinftycats}
Let $\vcat$ be a monoidal $\infty$-category. 
\begin{enumerate}
\item An \emph{equivalence} in a categorical $\calV$-algebra $\sfC$ is a map $E^1 \rightarrow \sfC$, where $E^1$ is the categorical $\calV$-algebra with two objects which is constant on the monoidal unit of $\calV$ \cite[Definition 5.1.8]{gepnerhaugseng2015enriched}. We write $\iota_1\sfC = \map_{\Alg_\Cat(\calV)}(E^1,\sfC)$ for the space of equivalences in $\sfC$.
\item A categorical $\calV$-algebra $\sfC = (\Ob(\sfC),\map_\sfC)$ is said to be \emph{Rezk complete} if the canonical map
\[
\Ob(\sfC) \rightarrow \iota_1 \sfC
\]
is an equivalence \cite[Definition 5.2.2, Corollary 5.2.10]{gepnerhaugseng2015enriched}. 
\item A \emph{$\calV$-enriched $\infty$-category} is a Rezk complete categorical $\calV$-algebra. The inclusion of the full subcategory
\[
\catinfty(\calV)\hookrightarrow \Alg_\Cat(\calV)
\]
spanned by the $\calV$-enriched $\infty$-categories admits a left adjoint
\[
L_\Rezk\colon \Alg_\Cat(\calV)\rightarrow\catinfty(\calV),
\]
which we call \emph{Rezk completion} \cite[Section 5.6]{gepnerhaugseng2015enriched}.
\item Rezk completion realizes $\catinfty(\calV)$ as the localization of $\Alg_\Cat(\calV)$ at those maps $f\colon \sfC \rightarrow\sfD$ of categorical $\calV$-algebras which are
\begin{itemize}
\item Fully faithful: $\map_\sfC(a,b) \rightarrow \map_\sfD(f(a),f(b))$ is an equivalence for all $a,b \in \Ob(\sfC)$;
\item Essentially surjective: if we define $\pi_0 \sfC$ to be the coequalizer of the source and target maps $\pi_0 \iota_1 \sfC \rightrightarrows \pi_0 \Ob(\sfC)$, then $\pi_0 \sfC\rightarrow\pi_0 \sfD$ is a surjection.
\end{itemize}
In particular, if $\sfC$ is any categorical $\calV$-algebra then the completion map $\sfC\rightarrow L_{\Rezk}\sfC$ is fully faithful and essentially surjective.
\end{enumerate}
\end{recollection}

\begin{ex}
As discussed in \cref{ex:categoricalspcalgebras}, the $\infty$-category $\Alg_\Cat(\spaces)$ of categorical $\spaces$-algebras is equivalent to the $\infty$-category of Segal spaces. Under this equivalence, a categorical $\spaces$-algebra is Rezk complete in the sense of \cref{rec:enrichedinftycats} exactly when it is a complete Segal space in the sense of Rezk \cite{rezk2001model}. In particular, $\spaces$-enriched $\infty$-categories are a model for $\infty$-categories.

More generally, if $\spaces_0\subset\spaces$ is any full subcategory closed under finite products, then $\spaces_0$-enriched $\infty$-categories may be identified as the full subcategory of $\catinfty$ spanned by those $\infty$-categories $\ccat$ whose mapping spaces all lie in $\spaces_0$.
\end{ex}

\begin{recollection}
\label{rec:changeofenrichment}
The formation of categorical $\vcat$-algebras is functorial in $\vcat$ by \cite[{Lemma 4.3.9}]{gepnerhaugseng2015enriched}. In more detail, given a lax monoidal functor $F \colon \calV_1\rightarrow\calV_2$:
\begin{enumerate}
\item For any space $\sfO$, there is a natural functor
\[
\Alg_{\calO_\sfO}(F)\colon \Alg_{\calO_\sfO}(\calV_1)\rightarrow\Alg_{\calO_\sfO}(\calV_2)
\]
given, informally, by postcomposition with $F$.
\item As $\sfO$ varies, these assemble into a functor
\[
\Alg_\Cat(F)\colon \Alg_\Cat(\calV_1)\rightarrow\Alg_\Cat(\calV_2).
\]
\item The functor $\Alg_\Cat(F)$ does not preserve Rezk complete objects, but is compatible with Rezk completion, meaning there is a unique functor
\[
\catinfty(F) = L_\Rezk \circ \Alg_\Cat(F)
\]
making the diagram
\begin{center}\begin{tikzcd}[column sep=large]
\Alg_\Cat(\calV_1)\ar[r,"\Alg_\Cat(F)"]\ar[d,"L_\Rezk"]&\Alg_\Cat(\calV_2)\ar[d,"L_\Rezk"]\\
\catinfty(\calV_1)\ar[r,"\catinfty(F)"]&\catinfty(\calV_2)
\end{tikzcd}\end{center}
commute. This follows from the fact that Rezk completion can be identified with localization at the essentially surjective fully faithful functors, and these are preserved by $\Alg_\Cat(F)$ by \cite[{Lemma 5.7.5}]{gepnerhaugseng2015enriched}. 
\end{enumerate}
\end{recollection}

\begin{defn}\label{def:changeofenrichment}
Given a lax monoidal functor $F\colon \vcat_1\to\vcat_2$, we refer to
\[
\catinfty(F)\colon \catinfty(\vcat_1)\to\catinfty(\vcat_2)
\]
as \emph{change of enrichment} along $F$.
\end{defn}

\subsection{The modification of an \texorpdfstring{$\infty$}{infty}-category}\label{ssec:modifyinftycats}

Let $\spaces_0\subset\spaces$ be a subcategory closed under finite products, and let $F \colon \spaces_0\to\spaces$ be a functor which preserves finite products. Fix an $\infty$-category $\ccat$ with mapping spaces in $\spaces_0$, which may therefore be identified as a $\spaces_0$-enriched $\infty$-category.

\begin{defn}\label{def:modification}
The \emph{modification of $\ccat$ along $F$} is the $\infty$-category
\[
\ccat_F \colonequals \catinfty(F)(\ccat)
\]
obtained by changing the enrichment of $\ccat$ along $F$ as in \cref{def:changeofenrichment}. 
\end{defn}

\begin{rmk}
The modification construction is evidently natural in $F$, assembling into a functor
\[
\Fun^\times(\spaces_0,\spaces) \rightarrow \Fun(\catinfty(\spaces_0),\catinfty)
\]
where $\Fun^\times(\spaces_0,\spaces)\subset \Fun(\spaces_0,\spaces)$ is the full subcategory of functors $\spaces_0\to\spaces$ which preserve finite products.
\end{rmk}

Our next goal is to establish some properties of modification that, in most cases, allow one to work with modified categories while avoiding the specific technical details of the construction.

\begin{lemma}
\label{lem:uniqueassembly}
Let $i\colon \spaces_0\hookrightarrow \spaces$ denote the inclusion. Then there is a unique natural transformation $i\rightarrow F$, given on an object $X\in \spaces_0$ by the comparison map
\[
X \simeq \colim_X \pt  \simeq \colim_X F(\pt)\rightarrow F(\colim_X \pt) \simeq F(X).
\]
\end{lemma}
\begin{proof}
This lemma requires only the assumption that $\spaces_0\subset\spaces$ contains the terminal object, the one-point space $\pt$, and that $F$ preserves it. As $\spaces_0$ contains the terminal object, the inclusion $i\colon \spaces_0\hookrightarrow\spaces$ can be identified as the left Kan extension of the inclusion $\{\pt\}\hookrightarrow \spaces$ along $\{\pt\}\hookrightarrow\spaces_0$. It follows that
\[
\map_{\Fun(\spaces_0,\spaces)}(i,F)\simeq \map_{\Fun(\{\pt\},\spaces)}(i|_{\{\pt\}},F|_{\{\pt\}})\simeq F(\pt)\simeq \pt
\]
as claimed, with the last equivalence applying the fact that $F$ preserves the terminal object. Tracing through this identification yields the given description of the unique transformation $i\rightarrow F$.
\end{proof}

\begin{construction}
For any $\infty$-category $\ccat$ with mapping spaces in $\spaces_0$, we write
\[
\pi_F\colon \ccat\rightarrow\ccat_F
\]
for the functor determined by the identification $\ccat\simeq \ccat_i$ and unique natural transformation $i\rightarrow F$.
\end{construction}

In most cases, the following proposition is sufficient to work with $\ccat_F$.

\begin{prop}
\label{prop:modificationproperties}
Let $\ccat$ be an $\infty$-category with mapping spaces in $\spaces_0$.
\begin{enumerate}
\item The functor $\pi_F\colon \ccat\rightarrow\ccat_F$ is essentially surjective;
\item For any $a,b\in \ccat$, we may identify
\[
\map_{\ccat_F}(\pi_F(a),\pi_F(b)) \simeq F(\map_\ccat(a,b)),
\]
and $\pi_F\colon \map_\ccat(a,b)\rightarrow \map_{\ccat_F}(\pi_F(a),\pi_F(b))\simeq F(\map_\ccat(a,b))$ is the comparison map of \cref{lem:uniqueassembly}.
\end{enumerate}
\end{prop}
\begin{proof}
By definition, $\ccat_F$ is constructed as a composition
\[
\ccat_F = L_\Rezk(\Alg_\Cat(F)(\ccat)),
\]
and $\pi_F$ factors as a composite of maps
\[
L \circ \pi_F' \colon \ccat\rightarrow\Alg_\Cat(F)(\ccat)\rightarrow L_\Rezk(\Alg_\Cat(F)(\ccat))
\]
of categorical $\spaces$-algebras. The claimed properties of $\pi_F$ are clear for $\pi_F'\colon \ccat\rightarrow\Alg_\Cat(F)(\ccat)$: the categorical $\spaces$-algebra $\Alg_\Cat(F)(\ccat)$ is constructed to satisfy $\Ob(\Alg_\Cat(F)(\ccat)) = \Ob(\ccat)$ and $\map_{\Alg_\Cat(F)(\ccat)}(a,b) = F(\map_\ccat(a,b))$, and the natural transformation $\ccat \rightarrow \Alg_\Cat(F)(\ccat)$ is constructed to act on mapping spaces through the natural maps provided by \cref{lem:uniqueassembly}. As the Rezk completion map $L\colon \Alg_\Cat(F)(\ccat)\rightarrow L_{\Rezk}(\Alg_\Cat(F)(\ccat))$ is fully faithful and essentially surjective, it follows that these properties hold also for $\pi_F$.
\end{proof}

\begin{cor}
If $\tau_{< n}\colon \spaces\rightarrow\spaces$ is the Postnikov truncation, then
\[
\pi_{\tau_{< n}}\colon \ccat\rightarrow\ccat_{\tau_{< n}}
\]
realizes 
\[
\ccat_{\tau_{< n}}\simeq \h_{n}\ccat
\]
as the homotopy $n$-category of $\ccat$.
 \qed
\end{cor}

\begin{cor}
\label{corrolary:explicit_formula_for_the_monad_associated_to_a_modification}
Let $\ccat$ be an $\infty$-category, and consider the left Kan extension and restriction adjunction
\[
\pi_{F!} : \presheaves(\ccat)\rightleftarrows \presheaves(\ccat_F) : \pi_F^\ast.
\]
The composite
\[
\pi_F^\ast\pi_{F!}\colon \presheaves(\ccat)\rightarrow\presheaves(\ccat)
\]
is the unique colimit-preserving functor extending
\[
F\circ y \colon \ccat\rightarrow\presheaves(\ccat)\rightarrow \presheaves(\ccat),\qquad c \mapsto F(\map_\dcat(\bs,c)).
\]
\end{cor}

\begin{proof}
If $f\colon \ccat\rightarrow\dcat$ is any functor, then
\[
f^\ast f_!\colon \presheaves(\ccat)\rightarrow\presheaves(\ccat)
\]
is the unique colimit-preserving functor extending
\[
c \mapsto \map_\ccat(f(\bs),f(c)).
\]
The corollary then follows from the identification $\map_{\ccat_F}(\pi_F(\bs),\pi_F(c))\simeq F(\map_\ccat(\bs,c))$.
\end{proof}

\begin{ex}
Suppose that $F(X) = \map(T,X)$ for a fixed space $T$. If $\ccat$ admits constant colimits indexed over $T$, then
\[
\pi_F^\ast\pi_{F!}\colon \presheaves(\ccat)\rightarrow\presheaves(\ccat)
\]
is equivalent to precomposition with the constant $T$-colimit functor
\[
T\otimes (\bs)\colon \ccat\rightarrow\ccat.
\]
\end{ex}

\subsection{Limits of modifications}\label{ssec:modificationlimits}

Let $\spaces_0\subset\spaces$ be a subcategory closed under finite products as before. The full subcategory $\Fun^\times(\spaces_0,\spaces) \subset \Fun(\spaces_0,\spaces)$ of finite product-preserving functors is closed under all small limits. We now analyze some of the extent to which the modification functor
\[
\Fun^\times(\spaces_0,\spaces)\rightarrow\Fun(\catinfty(\spaces_0),\catinfty)
\]
preserves these. In other words, we study the extent to which the construction $\ccat \mapsto \ccat_{F}$ preserves limits in $F$ (not to be confused with limits in $\ccat$). We start with the following.

\begin{prop}
\label{prop:modifierlimitsff}
Fix a diagram $F_\bullet\colon \calJ \rightarrow \Fun^\times(\spaces_0,\spaces)$. Then the comparison map
\[
p\colon \ccat_{\lim F_\bullet}\rightarrow \lim \ccat_{F_\bullet}
\]
is always fully faithful.
\end{prop}
\begin{proof}
For each $j \in \calJ$, there is a diagram
\begin{center}\begin{tikzcd}
&\ccat\ar[dl,"\pi_{\lim F_\bullet}"']\ar[d,"\pi"]\ar[dr,"\pi_{F_j}"]\\
\ccat_{\lim F_\bullet}\ar[r,"p"]&\lim \ccat_{F_\bullet}\ar[r]&\ccat_{F_j}
\end{tikzcd}.\end{center}
Given $a,b\in \ccat$, as mapping spaces in a limit of categories are computed pointwise, we may identify
\begin{align*}
\map_{\ccat_{\lim F_\bullet}}(\pi_{\lim F_\bullet}(a),\pi_{\lim F_\bullet}(b)) & \simeq \lim F_\bullet(\map_\ccat(a,b))\\
&\simeq  \lim \map_{\ccat_{F_\bullet}}(\pi_{F_\bullet}(a),\pi_{F_\bullet}(b))\simeq \map_{\lim \ccat_{F_\bullet}}(\pi(a),\pi(b)).
\end{align*}
As $\pi_{\lim F_\bullet}$ is essentially surjective, it follows that 
\[
\map_{\ccat_{\lim F_\bullet}}(a,b)\simeq \map_{\lim \ccat_{F_\bullet}}(p(a),p(b))
\]
for any $a,b\in \ccat_{\lim F_\bullet}$, proving that $p$ is always fully faithful.
\end{proof}

It is more delicate to determine when $p\colon \ccat_{\lim F_\bullet}\rightarrow\lim \ccat_{F_\bullet}$ is essentially surjective, and thus an equivalence.

\begin{ex}\label{ex:trivialaction}
If $i\colon \spaces_0\hookrightarrow \spaces$ is the inclusion and $T$ is a space, then the comparison functor associated to the constant diagram $T\to\Fun^\times(\spaces_0,\spaces)$ on $i$ takes the form
\[
\ccat_{\map(T,\bs)}\simeq \ccat_{\lim_T i}\to \lim_T \ccat_i \simeq \ccat^T.
\]
This functor is fully faithful, and realizes $\ccat_{\map(T,\bs)}$ as the full subcategory of $\ccat^T$ spanned by the constant functors $T \to \ccat$. In particular, it is generally not essentially surjective.
\end{ex}

The difficulty in determining when $p$ is essentially surjective arises, at least in part, from the fact that in general there seems no easy way to describe the maximal subgroupoid $(\ccat_F)^\core\subset\ccat_F$. Our next goal is to identify some restrictions on $F$ which allow this space to be nicely identified.

For the rest of this subsection, we suppose that $\spaces_0\subset\spaces$ is closed under subobjects, in the sense that if $X \coprod Y \in \spaces_0$ then $X \in \spaces_0$. This ensures that if $\ccat$ is enriched in $\spaces_0$, then so is $\ccat^\core\subset\ccat$. As an $\infty$-category $\ccat$ is an $\infty$-groupoid if and only if the shearing maps
\[
\map_\ccat(a,b)\times\map_\ccat(a,b) \to \map_\ccat(a,b)\times\map_\ccat(a,a),\qquad (f,g)\mapsto(f,gf)
\]
are all equivalences, we see that modification preserves $\infty$-groupoids. In particular, modifying the inclusion $\ccat^\core \to \ccat$ provides a functor $(\ccat^\core)_F \to \ccat_F$ which factors through the maximal subgroupoid of $\ccat_F$, proving a comparison map
\[
(\ccat^\core)_F \to \ccat_F^\core.
\]
This map is not always an equivalence, but it will be for the modifying functors $F$ that are of most interest to us.

\begin{prop}\label{prop:modifycore}
Let $F \in \Fun^\times(\spaces_0,\spaces)$.
\begin{enumerate}
\item If $F$ preserves coproducts, then there is a unique natural transformation $F \to \tau_{\leq 0}$.
\item If there exists a natural transformation $F\rightarrow\tau_{\leq 0}$, then the map $\pi_F\colon\ccat\rightarrow\ccat_F$ induces a bijection $\tau_{\leq 0}(\ccat^\core) \cong \tau_{\leq 0}(\ccat_F^\core)$.
\item Suppose that there exists a natural transformation $F \to \tau_{\leq 0}$, and that $\tau_{\leq 0} F \to \tau_{\leq 0}$ admits the structure of a group object in $\Fun^\times(\spaces_0,\spaces)_{/\tau_{\leq 0}}$. Then the associated functor $p\colon \ccat_F\rightarrow\h\ccat$ is conservative.
\item If $F$ preserves coproducts and the natural transformation $\ccat_F \to \h\ccat$ guaranteed by (1) is conservative, then $(\ccat^\core)_F\simeq \ccat_F^\core$.
\end{enumerate}
\end{prop}
\begin{proof}
(1)~~In fact this does not require that $F$ preserves products. As $F$ preserves coproducts and $\spaces_0\subset\spaces$ is closed under subobjects, $F$ is determined by its restriction to the full subcategory $\spaces_0^\cn\subset\spaces_0$ of connected spaces in $\spaces_0$. As $\tau_{\leq 0}\colon \spaces_0^\cn \to \spaces$ is the terminal functor, there is a unique natural transformation $F \to \tau_{\leq 0}$.

(2)~~Given a natural transformation $F\rightarrow \pi_0$, we obtain a factorization
\[
\ccat\rightarrow\ccat_F\rightarrow\h\ccat,
\]
and thus a retraction
\[
\tau_{\leq 0} \ccat^\core \rightarrow \tau_{\leq 0}\ccat_F^\core\rightarrow\tau_{\leq 0}\h\ccat^\core \cong \tau_{\leq 0}\ccat^\core,
\]
implying that $\pi_F$ induces an injection $\tau_{\leq 0} \ccat^\core \rightarrow \tau_{\leq 0} \ccat_F^\core$. As $\pi_F$ is essentially surjective, this is also a surjection and thus a bijection.

(3)~~Our assumptions provide for every $a,b\in\ccat$ a bundle
\begin{center}\begin{tikzcd}
\pi_0 F(\map_\ccat(a,b))\ar[d,"p"]\\
\pi_0\map_\ccat(a,b)\ar[u,"\pi_F",bend left]
\end{tikzcd}\end{center}
of groups. Fix $f\in F(\map_\ccat(a,b))$ for which $p(f)$ is an equivalence. We must show that $f$ is an equivalence. As $p(f)$ is an equivalence, we may find $p(f)^{-1} \in \pi_0 \map_\ccat(b,a)$, yielding $g = \pi_F(p(f)^{-1}) \circ f \in \pi_0 F(\map_\ccat(a,a))$ satisfying $p(g) = \id_a$. By the Eckmann--Hilton argument, the monoid structure on the kernel of $\pi_0 F(\map_\ccat(a,a))\rightarrow\pi_0 \map_\ccat(a,a)$ arising from composition agrees with the group structure arising from the group operation on $\pi_0 F$ over $\pi_0$. In particular the kernel is a group under composition, and therefore we may construct
\[
h = (\pi_F (p(f)^{-1})\circ f)^{-1} \circ \pi_F(p(f)^{-1}) \in \pi_0 F(\map_\ccat(b,a)),
\]
satisfying $h\circ f = \id_{\pi_F(a)}$. The same argument shows that $f$ has a right inverse, and is thus invertible as claimed.

(4)~~As $F$ preserves coproducts, the square
\begin{center}\begin{tikzcd}
F(\iso_\ccat(a,b))\ar[r]\ar[d]&F(\map_\ccat(a,b))\ar[d,"p"]\\
\pi_0\iso_\ccat(a,b)\ar[r]&\pi_0 \map_\ccat(a,b)
\end{tikzcd}\end{center}
is cartesian. It therefore suffices to show that this square remains cartesian with $F(\iso_\ccat(a,b))$ replaced by $\iso_{\ccat_F}(a,b)$. As $\iso_{\ccat_F}(a,b)\rightarrow F(\map_\ccat(a,b))$ is an inclusion of path components, this is the claim that if $f\in F(\iso_\ccat(a,b))$ and $p(f)$ is invertible, then $f$ is invertible; in other words, that $\ccat_F\rightarrow \h\ccat$ is conservative.
\end{proof}

\begin{prop}\label{prop:modifierlimitses}
Fix a diagram $F_\bullet\colon \calJ\rightarrow\Fun^\times(\spaces_0,\spaces)$ satisfying the following conditions:
\begin{enumerate}
\item $\calJ$ is connected;
\item The comparison map $(\ccat^\core)_{F_j}\to \ccat_{F_j}^\core$ is an equivalence for all $j\in\jcat$;
\item The functor $\pi_{F_j}\colon \ccat\to \ccat_{F_j}$ induces a bijection $\pi_0\ccat^\core \cong \pi_0 \ccat_{F_j}^\core$ for all $j\in \jcat$;
\item For any loop space $X \in \spaces_0$, the comparison $B \lim F_\bullet X \rightarrow \lim BF_\bullet X$ is an equivalence. This includes the following cases:
\begin{enumerate}
\item $F_\bullet$ is a tower $F_0\leftarrow F_1\leftarrow F_2\leftarrow\cdots$ and $\pi_0 F_\bullet$ is eventually constant;
\item $F_\bullet$ is a span $F_0\rightarrow F_1\leftarrow F_2$ and $\pi_0F_0\rightarrow\pi_0F_1$ is a surjection.
\end{enumerate}
\end{enumerate}
Then $p\colon \ccat_{\lim F_\bullet}\rightarrow\lim \ccat_{F_\bullet}$ is an equivalence and $\ccat_{\lim F_\bullet}^\core \simeq (\ccat_{\lim F_\bullet})^\core$ for any $\ccat$.
\end{prop}
\begin{proof}
It suffices to prove just that $p$ is an equivalence for any $\ccat$, for then
\[
(\ccat_{\lim F_\bullet})^\core \simeq (\lim \ccat_{F_\bullet})^\core \simeq \lim ((\ccat_{F_{\bullet}})^\core) \simeq  \lim( (\ccat^\core)_{F_\bullet}) \simeq (\ccat^\core)_{\lim F_\bullet}.
\]
By \cref{prop:modifierlimitsff}, to prove that $p$ is an equivalence it suffices to prove that $p$ is essentially surjective. To that end, it suffices to show that the composite
\[
\ccat^\core_{\lim F_\bullet}\rightarrow (\ccat_{\lim F_\bullet})^\core \rightarrow (\lim \ccat_{F_\bullet})^\core \simeq \lim(\ccat_{F_\bullet})^\core \simeq \lim (\ccat^\core)_{F_\bullet}
\]
is an equivalence, i.e.\ we reduce to the case where $\ccat$ is an $\infty$-groupoid. By choosing a section of $\ccat^\core\rightarrow\pi_0\ccat^\core$, we may identify this map as
\[
\coprod_{x\in \pi_0 \ccat^\core} B \lim F_\bullet \aut_\ccat(x) \rightarrow \lim \coprod_{x\in \pi_0\ccat^\core} B F_\bullet \aut_\ccat(x).
\]
As $\calJ$ is connected, we may identify
\[
\lim \coprod_{x\in \pi_0\ccat^\core} B F_\bullet \aut_\ccat(x) \simeq \coprod_{x\in \pi_0\ccat^\core} \lim B F_\bullet \aut_\ccat(x),
\]
so it suffices to prove that if $x\in \ccat$ then the map
\[
B \lim F_\bullet \aut_\ccat(x)\rightarrow \lim BF_\bullet\aut_\ccat(x),
\]
is essentially surjective. This holds under the assumption that $\lim BF_\bullet\aut_\ccat(x)$ is connected.
\end{proof}

Finally, we record how modifications of $\ccat$ behaves with respect to limits and colimits in $\ccat$.

\begin{lemma}\label{prop:colimitsinmodification}
Fix a diagram $\calJ$, and suppose that $F \in \Fun^\times(\spaces_0,\spaces)$ preserves $\calJ$-shaped limits. Then $\pi_F\colon \ccat\rightarrow\ccat_F$ preserves all $\calJ$-shaped limits and colimits.
\end{lemma}
\begin{proof}
As $(\ccat_F)^\op\simeq (\ccat^\op)_F$, it suffices to just prove that $\pi_F$ preserves all $\calJ$-shaped colimits. Given a diagram $a_\bullet\colon \calJ\rightarrow\ccat$ admitting a colimit in $\ccat$, we can compute
\begin{align*}
\map_{\ccat_F}(\pi_F(\colim a_\bullet),\pi_F(b)) &\simeq F(\map_\ccat(\colim a_\bullet,b))\simeq F(\lim \map_\ccat(a_\bullet,b))\\
&\simeq \lim F(\map_\ccat(a_\bullet,b))\simeq \lim \map_{\ccat_F}(\pi_F(a_\bullet),\pi_F(b))
\end{align*}
for any $b\in \ccat$. As $\pi_F$ is essentially surjective, it follows that
\[
\map_{\ccat_F}(\pi_F(\colim a_\bullet),c)\simeq \lim \map_{\ccat_F}(\pi_F(a_\bullet),c)
\]
for all $c\in \ccat_F$, realizing $\pi_F(\colim a_\bullet)$ as a colimit of the diagram $\pi_F(a_\bullet)$.
\end{proof}

\subsection{Modifications of Malcev theories}\label{ssec:modificationmalcev}

We now further specialize the above discussion to the case where $\spaces_0 = \ukanspaces$ is the $\infty$-category of supersimple spaces, introduced in \cite[Definition 2.2.9]{usd1}. By \cite[Theorem 2.4.2, Corollary 2.4.3]{usd1}, the full subcategory $\ukanspaces\subset\spaces$ of supersimple spaces is closed under products, coproducts, and retracts, and has the property that every finite product-preserving functor $F\colon \ukanspaces\to\spaces$ restricts to an endofunctor of $\ukanspaces$. This leads us to the following definition.

\begin{defn}\label{def:modifiers}
The $\infty$-category of \emph{modifiers} is the full subcategory
\[
\malcevmodifiers\subset\End(\ukanspaces)
\]
spanned by those endofunctors of $\ukanspaces$ that preserve all small products.
\end{defn}

\begin{ex}
Let $T$ be any space. Then
\[
\map(T,\bs)\colon\ukanspaces\to\ukanspaces
\]
is a modifier.
\end{ex}

\begin{ex}
For all $n \geq -2$, the $n$-truncation functor
\[
\tau_{\leq n}\colon \ukanspaces\to\ukanspaces
\]
is a modifier.
\end{ex}

\begin{ex}
Let $\calU$ be a nonprincipal ultrafilter on a set. Then the ultrapower functor
\[
\prod_\calU\colon \ukanspaces\to\ukanspaces
\]
preserves finite products but not infinite products, and so is \emph{not} a modifier in the sense of \cref{def:modifiers}.
\end{ex}

Every Malcev theory is enriched in supersimple spaces. We have required that our modifiers preserve all small products in order to make the following true.

\begin{prop}
Let $F\in \malcevmodifiers$ be a modifier. 
\begin{enumerate}
\item If $\pcat$ is a Malcev theory, then $\pcat_F$ is a Malcev theory.
\item If $f\colon \pcat \to \qcat$ is a homomorphism of Malcev theories, then $f_F\colon \pcat_F \to \qcat_F$ is a homomorphism of Malcev theories.
\end{enumerate}
In other words, modification defines a functor
\[
\malcevmodifiers \to \End(\Malc).
\]
\end{prop}
\begin{proof}
(1)~~By \cref{prop:modificationproperties} and \cref{prop:colimitsinmodification}, the functor $\pi_F\colon \pcat\to\pcat_F$ preserves coproducts and is essentially surjective, implying by \cite[Lemma 4.1.5.(3)]{usd1} that $\pcat_F$ is again a Malcev theory.

(2)~~Fix a homomorphism $f\colon \pcat\to\qcat$. We must show that $f_F\colon \pcat_F \to \qcat_F$ preserves coproducts. Given a set of objects $\{P_i : i \in I\}$ in $\pcat_F$, for each $i\in I$ we may write $P_i = \pi_F(P_i')$ for some $P_i' \in \pcat$, in which case
\begin{align*}
f_F(\coprod_{i\in I}P_i) &\simeq f_F(\coprod_{i\in I}\pi_F P_i') \simeq f_F\pi_F( \coprod_{i\in I} P_i') \\
&\simeq \pi_F f(\coprod_{i\in I}P_i')\simeq  \coprod_{i\in I}\pi_Ff (P_i') \simeq \coprod_{i\in I} f_F(P_i)
\end{align*}
by \cref{prop:colimitsinmodification}.
\end{proof}

\begin{rmk}\label{rmk:continuousmodifier}
Let $\kappa$ be a regular cardinal. Say that $F \in \malcevmodifiers$ is \emph{$\kappa$-continuous} if the following condition holds:
\begin{itemize}
\item Let $\jcat$ be a $\kappa$-filtered $\infty$-category and suppose that $X \in \Fun(\jcat,\spaces)$ admits a Malcev operation. Then the comparison map $\colim_{j\in\jcat}F(X_j) \to F(\colim_{j\in \jcat}X_j)$ is an equivalence.
\end{itemize}
If $F$ is $\kappa$-continuous and $\pcat$ is a $\kappa$-bounded Malcev theory in the sense of \cite[Definition 3.5.5]{usd1}, then it is easily seen that $\pcat_F$ is again a $\kappa$-bounded Malcev theory.
\end{rmk}

\subsection{Elementary modifiers}\label{ssec:elemmodifiers}

We now single out a particularly well behaved class of modifiers.

\begin{lemma}
The full subcategory
\[
\malcevmodifiers\subset\End(\ukanspaces)
\]
is closed under geometric realizations.
\end{lemma}
\begin{proof}
$\malcevmodifiers$ may be identified as the $\infty$-category of nonbounded models for the Malcev pretheory $(\ukanspaces)^\op$, so this follows from \cite[Proposition 4.1.7]{usd1}.
\end{proof}

As in \cite[Definition 7.1.7]{usd1}, we write
\[
\spheresandmore \subset \spaces
\]
for the full subcategory generated by $S^1$ under finite products, wedges, and smash products (with respect to any basepoint).

\begin{defn}\label{def:elementarymodifiers}\label{definition:elementary_modifier}
The $\infty$-category of \emph{elementary modifiers} is the smallest full subcategory
\[
\elementarymodifiers\subset\malcevmodifiers
\]
which is closed under geometric realizations and contains $\map(T,\bs)$ for $T \in \spheresandmore$.
\end{defn}

\begin{warning}
The $n$-sphere $S^n$ is \emph{not} supersimple unless $n \in \{0,1,3,7\}$. In particular,
\[
\map(T,\bs)\colon \ukanspaces\to\ukanspaces
\]
is rarely a corepresentable functor for $T \in \spheresandmore$.
\end{warning}

In \S\ref{ssec:elementarypostnikov}, we will prove that the class of elementary modifiers contains all modifiers relevant to the theory of Postnikov towers. We record here some general properties of elementary modifiers.

\begin{prop}\label{prop:properties_of_elementary_modifiers}
Let $F$ be an elementary modifier.
\begin{enumerate}
\item If $T \in \spheresandmore$, then $X \mapsto F(X^T)$ is elementary.
\item $F$ preserves coproducts. In particular, there is a unique natural transformation $F \to \tau_{\leq 0}$.
\item $F$ is $\omega$-continuous (see \cref{rmk:continuousmodifier}). 
\item For any $\ccat \in \catinfty(\ukanspaces)$, the projection $\ccat_F \to \h\ccat$ is conservative and the comparison map $(\ccat^\core)_F \to \ccat_F^\core$ is an equivalence.
\item The modification functor
\[
\elementarymodifiers \to \End(\catinfty(\ukanspaces)),\qquad F \mapsto (\ccat \mapsto \ccat_F)
\]
preserves levelwise pullbacks along levelwise effective epimorphisms.
\end{enumerate}
\end{prop}
\begin{proof}
(1)~~Consider the class of modifiers $F$ for which $X \mapsto F(X^T)$ is elementary. If $F(X) = \map(S,X)$ for $S\in\spheresandmore$, then
\[
F(X^T) \simeq \map(T\times S,X)
\]
is elementary, so we must show that this class of closed under geometric realizations. If $F_\bullet$ is a simplicial object in modifiers, then
\[
|F_\bullet(X^T)| \simeq |F_\bullet|(X^T),
\]
so this is clear.

(2,3)~~As these properties are preserved by geometric realizations of modifiers, it suffices to prove that if $T\in\spheresandmore$ then $\map(T,\bs)$ preserves coproducts and is $\omega$-continuous. This holds as $T$ is compact and connected. That there exists a unique natural transformation $F \to \tau_{\leq 0}$ follows from \cref{prop:modifycore}.(1).

(4)~~We must prove that $\ccat_F \to \h \ccat$ is conservative; the second statement then follows from (2) and \cref{prop:modifycore}.

First we claim that the class of modifiers $F$ over $\tau_{\leq 0}$ for which $\ccat_F\to\h\ccat$ is conservative is closed under geometric realizations. Let $F_\bullet$ be a simplicial object in $\malcevmodifiers$ and consider the commutative diagram
\begin{center}\begin{tikzcd}
\ccat_{F_0}\ar[rr]\ar[dr]&&\ccat_{|F_\bullet|}\ar[dl]\\
&\h\ccat
\end{tikzcd}.\end{center}
As $F_0 \to |F_\bullet|$ is an effective epimorphism, $\ccat_{F_0} \to \ccat_{|F_\bullet|}$ is full. As $\ccat_{F_0} \to \h\ccat$ is conservative by assumption, it follows that $\ccat_{|F_\bullet|}\to\h\ccat$ is conservative.

Next we claim that if $T$ is any connected space, then $\ccat_{\map(T,\bs)} \to \h\ccat$ is conservative. As $\ccat\to\h\ccat$ is conservative, it suffices to pick a basepoint of $T$ and verify that restriction along this basepoint defines a conservative functor $\ccat_{\map(T,\bs)}\to\ccat$. Observe that $\map(T,\bs)\simeq\lim_T \id$ can be identified as the constant limit of the identity functor indexed over $T$. It follows that $\ccat_{\map(T,\bs)}\to\ccat$ factors as
\[
\ccat_{\map(T,\bs)}\simeq \ccat_{\lim_T \id}\to \lim{}_T\,\ccat_\id\simeq \ccat^T \to \ccat,
\]
where the first possibly non-invertible map is fully faithful and therefore conservative by \cref{prop:modifierlimitsff}. As $T$ is connected, $\ccat^T\to\ccat$ is conservative, and the claim follows.

(5)~~This now follows from \cref{prop:modifierlimitses}.
\end{proof}

Although not needed for the work of this paper, we also record an additional pleasant property of elementary modifiers in the setting of higher universal algebra. We need the following.

\begin{lemma}\label{lem:constantcolimits}
Let $\ccat$ be an $\infty$-category and $T$ be a space. The following are equivalent:
\begin{enumerate}
\item $\ccat$ admits constant colimits indexed by $T$;
\item $\pi_{\map(T,\bs)}\colon \ccat \to \ccat_{\map(T,\bs)}$ admits a left adjoint.
\end{enumerate}
In this case, the left adjoint $L$ satisfies $L \pi_{\map(T,\bs)}(X)\simeq T\otimes X$ for $X \in \ccat$.
\end{lemma}
\begin{proof}
The functor $\pi_{\map(T,\bs)}\colon \ccat \to \ccat_{\map(T,\bs)}$ admits a left adjoint $L$ if and only if for every $X' \in \ccat_{\map(T,\bs)}$, the functor
\[
\map_{\ccat_{\map(T,\bs)}}(X',\pi_{\map(T,\bs)}(\bs))\colon \ccat\to\spaces
\]
is corepresentable, in which case $LX$ is the corepresenting object. As $\pi_{\map(T,\bs)}$ is essentially surjective, we may suppose without loss of generality that $X' = \pi_{\map(T,\bs)}X$ for some $X \in \ccat$. By construction we have
\[
\map_{\ccat_{\map(T,\bs)}}(\pi_{\map(T,\bs)} X,\pi_{\map(T,\bs)}(\bs)) \simeq \map_\ccat(X,\bs)^T,
\]
and this is corepresentable if and only if the constant colimit $T \otimes X = \colim_T X$ exists, in which case $T \otimes X$ is the corepresenting object. This establishes the lemma.
\end{proof}

\begin{theorem}\label{thm:wle}
Let $\alpha\colon F \to G$ be a natural transformation between elementary modifiers. If $\pcat$ is a loop theory, then $\pi_\alpha\colon \pcat_F \to \pcat_G$ is weakly left exact.
\end{theorem}
\begin{proof}
The homomorphism $\pi_\alpha\colon \pcat_F\to\pcat_G$ sits in a commutative diagram
\begin{center}\begin{tikzcd}
&\pcat\ar[dl,"\pi_F"']\ar[dr,"\pi_G"]\\
\pcat_F\ar[rr,"\pi_\alpha"]&&\pcat_G
\end{tikzcd}.\end{center}
As $\pi_F$ is essentially surjective, by \cite[Proposition 6.3.7.(3)]{usd1} we reduce to the case where $\pi_\alpha$ is of the form $\pi_F\colon \pcat\to\pcat_F$ for an elementary modifier $F$.

Without loss of generality, we may suppose that $\pcat$ is idempotent complete. In particular, by \cite[Proposition 7.1.8]{usd1} $\pcat$ admits constant colimits indexed by any space $T \in \spheresandmore$. By \cref{lem:constantcolimits}, this ensures that if $F \simeq \map(T,\bs)$ for some $T \in \spheresandmore$, then $\pi_{\map(T,\bs)}\colon \pcat \to \pcat_{\map(T,\bs)}$ is a right adjoint, and so the same is true of $\pi_{\map(T,\bs)!}\colon \Model_{\pcat} \to \Model_{\pcat_{\map(T,\bs)}}$. In particular, $\pi_{\map(T,\bs)!}$ preserves all limits, and is therefore weakly left exact.

This reduces us to proving that if $\Fss$ is a simplicial modifier for which each $\pi_{F_n}\colon \pcat\to \pcat_{F_n}$ is weakly left exact, then $\pi_{|\Fss|}\colon \pcat\to\pcat_{|\Fss|}$ is weakly left exact. Fix a cartesian square $\sigma$ in $\Model_\pcat$ in which at least one arrow of the underlying cospan is an effective epimorphism. We must show that $\pi_{|\Fss|!}\sigma$ remains cartesian. As $\pi_{|\Fss|}\colon \pcat \to\pcat_{|\Fss|}$ is essentially surjective, it suffices to prove that $\pi_{|\Fss|}^\ast\pi_{|\Fss|!}\sigma$ is cartesian. Observe that
\[
\pi_{|\Fss|}^\ast\pi_{|\Fss|!} \sigma \simeq |\pi_{\Fss}^\ast \pi_{\Fss!}\sigma|.
\]
As each $\pi_{F_n}\colon \pcat\to\pcat_{F_n}$ is weakly left exact and $\pi_{F_n}^\ast\pi_{F_n!}$ preserves effective epimorphisms, $\pi_{\Fss}^\ast \pi_{\Fss!}\sigma$ is a simplicial diagram of cartesian squares in $\Model_\pcat$ in which at least one arrow of the underlying cospan is an effective epimorphism. By \cite[Corollary 4.1.8.(2)]{usd1}, it follows that the geometric realization $|\pi_{\Fss}^\ast \pi_{\Fss!}\sigma|$ remains cartesian as needed.
\end{proof}

\section{Postnikov decompositions revisited}\label{sec:postnikovsquares}

A fundamental construction of homotopy theory associates to a space $X$ its \emph{Postnikov tower}
\[
X \to \cdots\to X_{\leq n}\to X_{<n}\to\cdots \to X_{\leq 0}.
\]
If $X$ is simply connected, then classical Postnikov theory tells us that the maps in this tower are classifed by \emph{Postnikov invariants}, or \emph{$k$-invariants}; that is, there are fibre sequences
\[
X_{\leq n} \to X_{<n} \xrightarrow{k} K(\pi_n X,n+1)
\]
which show that $X_{\leq n}$ is determined by a certain class $k\in \Hrm^{n+1}(X_{<n};\pi_n X)$. With a bit of care, this can be extended to the case when $X$ is not simply connected, by considering $\pi_n X$ not as a group but as a bundle of groups over the fundamental groupoid $X_{\leq 1}$. This classical construction presents an arbitrary space as being inductively built from linear data, one homotopy group at a time. 

Our primary goal in this section is to revisit this classical story and extend it to describe the extent to which the map $X_{<n+r} \rightarrow X_{<n}$ is similarly classified by linear data for $r > 1$. As these constructions are functorial in $X$, combined with the theory of modifications this leads to a corresponding theory of generalized Postnikov squares of $\infty$-categories. 

\subsection{Postnikov squares of spaces}\label{ssec:postnikovsquares}

Our Postnikov theory is derived from the \emph{Blakers--Massey theorem}, which we now recall.

\begin{recollection}\label{rec:tfib}
The \emph{total fibre} of a square
\begin{center}
\begin{tikzcd}
A\ar[r,"f"]\ar[d,"g"']\ar[dr, phantom, "\sigma"]&B\ar[d,"h"]\\
C\ar[r,"k"]&D
\end{tikzcd}\end{center}
of pointed spaces is defined as
\[
\tFib(\sigma) = \Fib\left( A \rightarrow B\times_D C\right) \simeq \Fib(\Fib(g)\rightarrow\Fib(h))\simeq\Fib(\Fib(f)\rightarrow\Fib(k)).
\]
\end{recollection}

\begin{theorem}[{Blakers--Massey theorem, Brown--Loday \cite[Theorem 4.2]{brownloday1987homotopical}}]\label{thm:blakersmassey}
Let $\sigma$ be a cocartesian square of pointed connected spaces as in \cref{rec:tfib}, satisfying
\[
\pi_i \Fib(f) = 0\text{ for }0 \leq i \leq n,\qquad \pi_i \Fib(g) = 0 \text{ for }0 \leq j \leq m.
\]
Then
\[
\pi_i \tFib(\sigma) = 0 \text{ for }i \leq n + m,
\]
and in the critical degree, $\pi_{n+m+1}\tFib(\sigma)$ is generated by the image of the generalized Whitehead product
\[
\pi_{n+1}\Fib(f) \times \pi_{m+1}\Fib(g) \rightarrow \pi_{n+m+1} \tFib(\sigma).
\]
In particular, $\pi_i \Fib(f)\rightarrow \pi_i \Fib(h)$ is an isomorphism for $i \leq n + m$ and a surjection for $i = n + m + 1$, with kernel generated by the image of the Whitehead product
\[
\pi_{n+1}\Fib(f)\times\pi_{m+1}\Fib(g)\rightarrow\pi_{n+m+1}\tFib(\sigma)\rightarrow\pi_{n+m+1}\Fib(g).
\]
\qed
\end{theorem}

\begin{notation}
Given a space $X$ and positive integers $n$ and $r$, write $[\pi_r X,\pi_n X] = 0$ if for every $x\in X$, $a\in \pi_r(X,x)$, $b\in \pi_n(X,x)$, the Whitehead product $[a,b] \in \pi_{r+n-1}(X,x)$ vanishes.
\end{notation}

\begin{prop}\label{prop:pushouttruncate}
Fix a space $X$ and positive integers $r$ and $n$. 
\begin{enumerate}
\item The square
\begin{center}\begin{tikzcd}
X_{< n+r}\ar[r]\ar[d]&X_{\leq r}\ar[d]\\
X_{< n}\ar[r]&(X_{< n}\cup_{X_{< n+r}}X_{\leq r})_{\leq n+r}
\end{tikzcd}\end{center}
is always cartesian.
\item The square
\begin{center}\begin{tikzcd}
X_{< n+r}\ar[r]\ar[d]&X_{<r}\ar[d]\\
X_{< n}\ar[r]&(X_{< n}\cup_{X_{<  n+r}}X_{< r})_{< n+r}
\end{tikzcd}\end{center}
is cartesian if and only if $[\pi_r X,\pi_n X] = 0$.
\end{enumerate}
\end{prop}
\begin{proof}
By working over each path component, we may suppose without loss of generality that $X$ is pointed and connected. We first prove (1). Abbreviate $C = X_{< n}\cup_{X_{< n+r}}X_{\leq r}$ and consider the cocartesian square
\begin{center}
\begin{tikzcd}
X_{< n+r}\ar[r,"f"]\ar[d,"g"]&X_{\leq r}\ar[d,"h"]\\
X_{<n}\ar[r]&C
\end{tikzcd}.
\end{center}
Write
\[
X_{[n,n+r)} = \Fib(g),\qquad X_{(r,n+r)} = \Fib(f).
\]
It follows from Blakers--Massey that
\[
\pi_i X_{[n,n+r)} \rightarrow \pi_i \Fib(h)
\]
is an isomorphism for $i< n+r$, so that there is a long exact sequence
\[
0\rightarrow\pi_{n+r}C \rightarrow\pi_{n+r-1}X_{[n,n+r)}\rightarrow\pi_{n+r-1}X_{\leq r}\rightarrow\pi_{n+r-1}C\rightarrow\cdots.
\]
As $r< n$, this splits into isomorphisms
\[
\begin{cases}
\pi_{i+1}C\cong \pi_i X & n \leq i < n+r,\\
\pi_{i+1}C\cong 0 & r \leq i < n,\\
\pi_i C \cong \pi_i X & 0 \leq i \leq r.
\end{cases}
\]
It follows that in
\begin{center}\begin{tikzcd}
X_{< n+r}\ar[r,"f"]\ar[d,"g"]&X_{\leq r}\ar[d,"h'"]\\
X_{< n}\ar[r]&C_{\leq n+r+1}
\end{tikzcd},\end{center}
the induced map $\Fib(g)\rightarrow\Fib(h')$ is an equivalence, and thus the square is cartesian.

We move on to (2). The same argument as (1) applies with $X_{<r}$ in place of $X_{\leq r}$, except that Blakers--Massey only guarantees that $\pi_i X_{[n,n+r)} \rightarrow \Fib(h)$ is an isomorphism for $i < n + r-1$ and surjection for $i = n + r - 1$. In the critical degree, there is a long exact sequence
\[
\pi_{n} X \otimes \pi_r X \rightarrow \pi_{n+r} X_{[n,n+r)}\rightarrow \pi_{n+r}\Fib(h)\rightarrow 0,
\]
with first map the Whitehead product. It is therefore an isomorphism for $i = n+r$ if and only if $[\pi_r X,\pi_n X] = 0$, and this is necessary for the square to be cartesian.
\end{proof}

\begin{construction}
\label{constr:simplepostnikov}
Fix a space $X$ and integers $1\leq r \leq n$, and suppose that $[\pi_r X, \pi_n X] = 0$. Under these assumptions, we construct, functorially in $X$, the following:
\begin{enumerate}
\item A bundle of loop spaces
\[
\Pi_{[n,n+r)}X \rightarrow X_{<r}
\]
over $X_{<r}$ with the property that, for any $x \in X_{< n+r}$, there is an equivalence of pointed spaces
\[
\{ \tilde{x} \} \times_{X_{<r}} \Pi_{[n,n+r)}X  \simeq \{ x \} \times_{X_{<n}}X_{<n+r},
\]
where $\tilde{x}$ is the image of $x$ in $X_{<r}$. 
\item A cartesian diagram
\begin{equation}
\label{equation:postnikov_squares_of_a_space}
\begin{tikzcd}
X_{< n+r}\ar[r]\ar[d,"\tau"]&X_{<r}\ar[d,"0"]\ar[dr,equals]\\
X_{< n}\ar[r,"k"]&B_{X_{<r}}\Pi_{[n,n+r)}X\ar[r,"p"']&X_{<r}
\end{tikzcd},
\end{equation}
where $B_{X_{<r}}\Pi_{[n,n+r)}X$ is the fibrewise delooping of $\Pi_{[n,n+r)}X$ over $X_{<r}$.
\end{enumerate}
To construct these, first observe that \cref{prop:pushouttruncate} produces a cartesian square of the form in (2), only with $B_{X_{<r}}\Pi_{[n,n+r)}X$ replaced by $(X_{< n}\cup_{X_{< n+r}}X_{< r})_{\leq n+r}$. The retraction
\[
X_{\leq r} \xrightarrow{\tilde{0}} (X_{< n}\cup_{X_{< n+r}}X_{<r})_{\leq n+r} \xrightarrow{\tilde{p}} X_{<r}
\]
realizes $\tilde{p}$ as a bundle of pointed spaces over $X_{<r}$. These fibres satisfy
\[
\Omega \tilde{p}{}^{-1}(x)\simeq \tilde{0}{}^{-1}(x)\simeq \tau^{-1}(\tilde{x}) = \{\tilde{x}\}\times_{X_{<n}}X_{<n+r}
\]
as pointed spaces for any $x\in X_{\leq r}$ and any $\tilde{x} \in X_{< n+r}$ lifting $x$. Therefore if we define 
\[
\Pi_{[n,n+r)} X \colonequals \Omega_{X_{< r}}\left((X_{< n}\cup_{X_{< n+r}}X_{< r})_{\leq n+r}\right),
\]
Then $\Pi_{[n,n+r)}X$ is a bundle of loop spaces over $X_{<r}$ satisfying
\[
B_{X_{<r}}\Pi_{[n,n+r)}X \simeq (X_{< n}\cup_{X_{< n+r}}X_{< r})_{\leq n+r}
\]
as needed.
\end{construction}

\begin{warning}
We warn the reader that when $n = r$, in general
\[
\Pi_{[n,2n)}X \neq X_{<2n}
\]
as spaces over $X_{<n}$, despite both spaces having equivalent fibres over any point of $X_{<n}$. The issue is that as a loop space in $\spaces_{/X_{<n}}$, the projection $\Pi_{[n,2n)}X \rightarrow X_{<n}$ admits a section, while $X_{<2n} \rightarrow X_{<n}$ need not.
\end{warning}

This construction of the bundle $\Pi_{[n,n+r)}X$ is somewhat indirect. We give a more direct construction in the special case where $X$ is supersimple in \cref{prop:homotopymoduleherd} below.

\begin{variation}[{Spaces with non-vanishing Whitehead products}]
\label{variation:nonsimplepostnikov}
For an arbitrary space $X$ and $r < n$, the same construction applies to product a cartesian diagram of the form
\begin{equation*}
\label{equation:non_simple_variant_of_postnikov_square}
\begin{tikzcd}
X_{< n+r}\ar[r]\ar[d,"\tau"]&X_{\leq r}\ar[d,"0"]\ar[dr,equals]\\
X_{< n}\ar[r,"k"]&B_{X_{\leq r}}\Pi^{\ns}_{[n,n+r)}X\ar[r,"p"]&X_{\leq r}
\end{tikzcd},
\end{equation*}
where $\Pi^{\mathrm{ns}}_{[n,n+r)}X \rightarrow X_{\leq r}$ is a bundle of loop spaces over $X_{\leq r}$.  In this case, no asumption on the Whitehead product if necessary. If it does happen that $[\pi_{r} X, \pi_{n} X] = 0$, then this is compatible with \cref{constr:simplepostnikov} in the sense that there is a canonical cartesian square 
\[
\begin{tikzcd}
	{\Pi^{\mathrm{ns}}_{[n,n+r)}X} & {\Pi_{[n,n+r)}X} \\
	{X_{\leq r}} & {X_{< r}}
	\arrow[from=1-1, to=1-2]
	\arrow[from=1-1, to=2-1]
	\arrow[from=1-2, to=2-2]
	\arrow[from=2-1, to=2-2]
\end{tikzcd}.
\]
\end{variation}

\begin{definition}
\label{definition:postnikov_squares_of_a_space}
We call the diagrams (\ref{equation:postnikov_squares_of_a_space}) and (\ref{equation:non_simple_variant_of_postnikov_square}) the \emph{Postnikov squares} of $X$ and refer to $k$ as the \emph{Postnikov invariant map}. 
\end{definition}

\begin{ex}\label{ex:classicalpostnikov}
Taking $r = 1$, these Postnikov squares recover the classical theory of Postnikov invariants. Let us first describe the non-simple variant. 

Given a space $X$ and point $x \in X$ and $n \geq 2$ we have an abelian group $\pi_n (X,x)$. This group depends functorially only on the image of $x$ in the fundamental groupoid $X_{\leq 1}$, and by allowing $x$ to vary this defines a bundle $\Pi_{n}^\ns X$ of abelian groups over $X_{\leq 1}$. Now the bundle $\Pi_{[n, n]}^\ns X$ of \cref{variation:nonsimplepostnikov} can be identified with the $n$th fibrewise delooping of $\Pi_n^\ns X$ over $X_{\leq 1}$, and the corresponding Postnikov square with the classical Postnikov square
\begin{center}
\begin{tikzcd}
X_{\leq n}\ar[d]\ar[r]&X_{\leq 1}\ar[d]\ar[dr,equals]\\
X_{<n}\ar[r]&B_{X_{\leq 1}}^{n+1}\Pi_n^\ns X\ar[r]&X_{\leq 1}
\end{tikzcd}.
\end{center}

We move on to the simple variant. A space $X$ satisfies $[\pi_1 X,\pi_n X] = 0$ if and only if the conjugation action of $\pi_1(X,x)$ on $\pi_n(X,x)$ is trivial for all basepoints $x\in X$. This ensures that the bundle $\Pi_n^\ns X$ of groups on $X_{\leq 1}$ descends uniquely to a bundle $\Pi_n X$ on $X_{\leq 0}$, and \cref{constr:simplepostnikov} recovers the classical Postnikov squares
\begin{center}
\begin{tikzcd}
X_{\leq n}\ar[r]\ar[d]&X_{\leq 0}\ar[d]\ar[dr,equals]\\
X_{<n}\ar[r]&B^{n+1}_{X_{\leq 0}}\Pi_{n}X\ar[r]&X_{\leq 0}
\end{tikzcd}
\end{center}
realizing $X_{\leq n}\rightarrow X_{<n}$ as a disjoint union of principal fibrations. Moreover, this improved square may be constructed when $n = 1$, where the condition that $[\pi_1 X,\pi_1 X] = 0$ asks that $\pi_1(X,x)$ is an abelian group for all $x\in X$.
\end{ex}

As \cref{ex:classicalpostnikov} shows, taking $r=1$ the bundles
\[
\Pi_{[n,n]}X \to X_{\leq 0}
\]
are not merely fibrewise loop spaces: if $x \in X_{\leq 0}$, then the fibre
\[
\{x\}\times_{X_{\leq 0}}\Pi_{[n,n]}X \simeq K(\pi_n (X,x),n),
\]
as a pointed space, carries in a unique way the structure of an abelian group object. In particular, it is an \emph{infinite} loop space. This extends to $r > 1$, as we now explain.

\begin{lemma}[Freudenthal suspension]
Let $X$ be a pointed connected space, and suppose $\pi_i X = 0$ for $0 \leq i < n$. Then the map
\[
\pi_i X \rightarrow \pi_i \Omega^r \Sigma^r X \cong \pi_{i+r}\Sigma^r X
\]
is an isomorphism for $i \leq 2n-2$ and a surjection for $i = 2n-1$, with kernel generated by the Whitehead products $[\pi_n X,\pi_n X] \subset \pi_{2n-1} X$.
\end{lemma}
\begin{proof}
This is a special case of the Blakers--Massey theorem \cref{thm:blakersmassey}. 
\end{proof}

This admits the following categorification, which describes certain subcategories of pointed spaces with restricted homotopy groups in terms of stable homotopy theory. Similar statements appear in \cite[\S5.1]{mathewstojanoska2016picard}.

\begin{notation}
Given an interval $I\subset \bbZ$, write $\Sp^I\subset\Sp$ and $\spaces_\ast^I \subset \spaces_\ast$ for the full subcategories of objects with homotopy groups vanishing outside $I$.
\end{notation}

\begin{prop}\label{prop:stablespaces}
The functor
\[
\Omega^\infty \colon \spectra_\ast^{[n,2n)}\rightarrow\spaces_\ast^{[n,2n)}
\]
is fully faithful, with essential image spanned by those $X$ satisfying $[\pi_n X,\pi_n X] = 0$.
\end{prop}
\begin{proof}
Fix $X\in \spaces_\ast^{[n,2n)}$. If $[\pi_n X,\pi_n X] = 0$, then $X\rightarrow\Omega^\infty\Sigma^\infty X$ induces an equivalence $X\simeq X_{< 2n }\simeq (\Omega^\infty\Sigma^\infty X)_{< 2n}$, proving that $X$ is in the essential image of $\Omega^\infty$.

To prove that $\Omega^\infty$ is fully faithful in this range, fix $X,Y\in\Sp^{[n,2n)}$. For
\[
\map_\Sp(X,Y)\rightarrow\map_\ast(\Omega^\infty X,\Omega^\infty Y)\simeq \map_\Sp(\Sigma^\infty\Omega^\infty X,Y)\simeq \map_{\Sp}((\Sigma^\infty \Omega^\infty X)_{< 2n},Y)
\]
to be an equivalence, we must prove that the counit $\epsilon \colon \Sigma^\infty\Omega^\infty X \rightarrow X$ induces an isomorphism on $\pi_i$ for $i < 2n$. To that end, it suffices to prove that $\Omega^\infty\epsilon\colon \Omega^\infty\Sigma^\infty\Omega^\infty X \rightarrow\Omega^\infty X$ is an equivalence in the same range. The triangle identity provides a retraction
\[
\Omega^\infty X \rightarrow\Omega^\infty\Sigma^\infty\Omega^\infty X \rightarrow\Omega^\infty X.
\]
The Freudenthal suspension theorem says that the unit $\eta\colon \Omega^\infty X \rightarrow \Omega^\infty\Sigma^\infty \Omega^\infty X$ induces a surjection on $\pi_i$ for $i < 2n$. The above retraction then implies that it is an isomorphism in this range, and so the same must be true of the counit.
\end{proof}

Let $X \in \spaces_\ast^{[n,2n)}$ be a space satisfying $[\pi_n X,\pi_n X] = 0$. \cref{prop:stablespaces} implies that $X$ carries the structure of an infinite loop space in a unique and natural way. It will be necessary for our purposes to regard this infinite loop space structure as actual \emph{structure}, rather than a mere connectivity property of the space $X$. This allows one to keep track of, for example, the fact that a geometric realization of a simplicial diagram of such pointed spaces is again canonically an infinite loop space, even though it need satisfy the needed connectivity properties. 

It will be convenient to encode the relevant infinite loop space structures diagrammatically, and the formalism of stable deloop theories introduced in \cite[\S7.2]{usd1} gives a convenient framework for this. We recall the particular case we need.

\begin{defn}
\label{def:moduleobjects}
For a connective $\bfE_1$-ring spectrum $R$, let 
\[
\calR_0^\omega(R)\subset\calR_-^\omega(R)\subset\RMod_R
\]
denote the full subcategories of right $R$-modules generated under finite coproducts by $R$ and by $\Omega^n R$ for $n \geq 0$ respectively. Given an $\infty$-category $\dcat$,
\begin{enumerate}
\item A \emph{connective right $R$-module} in $\dcat$ is a functor
\[
X\colon \calR_0^\omega(R)^{\op} \to\dcat
\]
which preserves finite products.
\item A \emph{right $R$-module object} in $\dcat$ is a functor
\[
X\colon \calR_-^\omega(R)^\op \to\dcat
\]
which preserves finite products and for which the canonical map
\[
X(P)\to\Omega X(\Omega P)
\]
is an equivalence for $P \in \calR_-^\omega(R)$.
\end{enumerate}
We write
\[
\RMod_R^{\geq 0}(\dcat) \subset \Fun(\calR_0^\omega(R)^{\op} ,\dcat),\qquad\RMod_R(\dcat) \subset \Fun(\calR_-^\omega(R)^\op,\dcat)
\]
for the full subcategories of these objects.
\end{defn}

When $R$ is an $\bfE_\infty$-ring, as it will be in our examples, we abbreviate $\Mod_R = \RMod_R$. The names are justified by the following.

\begin{lemma}\label{lem:rmoduleobjects}
Let $\dcat$ be a presentable $\infty$-category. Then there are natural equivalence
\[
\RMod_R(\dcat) \simeq \RMod_R\otimes\dcat,\qquad \RMod_R^{{\geq 0}}(\dcat)\simeq \RMod_R^{{\geq 0}}\otimes\dcat
\]
where the right hand side denotes the tensor product of presentable $\infty$-categories.
\end{lemma}

\begin{proof}
We just treat right $R$-module objects, the connective case being identical. It follows from \cite[Proposition 7.2.3]{usd1} that there is an equivalence
\[
\RMod_R(\spaces)\simeq\RMod_R.
\]
The general case follows as the essential image of
\begin{align*}
\RMod_R\otimes \dcat&\simeq\Fun^R(\dcat^\op,\RMod_R) \subset \Fun(\dcat^\op,\Fun(\calR_-^\omega(R)^\op,\spaces))\\
&\simeq\Fun(\calR_-^\omega(R)^\op \times\dcat^\op,\spaces)\simeq\Fun(\calR_-^\omega(R)^\op,\presheaves(\dcat))
\end{align*}
consists of those right $R$-module objects $\calR_-^\omega(R)^\op \to \presheaves(\dcat)$ which land in the essential image of the Yoneda embedding $\dcat\subset\presheaves(\dcat)$, and this is equivalent $\RMod_R(\dcat)$.
\end{proof}

We also note the following.

\begin{lemma}\label{lem:connectivemodules}
Let $\xcat$ be an $\infty$-topos. Then
\[
\RMod_R^{\geq 0}(\dcat) \simeq \RMod_R^{\geq 0} \otimes \xcat \to \RMod_R\otimes \xcat\simeq\RMod_R(\xcat)
\]
is fully faithful, with essential image spanned by those right $R$-module objects $X\colon \calR_-^\omega(R)^\op \to \xcat$ with the property that the canonical map
\[
\mathrm{B} X(P) \to X(\Omega P)
\]
is an equivalence for $P \in \calR_-^\omega(R)$.
\end{lemma}
\begin{proof}
When $\xcat = \spaces$ this is clear; the general case follows by writing $\xcat$ as a left exact localization of a presheaf category and working levelwise.
\end{proof}

We can now give the following.

\begin{prop}\label{prop:pinrsmodule}
Let $X$ be a space. Fix integers $1 \leq r \leq n < \infty$, and suppose that $[\pi_r X,\pi_n X] = 0$. Then the pointed bundle
\[
\Pi_{[n,n+r)}X \in \spaces_{X_{<r}//X_{<r}}
\]
lifts uniquely and naturally to a bundle of connective $\thesphere_{<r}$-module objects over $X$. If $r < n$ but we do not assume that $[\pi_r X,\pi_n X] = 0$, then the same is true of $\Pi_{[n,n+r)}^{\ns}X \in \spaces_{X_{\leq r}//X_{\leq r}}$.
\end{prop}
\begin{proof}
We just treat the case where $[\pi_r X,\pi_n X] = 0$. Under the straightening equivalence, we may regard
\[
\Pi_{[n,n+r)}X \in \spaces_{X_{<r}//X_{<r}}\simeq \Fun(X_{<r},\spaces_\ast)
\]
as the functor sending $x \in X_{<r}$ to its fibre $\Pi_{[n,n+r)}(X,x)\simeq \{\tilde{x}\}\times_{X_{<n}}X_{<n+r}$. By \cref{prop:stablespaces}, there is a unique and natural lift in
\begin{center}\begin{tikzcd}[column sep=large]
&\Sp^{[n,n+r)}\ar[d,"\Omega^\infty"]\\
X_{<r}\ar[r,"\Pi_{[n,n+r)}"']\ar[ur,dashed]&\Spc_\ast^{[n,n+r)}
\end{tikzcd}.\end{center}
The proposition follows as
\begin{align*}
\Fun(X_{<r},\Sp^{[n,n+r)})&\simeq\Fun(X_{<r},\Mod_{\thesphere_{<r}}(\Sp^{[n,n+r)}))\\&\subset\Fun(X_{<r},\Mod_{\thesphere_{<r}}^{{\geq 0} })\simeq\Mod_{\thesphere_{<r}}^{\geq 0}(\spaces_{/X_{<r}})
\end{align*}
by \cref{lem:rmoduleobjects}.
\end{proof}

\begin{notation}\label{notation:kinvariantnotation}
Given a space $X$ and $1 \leq r \leq n$, we write
\[
K_{n,r}^\bullet(X) \in \Mod_{\thesphere_{<r}}(\spaces_{/X_{<r}})
\]
for the $\thesphere_{<r}$-module object
\[
\Omega_{X_{<r}}^n\Pi_{[n,n+r)}X \in \Mod_{\thesphere_{<r}}^{\geq 0}(\spaces_{/X_{<r}})\subset\Mod_{\thesphere_{<r}}(\spaces_{/X_{<r}}),
\]
with $\Omega_{X_{<r}}^n\Pi_{[n,n+r)}X$ constructed in \cref{prop:pinrsmodule} and second inclusion as in \cref{lem:connectivemodules}. In particular, $K_{n,r}^\bullet(X)$ has underlying infinite loop spaces
\[
K_{n,r}^m(X) := K_{n,r}^\bullet(X)(\Sigma^m\thesphere_{<r})\simeq \begin{cases}
	\Omega^{n-m}_X \Pi_{[n,n+r)}X & m \leq n, \\
	B^{m-n}_X\Pi_{[n,n+r)}X & m \geq n.
    \end{cases}
\]
for $m\geq 0$, indexed so that $K_{n,r}^n(X) = \Pi_{[n,n+r)}(X)$.
\end{notation}

\begin{rmk}\label{rmk:naturalcartsquare}
All of the constructions of this section are manifestly natural in $X$. In particular, let $\spaces_0\subset\spaces$ denote a full subcategory of spaces $X$ satisfying $[\pi_r X,\pi_n X] = 0$. Then the functors $K_{n,r}^\bullet$ define a $\thesphere_{<r}$-module object
\[
K_{n,r}^\bullet \in \Mod_{\thesphere_{<r}}(\Fun(\spaces_0,\spaces)_{/\tau_{<r}})
\]
in the $\infty$-category of functors over $\tau_{<r}$, and the Postnikov squares of \cref{definition:postnikov_squares_of_a_space} define cartesian squares
\begin{center}\begin{tikzcd}
\tau_{<n+r}\ar[r]\ar[d]&\tau_{<r}\ar[d]\ar[dr,equals]\\
\tau_{<n}\ar[r]&K^{n+1}_{n,r}\ar[r]&\tau_{<r}
\end{tikzcd}\end{center}
of functors. 
\end{rmk}

\subsection{Postnikov squares of \texorpdfstring{$\infty$}{infty}-categories}\label{ssec:categoricalpostnikovsquares}

The Postnikov squares constructed in \S\ref{ssec:postnikovsquares} consist of product-preserving endofunctors of spaces and modifying along them allows us to construct Postnikov squares of categories. The goal of this section is to briefly elaborate on this construction.

\begin{remark}
Since we will be mainly interested in modifying Malcev theories, whose mapping spaces have vanishing Whitehead products, we will restrict to this class of $\infty$-categories. We leave the variations for general categories, which can be obtained using \cref{variation:nonsimplepostnikov} in place of \cref{constr:simplepostnikov}, to the interested reader (but see also \cref{ex:ssysteminftym}).
\end{remark}

\begin{notation}
\label{constr:kinvariantcategory}
Given an $\infty$-category $\ccat$ enriched in spaces with vanishing Whitehead products, for $1 \leq r \leq n$ we write
\[
\kinv_{n,r}^m\ccat = \ccat_{K_{n,r}^m},
\]
for the modification along $K_{n,r}^m$ of \cref{notation:kinvariantnotation}.
\end{notation}

\begin{rmk}
As a consequence of \cref{prop:modifierlimitses}, the composite
\[
\begin{tikzcd}
	{ \calR_-^\omega(\thesphere_{<r})^\op} & {\Fun(\spaces_0,\spaces)_{/\tau_{<r}}} & {(\catinfty)_{/\h_r\ccat}}
	\arrow["{K_{n_r}^{\bullet}}", from=1-1, to=1-2]
	\arrow["{\ccat_{-}}", from=1-2, to=1-3]
\end{tikzcd}
\]
is a $\thesphere_{<r}$-module object. In particular, 
\[
\kinv_{n,r}^m\ccat \simeq (\kinv_{n,r}^\bullet\ccat)(\Sigma^m \thesphere_{<r})
\]
of \cref{constr:kinvariantcategory} assemble into a spectrum object in $\infty$-categories over $\h_{r}\ccat$. 
\end{rmk}

\begin{construction}\label{constr:categoricalpostnikovsquare}
Let $\ccat$ be an $\infty$-category enriched in spaces with vanishing Whitehead products, and let $1 \leq r \leq n < \infty$ be integers. The \emph{categorical Postnikov square}
\begin{center}\begin{tikzcd}
\h_{n+r}\ccat\ar[r,"\tau_{(n+r,r)}"]\ar[d,"\tau_{(n+r,n)}"']&\h_r\ccat\ar[dr,equals]\ar[d,"0"]\\
\h_n\ccat\ar[r,"k"]&\kinv_{n,r}^{n+1}\ccat\ar[r]&\h_r\ccat
\end{tikzcd}\end{center}
is obtained by modifying $\ccat$ along the Postnikov square of \cref{constr:simplepostnikov}. This is guaranteed to be a cartesian square by \cref{prop:modifierlimitses}.
\end{construction}

\begin{ex}\label{ex:categoricalpostnikovspace}
If $X$ is an $\infty$-groupoid, then when considered as an $\infty$-category, $X$ is enriched in spaces with vanishing Whitehead products, and in fact in supersimple spaces: for any two points $x_1,x_2\in X$, we may identify $\map_X(x_1,x_2)\simeq \{x_1\}\times_X\{x_2\}$, which is either empty or equivalent to a loop space of $X$. The categorical Postnikov square of \cref{constr:categoricalpostnikovsquare} applied to $X$ is its ordinary Postnikov square in one dimension higher:
\begin{center}
\begin{tikzcd}
\h_{n+r} X\ar[r]\ar[d]&\h_r X\ar[d]\ar[dr,equals]\\
\h_n X \ar[r]&\kinv_{n,r}^{n+1} X\ar[r]&\h_r X
\end{tikzcd}
$\quad = \quad$
\begin{tikzcd}
X_{\leq n+r}\ar[r]\ar[d]&X_{\leq r}\ar[d]\ar[dr,equals]\\
X_{\leq n}\ar[r]&B_{X_{\leq r}}\Pi_{[n+1,n+1+r)}X^\ns\ar[r]&X_{\leq r}
\end{tikzcd}.\end{center}
In other words, \cref{variation:nonsimplepostnikov} may be recovered from \cref{constr:simplepostnikov} via \cref{constr:categoricalpostnikovsquare}.
\end{ex}

\begin{warning}
The first Postnikov square of a space $X$ with abelian fundamental groupoid takes the form
\begin{equation}\label{eq:firstpostnikov}\begin{tikzcd}
X_{\leq 1}\ar[r]\ar[d]&X_{\leq 0}\ar[d,"0"]\ar[dr,equals]\\
X_{\leq 0}\ar[r,"k"]&B^2_{X_{\leq 0}}\Pi_1 X\ar[r,"p"]&X_{\leq 0}
\end{tikzcd}.\end{equation}
The target of the first $k$-invariant may be written as
\[
B^2_{X_{\leq 0}}\Pi_1 X \simeq \coprod_{x\in \pi_0 X} K(\pi_1 (X,x),2).
\]
As each $K(\pi_1(X,x),2)$ is simply connected, the space of sections of $p$ is connected. It follows that $k\simeq 0$ over $X_{\leq 0}$ and thus 
\[
X_{\leq 1}\simeq B_{X_{\leq 0}}\Pi_1 X  \simeq \kinv^1_{1,1} X\simeq \coprod_{x\in \pi_0 X}K(\pi_1(X,x),1).
\]
However, the homotopy $k\simeq 0$ is \emph{not} natural in $X$, and the two functors $k^1_{1,1}$ and $\tau_{\leq 1}$ are \emph{not} equivalent. To see this formally, observe that by \cref{ex:categoricalpostnikovspace}, the next Postnikov square
\begin{center}\begin{tikzcd}
X_{\leq 2}\ar[r]\ar[d]&X_{\leq 1}\ar[d,"0'"]\ar[dr,equals]\\
X_{\leq 1}\ar[r,"k'"]&B_{X_{\leq 1}}^3\Pi_2^\ns X\ar[r]&X_{\leq 1}
\end{tikzcd}\end{center}
is obtained by modifying $X$ along the squares of (\ref{eq:firstpostnikov}). Thus if there were a natural equivalence $k\simeq 0$, then there would be a natural equivalence $k'\simeq 0'$. Continuing in this way, a natural equivalence $k\simeq 0$ would imply that all spaces have trivial $k$-invariants, an absurdity.
\end{warning}

\subsection{Postnikov modifiers are elementary}\label{ssec:elementarypostnikov}

Our goal in this subsection is to show that all constructions appearing in the Postnikov tower of a supersimple space may be described by means of elementary modifiers in the sense of \cref{def:elementarymodifiers}.

We begin by revisiting the construction of $\Pi_{[n,n+r)}$, whose original description given in \cref{constr:simplepostnikov} was quite indirect. When $X$ is supersimple, it admits the following more convenient characterization. 

\begin{prop}
\label{prop:homotopymoduleherd}
Let $X$ be a supersimple space.
\begin{enumerate}
\item The comparison map
\begin{align*}
\tau_{<r}(X^{S^n\vee \cdots\vee S^n}) &\simeq \tau_{<r}(X^{S^n}\times_X\cdots\times_X X^{S^n}) \\
&\to \tau_{<r}(X^{S^n})\times_{X_{<r}}\cdots\times_{X_{<r}}\tau_{<r}(X^{S^n})
\end{align*}
is an equivalence for all $n\geq 1$ and $r\geq 1$, and therefore the cogroup structure on $S^n$ makes $\tau_{<r}(X^{S^n}) \to X_{<r}$ into a bundle of grouplike $\bfE_n$ spaces.
\item If $r \leq n$, then this refines uniquely to a bundle of connective $\thesphere_{<r}$-modules, and there is a natural equivalence 
\[
B^n_{X_{<r}}\tau_{<r}(X^{S^n})\simeq \Pi_{[n,n+r)}X
\]
of bundles of connective $\thesphere_{<r}$-modules over $X_{<r}$.
\end{enumerate}

In other words, $K_{n,r}^0(X) = \tau_{<r}(X^{S^n})$ when $X$ is supersimple.

\end{prop}
\begin{proof}
(1)~~This follows from \cite[Corollary 2.2.15.(2)]{usd1}.

(2)~~It suffices to just produce a natural equivalence $B^n_{X_{<r}}\tau_{<r}(X^{S^n})\simeq \Pi_{[n,n+r)}X$ of pointed bundles over $X_{<r}$, as the existence of and compatibility with $\thesphere_{<r}$-module structures then follows by uniqueness. For any space $X$, there is a cartesian square
\begin{center}\begin{tikzcd}
B^n_X X^{S^n}\ar[r]\ar[d]&X\ar[d]\\
X\ar[r]&X_{<n}
\end{tikzcd}\end{center}
realizing $B^n_X X^{S^n}$ as a fibrewise $n$-connective cover of $X$. Applied to $X_{<n+r}$, this produces for us the left square in a cartesian diagram
\begin{center}\begin{tikzcd}
B^n_{X_{<n+r}}( (X_{<n+r})^{S^n})\ar[r]\ar[d]&X_{<n+r}\ar[r]\ar[d]&X_{<r}\ar[d,"0"]\\
X_{<n+r}\ar[r]\ar[ur,dashed]&X_{<n}\ar[r,"k"]&B_{X_{<r}}\Pi_{[n,n+r)}X
\end{tikzcd}.\end{center}
The existence of the dashed arrow shows that the bottom composite factors through the zero-section of $B_{X_{<r}}\Pi_{[n,n+r)}X$, and this provides a cartesian square
\begin{center}\begin{tikzcd}
B^n_{X_{<n+r}}((X_{<n+r})^{S^n})\ar[r]\ar[d]&\Pi_{[n,n+r)}X\ar[d]\\
X_{<n+r}\ar[r]&X_{<r}
\end{tikzcd}\end{center}
compatible with zero-sections. On fibrewise $n$-fold loop spaces, this produces the outer rectangle in a diagram
\begin{center}\begin{tikzcd}
(X_{<n+r})^{S^n}\ar[r]\ar[d]&\tau_{<r}\left((X_{<n+r})^{S^n}\right)\ar[r]\ar[d]&\Omega_{X_{<r}}^n\Pi_{[n,n+r)}X\ar[d]\\
X_{<n+r}\ar[r]&X_{<r}\ar[r]&X_{<r}
\end{tikzcd},\end{center}
which uniquely factors through the middle column. It therefore suffices to prove that the left square (and so also the right square) is cartesian and that the canonical map
\[
\tau_{<r}(X^{S^n}) \to \tau_{<r}\left((X_{<n+r})^{S^n}\right)
\]
is an equivalence. By restricting to various path components of $X$, we may suppose that $X$ is a grouplike $H$-space. This $H$-space structure provides compatible splittings
\[
X^{S^n}\simeq \Omega^n X \times X,\qquad (X_{<n+r})^{S^n}\simeq X_{<n+r}\times \Omega^n(X_{<n+r}),
\]
so it suffices to verify that the canonical maps
\[
\tau_{<r}(X) \to \tau_{<r}(X_{<n+r}),\qquad \tau_{<r}(\Omega^n X) \to \tau_{<r}(\Omega^n(X_{<n+r})) \leftarrow \Omega^n(X_{<n+r})
\]
are equivalences, which is clear.
\end{proof}

We can now give the following.

\begin{prop}\label{prop:elementarypostnikov}
The following modifiers are all elementary:
\begin{enumerate}
\item $X \mapsto B^m_X (X^T)$, where $T = S^{m_1}\vee\cdots \vee S^{m_k}$ is a wedge of spheres of dimension $m_i \geq m$.
\item $X \mapsto \tau_{<r}(X)$ for any $r\geq 1$.
\item $X \mapsto K^m_{n,r}(X)\times_{X_{<r}}\cdots\times_{X_{<r}}K^m_{n,r}(X)$ for $1\leq r \leq n < \infty$ and $m \geq 0$.
\end{enumerate}
\end{prop}
\begin{proof}
(1)~~We induct on $n \leq m$, the case $n = 0$ holding by the definition of an elementary modifier. In the inductive step, observe that
\[
B^{n}_X(X^T)\times_X B^{n}(X^T)\simeq B^{n}(X^T\times_X X^T)\simeq B^{n}(X^{T\vee T}).
\]
Continuing in this way, we can identify the bar construction for $B^{n+1}_X(X^T)$ as
\[
B^{n+1}_X(X^T)\simeq\left| (B^n_X(X^T))^{\times_X \bullet}\right| \simeq \left|B^n_X(X^{T\vee\cdots\vee T})\right|.
\]
Each term in this simplicial object is elementary by induction, and thus the geometric realization is elementary by definition.

(2)~~Observe that
\[
X\times_{X_{<r}}X\simeq B_X^r(X^{S^r}),\qquad X\times_{X_{<r}}X\times_{X_{<r}}X\simeq X\times_{X_{<r}}X\times_X X_{<r}X\simeq B_X^r(X^{S^r\vee S^r}),
\]
and so on. Continuing in this way, the \v{C}ech nerve of the effective epimorphism $X \to X_{<r}$ gives an equivalence
\[
X_{<r}\simeq \colim\left(\begin{tikzcd}X &\ar[l,shift left=1mm]\ar[l,shift right=1mm]B_X^r(X^{S^r})&\ar[l,shift left=1mm]\ar[l,shift right=1mm]\ar[l]B_X^r(X^{S^r\vee S^r})&\ar[l,shift left=1.5mm]\ar[l,shift left=0.5mm]\ar[l,shift right=0.5mm]\ar[l,shift right=1.5mm]\cdots\end{tikzcd}\right),
\]
where each term in this \v{C}ech nerve is an elementary modifier in $X$ by (1). It follows that $X \mapsto X_{<r}$ is elementary.

(3)~~We induct on $m$. Combining \cref{prop:properties_of_elementary_modifiers} with (2), we see that if $T$ is a finite product of finite wedges of spheres then
\[
X \mapsto (X^T)_{<r}
\]
is elementary. \cref{prop:homotopymoduleherd} implies that
\[
K^0_{n,r}(X)\times_{X_{<r}}\cdots\times_{X_{<r}}K^0_{n,r}(X)\simeq (X^{S_n\vee\cdots\vee S^n})_{<r},
\]
so this handles the case $m=0$. The inductive step follows by writing
\[
K^m_{n,r}(X)\times_{X_{<r}}\cdots\times_{X_{<r}}K^m_{n,r}(X)\simeq B^m_{X_{<r}}( (X^{S^n\vee \cdots \vee S^n})_{<r})
\]
as a geometric realization of terms which are all elementary by induction.
\end{proof}

\subsection{Spiral systems}\label{ssec:spiralsystems}

As we have seen, the Postnikov tower of a space carries a significant amount of extra structure. For supersimple spaces, this structure can be conveniently encoded in the following.

\begin{defn}
\label{definition:spiral_system}
Let $\dcat$ be an $\infty$-category. A \emph{spiral system} in $\dcat$ is a functor
\[
H\colon \elementarymodifiers \to \dcat
\]
which preserves levelwise pullbacks along levelwise effective epimorphisms. We say that a spiral system $H$ is \emph{convergent} if the comparison map
\[
H(\id) \to \lim_{n\to\infty}H(\tau_{<n})
\]
is an equivalence. We write
\[
\Spiral(\dcat)\subset\Fun(\elementarymodifiers,\dcat)
\]
for the full subcategory of spiral systems.
\end{defn}

\begin{remark}
Levelwise pullback diagrams of elementary modifiers are exactly those pullback diagrams which are preserved by the inclusion $\elementarymodifiers \subseteq \End_{\sigma}(\ukanspaces)$. We do not know if all pullback diagrams in $\elementarymodifiers$ have this property. 
\end{remark}

A spiral system $H$ is a refinement of the tower
\[
H(\id) \to \cdots \to H(\tau_{\leq n}) \to H(\tau_{<n}) \to \cdots H(\tau_{\leq 0}),
\]
endowing this tower with additional structure modeled on the Postnikov tower of a supersimple space. Examples of spiral systems abound in nature, as we now explain. 

\begin{ex}
\label{ex:ssystemtautological}
Let $X$ be a supersimple space. Then the assignment
\[
\elementarymodifiers\to\spaces,\qquad F \mapsto F(X)
\]
is tautologically a spiral system of spaces.
\end{ex}

\begin{ex}
\label{ex:ssystemssimplecat}
Let $\pcat$ be an $\infty$-category with supersimple mapping spaces (such as a Malcev theory or an $\infty$-groupoid, see \cref{ex:categoricalpostnikovspace}).
Then \cref{prop:properties_of_elementary_modifiers} implies that
\[
\elementarymodifiers\to\catinfty,\qquad F \mapsto \pcat_F
\]
is a convergent spiral system of $\infty$-categories.
\end{ex}

\begin{rmk}
\label{ex:spiralclosure}
If $\dcat$ is any $\infty$-category with limits then
\[
\Spiral(\dcat)\subset\Fun(\elementarymodifiers,\dcat)
\]
is closed under limits. Moreover, postcomposition with any pullback-preserving  $f\colon \ccat\to\dcat$ induces a functor
\[
\Spiral(\ccat)\to\Spiral(\dcat).
\]
\end{rmk}

The spiral system formalism is also sufficiently robust as to encode Postnikov structure in non-supersimple contexts, such as the $\infty$-category $\spaces$ of arbitrary spaces. 

\begin{ex}\label{ex:ssysteminftym}
For $m \geq 0$, let $\Cat_{(\infty,m)}$ denote the $\infty$-category of $(\infty,m)$-categories and let $\End^\times(\Cat_{(\infty,m)})$ denote the $\infty$-category of product-preserving endofunctors of $\Cat_{(\infty,m)}$. We claim that for all $m \geq 0$, there is a spiral system extending the assignment
\begin{equation}\label{eq:ssysteminftym}
\elementarymodifiers\to\End^\times(\Cat_{(\infty,m)}),\qquad \tau_{<n} \mapsto (\ctwocat \mapsto \h_{(m+n,m)}\ctwocat).
\end{equation}
In particular, the linear extensions of \cref{prop:spiralsquares} below extend the Postnikov invariants of $(\infty,m)$-categories constructed by Harpaz--Nuiten--Prasma in \cite{harpaznuitenprasma2020kinvariants}.

This spiral system is constructed by induction on $m$ as follows. If $m = 0$, then $\Cat_{(\infty,0)}\simeq\spaces$ is equivalent to the $\infty$-category of spaces. By treating spaces as $\infty$-groupoids, \cref{ex:ssystemtautological} and \cref{ex:categoricalpostnikovspace} combine to construct for any space $X$ a spiral system extending the assignment
\[
\tau_{<n} \mapsto X_{\leq n}\simeq \h_n X.
\]
This construction is natural in $X$, and determines a spiral system
\[
\elementarymodifiers\to\End^\times(\spaces) \simeq \End^\times(\Cat_{(\infty,0)}),
\]

Now, suppose inductively that $m > 0$ and that we have constructed the spiral system of (\ref{eq:ssysteminftym}). As
\[
\Cat_{(\infty,m+1)}\simeq\Cat(\Cat_{(\infty,m)}),
\]
the change of enrichment process of \cref{rec:changeofenrichment} determines a functor
\[
\End^\times(\Cat_{(\infty,m)}) \to \End(\Cat_{(\infty,m+1)}).
\]
The composite
\[
\elementarymodifiers\to \End^\times(\Cat_{(\infty,m)}) \to \End(\Cat_{(\infty,m+1)})
\]
is seen to define a spiral system, land in the full subcategory of product-preserving functors, and extend the assignment $\tau_{<n} \mapsto (\ctwocat\mapsto \h_{(m+n+1,m+1)}\ctwocat)$.
\end{ex}

\begin{ex}\label{ex:mappingspiral}
Let $H$ be a spiral system of $(\infty,n+1)$-categories. For $X,Y\in H(\id)$, the assignment
\[
\map_H(X,Y)(F) = \map_{H(F)}(\pi_{F!}X,\pi_{F!}Y)
\]
defines a spiral system of $(\infty,n)$-categories, convergent if $H$ is.
\end{ex}

\begin{ex}\label{ex:ssystemsheaf}
Let $\xcat$ be an $\infty$-topos. For $X \in \xcat$, there is a spiral system extending the assignment
\[
\tau_{<n} \mapsto X_{\leq n}. 
\]
Moreover, this system is convergent if and only if $X$ is hypercomplete. This spiral system may be constructed by writing $\xcat$ as a left exact localization of a presheaf $\infty$-topos and sheafifying the construction of \cref{ex:ssysteminftym} for $m=0$. If $X$ is a sheaf of supersimple spaces, then one may instead sheafify $F\mapsto F(X)$ directly.
\end{ex}

Perhaps the most important piece of structure that a spiral system encodes is the following.

\begin{prop}\label{prop:spiralsquares}
Let $H$ be a spiral system in an $\infty$-category $\dcat$. For all $1 \leq r \leq n < \infty$, the map $H(\tau_{<n+r})\to H(\tau_{<n})$ is a square-zero extension in the following sense: the collection
\[
\{H(K^m_{n,r}) : m \geq 0 \}
\]
assembles into an object 
\[
H(K^\bullet_{n,r}) \in \Mod_{\thesphere_{<r}}(\dcat_{/H(\tau_{<r})}),
\]
and there are natural cartesian squares
\begin{center}\begin{tikzcd}
H(\tau_{<n+r})\ar[r,"\tau_{(n+r,r)!}"]\ar[d,"\tau_{(n+r,n)!}"]&H(\tau_{<r})\ar[d,"0_!"]\ar[dr,equals]\\
H(\tau_{<n})\ar[r,"k_!"]&H(K^{n+1}_{n,r})\ar[r,"q_!"]&H(\tau_{<r})
\end{tikzcd}.\end{center}
\end{prop}
\begin{proof}
By \cref{notation:kinvariantnotation}, we have
\[
K^\bullet_{n,r} \in \Mod_{\thesphere_{<r}}(\elementarymodifiers_{/\tau_{<r}}).
\]
By definition, this is encoded by a functor
\[
K^\bullet_{n,r}\colon \calR_-^\omega(\thesphere_{<r})^\op \to \elementarymodifiers_{/\tau_{<r}}
\]
which preserves the terminal object and pullbacks along certain maps sent to levelwise effective epimorphisms in $\elementarymodifiers_{/\tau_{<r}}$. The assumption that $H$ is a spiral system therefore ensures that the composite
\[
H\circ K^\bullet_{n,r}\colon \calR_-^\omega(\thesphere_{<r})^\op \to \elementarymodifiers_{/\tau_{<r}} \to \dcat_{/H(\tau_{<r})}
\]
again defines a $\thesphere_{<r}$-module object, and that $H$ sends the cartesian squares of \cref{rmk:naturalcartsquare} to cartesian squares in $\dcat$.
\end{proof}

\begin{rmk}
In the setting of \cref{prop:spiralsquares}, when $\dcat = \catinfty$ one says that $H(\tau_{<n+r})$ is a \emph{linear extension of $H(\tau_{<n})$ by $H(K^\bullet_{n,r})$}. We will study linear extensions in greater detail and generality further on in \cref{ssec:linearextensions}.
\end{rmk}

\begin{rmk}\label{rmk:spiralstructure}
Let $H$ be a spiral system in an $\infty$-category $\dcat$. As $1 \leq r \leq n$ vary, the objects
\[
H(K^\bullet_{n,r}) \in \Mod_{\thesphere_{<r}}(\dcat_{/H(\tau_{<r})})
\]
satisfy various compatibilities: for example, there are fibre sequences
\[
H(\tau_{<r+t})\times_{H(\tau_{<t})}H(K^{\bullet+r}_{n+r,n+r+t})\to H(K^\bullet_{n,n+r+t})\to H(\tau_{<r+t})\times_{H(\tau_{<r})}H(K^\bullet_{n,n+r})
\]
in $\Mod_{\thesphere_{<r}}(\dcat_{/H(\tau_{<r+t})})$, when these terms are defined. Spiral systems give a convenient way of packaging all of this additional structure into one coherent object.
\end{rmk}

\section{Infinitesimal extensions of Malcev theories}

Our goal in this section is to study the extent to which the assignment
\[
\malcevtheories\to\catinfty,\qquad \pcat \mapsto \Model_\pcat
\]
to a Malcev theory $\pcat$ its $\infty$-category of models preserves pullbacks. In particular, we establish sufficient conditions for a cartesian square of Malcev theories to be preserved by passage to categories of models. These sufficient conditions are easily verified for the Postnikov squares of $\pcat$, allowing us to prove that the spiral squares of $\pcat$ discussed in the introduction are cartesian.

\subsection{A weak pullback theorem}

We start by establishing a much easier weak pullback theorem, 
\cref{prop:weakpb}, which gives conditions under which the comparison map into the pullback of $\infty$-categories of models is fully faithful. This result is already sufficient, for example, to construction useful decompositions of mapping spaces for models of a Malcev theory, as we describe in \S\ref{ssec:derivedpostnikov}. Later, in \S\ref{subsection:clutching_for_malcev_theories}, we strengthen the theorem given here to an actual equivalence of $\infty$-categories under further assumptions.

\begin{defn}
Say that a functor $f\colon \ccat\to\dcat$ is \emph{locally $n$-connective} if, for all $a,b\in \ccat$, the induced map $\map_\ccat(a,b)\to\map_\dcat(a,b)$ is $n$-connective.
\end{defn}

\begin{ex}
A functor $f\colon \ccat\to\dcat$ is locally $0$-connective if and only if it is full.
\end{ex}

\begin{ex}\label{ex:additivelocalconnective}
Given a map $\phi\colon R \to S$ of connective ring spectra, the induced homomorphism
\[
\cfrees(R) \to \cfrees(S)
\]
of theories (see \cite[\S9.1]{usd1}) is locally $n$-connective if and only if $\phi$ is $n$-connective.
\end{ex}

\begin{lemma}\label{lem:localconnectedunit}
Let $f\colon \pcat\to\qcat$ be a locally $n$-connective map of Malcev theories. 
\begin{enumerate}
\item If $X \in \Model_\pcat$, then the unit $\eta_X\colon X \to f^\ast f_! X$ is $n$-connective.
\item If $f$ is essentially surjective and $Y\in \Model_\qcat$, then the counit $\epsilon_Y\colon f_! f^\ast Y \to Y$ is $(n+1)$-connective.
\end{enumerate}
\end{lemma}
\begin{proof}
(1)~~As $n$-connective maps of spaces are stable under colimits and geometric realizations in $\Model_\pcat$ are computed levelwise, we may resolve $X$ by representables to reduce to the case where $X = \nu P$ for some $P \in \pcat$. The unit $\nu P \to f^\ast f_! \nu P$, when evaluated on $P' \in \pcat$, is equivalent to $\map_\pcat(P',P) \to \map_\qcat(fP',fP)$, which is $n$-connective by assumption.

(2)~~Consider the triangle identity:
\begin{center}\begin{tikzcd}
f^\ast Y\ar[d,"\eta_{f^\ast Y}"']\ar[dr,equals]\\
f^\ast f_! f^\ast Y\ar[r,"f^\ast\epsilon_Y"']&f^\ast Y
\end{tikzcd}.\end{center}
By (1), the unit $\eta_{f^\ast Y}$ is $n$-connective. It follows from \cite[Proposition 6.5.120]{lurie2017highertopos} that $f^\ast\epsilon_Y$ is $(n+1)$-connective. As $f$ is essentially surjective, we deduce that $\epsilon_Y$ is $(n+1)$-connective.
\end{proof}

\begin{theorem}\label{prop:weakpb}
Let
\begin{center}\begin{tikzcd}
\pcat'\ar[r,"g'"]\ar[d,"f'"]&\pcat\ar[d,"f"]\\
\qcat'\ar[r,"g"]&\qcat
\end{tikzcd}
$\qquad$
\begin{tikzcd}
\Model_{\pcat'}\ar[r,"g'_!"]\ar[d,"f'_!"]&\Model_\pcat\ar[d,"f_!"]\\
\Model_{\qcat'}\ar[r,"g_!"]&\Model_\qcat
\end{tikzcd}
\end{center}
be a commuting square of Malcev theories and coproduct-preserving functors, with corresponding commuting square of $\infty$-categories of models. Suppose that $\pcat'\to\qcat'\times_\qcat\pcat$ is fully faithful and that $f$ is full. Then
\begin{enumerate}
\item The comparison functor $L\colon \Model_{\pcat'}\to\Model_{\qcat'}\times_{\Model_\qcat}\Model_\pcat$ admits a right adjoint $R$, sending an object of the pullback $\infty$-category written as a triple $\langle Y,X,\alpha\rangle$ with $Y \in \Model_{\qcat'}$, $X\in \Model_\pcat$, and $\alpha\colon g_! Y \simeq f_! X$, to 
\[
R(\langle Y,X,\alpha\rangle) = f'{}^\ast Y \times_{d^\ast f_! X}g'{}^\ast X,
\]
where $d \colonequals f\circ g' = g\circ f'$.
\item For any $X \in \Model_{\pcat'}$, the unit $X \rightarrow RLX$ is an equivalence. Equivalently, $L$ is fully faithful. 
\end{enumerate}
\end{theorem}
\begin{proof}
Claim (1) is formal. For (2), observe that since $\pcat' \to \qcat'\times_{\qcat}\pcat$ is fully faithful, the unit is an equivalence when $X = \nu P'$ for some $P' \in \pcat'$. To show that it is always an equivalence for any $X$, it therefore suffices to verify that 
\[
RL X = f'{}^\ast f'_! X \times_{d^\ast d_! X}g'{}^\ast g'_! X
\]
preserves geometric realizations in $X$. Each of the functors $f'{}^\ast f'_!$, $d^\ast d_!$, and $g'{}^\ast g'_!$ preserve geometric realizations, and the assumption that $f$ is full guarantees by \cref{lem:localconnectedunit} that $g'{}^\ast g'_! X \to d^\ast d_! X$ is an effective epimorphism, so \cite[Corollary 4.1.8.(2)]{usd1} implies that the relevant pullback also preserves with geometric realization. 
\end{proof}

\subsection{Derived Postnikov squares}
\label{ssec:derivedpostnikov}

Fix a Malcev theory $\pcat$. As an application of \cref{prop:weakpb}, in this section we explain how one may associate to any model of $\pcat$ a derived version of its Postnikov tower, previously considered in \cite[\S5.4]{balderrama2021deformations}. In fact, we will do more: we will produce not just a derived Postnikov tower, but the more elaborate data of a spiral system, guaranteeing that derived Postnikov towers have all of the additional structure and naturality discussed in \cref{ssec:spiralsystems}.

\begin{notation}\label{notation:derivedcotensor}
Given a space $T$ and $1\leq r \leq \infty$, we write
\[
(\bs)_T\colon \Model_{\h_r\pcat}\to\Model_{\h_r\pcat}
\]
for the unique geometric realization-preserving construction extending the assignment 
\[
\nu_r(P) \mapsto \tau_{<r}(\nu(P)^{T})
\]
on representables. 
\end{notation}

\begin{rmk}\label{rmk:modifyalongspace}
Given $X \in \Model_{\h_r\pcat}$, we may identify
\[
X_T \simeq \pi_{\alpha}^\ast\pi_{\alpha!}X
\]
where
\[
\pi_\alpha\colon \h_r\pcat \simeq \pcat_{\tau_{<r}} \to \pcat_{\tau_{<r}\map(T,\bs)}\simeq\h_r\pcat_{\map(T,\bs)}
\]
is induced by the natural transformation $\alpha\colon\tau_{<r}\to\tau_{<r}\map(T,\bs)$ restricting along the unique map $T \to \ast$. If $T \in \Sph$, then this is a natural transformation between elementary modifiers by \cref{prop:properties_of_elementary_modifiers}.(1) and \cref{prop:elementarypostnikov}.(2).
\end{rmk}

\begin{warning}
Beware that the functor $(\bs)_T\colon \Model_{\h_r\pcat}\to\Model_{\h_r\pcat}$ depends on the $\infty$-category $\pcat$ and not just its homotopy $r$-category $\h_r\pcat$. This extends the observation that for a space $X$, the homotopy type of $\tau_{<r}(X^{T})$ depends on more than just the truncation $\tau_{<r} X$.
\end{warning}

\begin{ex}
If $\pcat$ admits constant colimits indexed by $T$, then $\Lambda_T$ is characterized by
\[
\map_{\h_r\pcat}(\nu_r P,\Lambda_T)\simeq\map_{\h_r\pcat}(\nu_r(T\otimes P),\Lambda);
\]
that is, $(\bs)_T$ is given by precomposition with the endofunctor of $\h_r\pcat$ induced by $T\otimes (\bs)\colon \pcat\to\pcat$. In particular, in this case $(-)_{T}$ is continuous. 
\end{ex}

\begin{ex}
\label{example:derived_cotensor_with_sn_for_the_category_of_modules}
Let $R$ be a connective $\mathbf{E}_1$-ring and $\lfrees_{0}(R)$ be the Malcev theory of free connective $R$-modules, so that $\Model_{\h\pcat}\simeq\LMod_{\pi_0 R}^{\geq 0}$ (see \cite[\S9.1]{usd1}). If $M \in \LMod_{\pi_0 R}^{\geq 0}$, then
\[
M_{S^n}\simeq (\pi_0 R \oplus \pi_n R) \otimes_{\pi_0 R} M \simeq \pi_n R \otimes_{\pi_0 R} M \oplus M.
\]
The additional copy of $M$ arises as the cotensors in \cref{notation:derivedcotensor} are unpointed.
\end{ex}

\begin{lemma}\label{lem:pseudocotensor}
Fix $1 \leq r,n \leq \infty$ and $\Lambda \in \Model_{\h_r\pcat}$.
\begin{enumerate}
\item The comparison $\Lambda_{S^n\vee S^n} \to \Lambda_{S^n}\times_\Lambda\Lambda_{S^n}$ is an equivalence, and therefore the coproduct on $S^n$ makes $\Lambda_{S^n}$ into a grouplike $\mathbf{E}_n$-algebra in $(\Model_{\h_{r}\pcat})_{/\Lambda}$.
\item If $r \leq n$, then there is a equivalence
\[
\Lambda_{S^n} \simeq 0^\ast 0_! \Lambda
\]
of $\mathbf{E}_{n}$-algebras, where $0\colon \h_r\pcat \to \kinv^{0}_{n,r}\pcat$ is as in \cref{constr:kinvariantcategory}. In particular, there is a unique natural refinement of the $\bfE_{n}$-algebra structure on $\Lambda_{S^n}$ over $\Lambda$ to the structure of a connective $\thesphere_{<r}$-module over $\Lambda$.
\item If $\pcat$ is a loop theory, then $\Lambda \mapsto \Lambda_T$ is compatible with derived truncation in the sense that
\[
\tau_{(r,i)!}(\Lambda_{T}) \simeq (\tau_{(r,i)!}\Lambda)_{T}
\]
for any $i\leq r$ and any space $T$.
\end{enumerate}
\end{lemma}
\begin{proof}
(1)~~This holds for $\Lambda = \nu P$ with $P \in \pcat$ by \cref{prop:homotopymoduleherd}. As both sides preserve geometric realizations in $\Lambda$, the left by construction and the right by \cite[Corollary 4.1.8]{usd1}, it therefore holds for $\Lambda$ arbitrary.

(2)~~This is a restatement of \cref{rmk:modifyalongspace} in the case $T = S^n$, as $\tau_{<r}\map(S^n,\bs)\simeq K_{n,r}^0$ by \cref{prop:homotopymoduleherd}.

(3)~~Both constructions
\[
\Lambda \mapsto \tau_{(r,i)!}(\Lambda_T),\qquad \Lambda \mapsto (\tau_{(r,i)!}\Lambda)_T
\]
preserve geometric realizations, so it suffices to show that they agree on restriction to representable objects. If $P\in\pcat$, then because loop models are closed under limits, we can identify
\begin{align*}
\tau_{(r,i)!}(\nu_r(P)_T)&\simeq \tau_{(r,i)!}(\tau_{<r}(\nu(P)^T))\simeq \tau_{(r,i)!}(\tau_{r!}(\nu(P)^T))\\
&\simeq \tau_{i!}(\nu(P)^T)\simeq \tau_{<i}(\nu(P)^T)\simeq \nu_i(P)_T\simeq (\tau_{(r,i)!}\nu_r(P))_T
\end{align*}
as claimed. Here we use the fact that, if $X$ is a loop model, then $\tau_{n!}X\simeq X_{<n}$ \cite[Proposition 7.1.10]{usd1}.
\end{proof}

\begin{ex}
If $X \in \Model_{\pcat}$, then
\[
(\tau_! X)_{S^n} = \Pi_{n!} X,
\]
and this is a $\mathbb{Z}$-module object over $\tau_{!}X$. 
\end{ex}

We now situate these derived homotopy modules inside derived Postnikov squares. The general theorem is the following.

\begin{theorem}
\label{thm:spiralmodel}
Let $\pcat$ be a Malcev theory and $X \in \Model_\pcat$ be a model. Then
\[
\elementarymodifiers\to\Model_\pcat,\qquad F\mapsto \pi_F^\ast\pi_{F!}X
\]
is a convergent spiral system. 
\end{theorem}

\begin{proof}
We first verify that this is a spiral system. Suppose given a levelwise cartesian square
\begin{center}\begin{tikzcd}
F'\ar[r,"\alpha'"]\ar[d,"\eta'"]&F\ar[d,"\eta"]\\
G'\ar[r,"\alpha"]&G
\end{tikzcd}\end{center}
of elementary modifiers in which $\eta$ is a levelwise effective epimorphism. \cref{prop:properties_of_elementary_modifiers}.(5) implies that if $\pcat$ is a Malcev theory, then
\begin{center}\begin{tikzcd}
\pcat_{F'}\ar[r,"\pi_{\alpha'}"]\ar[d,"\pi_{\eta'}"]&\pcat_F\ar[d,"\pi_\eta"]\\
\pcat_{G'}\ar[r,"\pi_\alpha"]&\pcat_G
\end{tikzcd}\end{center}
is a cartesian square of Malcev theories and homomorphisms. As $\eta$ is a levelwise effective epimorphism, $\pi_\eta$ is full, and therefore \cref{prop:weakpb} implies that if $Y \in \Model_{\pcat_{F'}}$ and we write $\delta = \alpha \circ \eta' = \eta\circ \alpha'$, then
\[
Y \to \pi_{\eta'}^\ast\pi_{\eta'!}Y\times_{\pi_\delta^\ast\pi_{\delta!}Y} \pi_{\alpha'}^\ast\pi_{\alpha'!}Y
\]
is an equivalence. Specializing to the case where $Y = \pi_{F'!}X$ and further applying $\pi_{F'}^\ast$ then proves that
\[
\pi_{F'}^\ast\pi_{F'!} X \to \pi_{G'}^\ast \pi_{G'!}X \times_{\pi_G^\ast \pi_{G!}X}\pi_F^\ast\pi_{F!}X
\]
is an equivalence as claimed. It remains to prove that this spiral system is convergent, i.e.\ that
\[
X \to \lim_{n\to\infty}\tau_n^\ast\tau_{n!}X
\]
is an equivalence. This holds as $X \to \tau_n^\ast \tau_{n!}X$ is $(n-1)$-connective by \cref{lem:localconnectedunit}.
\end{proof}

Having constructed a spiral system, we single out the Postnikov squares themselves. 

\begin{definition}\label{constr:derivedpostnikovsquare}
Let $\pcat$ be a Malcev theory and $1 \leq r \leq n < \infty$ be integers. The \emph{derived Postnikov square} of a model $X \in \Model_{\h_{n+r}\pcat}$ is the cartesian square
\begin{center}\begin{tikzcd}
X\ar[r]\ar[d]&\tau_{(n+r,r)}^\ast\tau_{(n+r,r)!}X\ar[d]\ar[dr,equals]\\
\tau_{(n+r,n)}^\ast\tau_{(n+r,n)!}X\ar[r]&\tau_{(n+r,r)}^\ast B^{n+1}_{\tau_{(n+r,r)!}X}(\tau_{(n+r,r)!}X)_{S^n}\ar[r]&\tau_{(n+r,r)}^\ast\tau_{(n+r,r)!}X 
\end{tikzcd}\end{center}
obtained by applying \cref{prop:weakpb} to the categorical Postnikov square
\begin{center}\begin{tikzcd}
\h_{n+r}\pcat\ar[r,"\tau_{(n+r,r)}"]\ar[d,"\tau_{(n+r,n)}"']&\h_r\pcat\ar[d,"0"]\ar[dr,equals]\\
\h_n\pcat\ar[r,"k"]&\kinv_{n,r}^{n+1}\pcat\ar[r,"k"]&\h_r\pcat
\end{tikzcd}\end{center}
of $\pcat$ in the sense of \cref{constr:categoricalpostnikovsquare}, and applying \cref{lem:pseudocotensor} to identify $0^\ast 0_! \Lambda = B^{n+1}_\Lambda\Lambda_{S^n}$ for $\Lambda \in \Model_{\h_r\pcat}$.
\end{definition}

In particular, any model $X \in \Model_\pcat$ admits a \emph{derived Postnikov tower}
\[
X\simeq \lim_{n\to\infty}\tau_n^\ast\tau_{n!}X,
\]
with layers fitting into cartesian squares 
\begin{center}\begin{tikzcd}
\tau_{(n+r)}^\ast\tau_{(n+r)!}X\ar[r]\ar[d]&\tau_r^\ast\tau_{r!}X\ar[d]\\
\tau_n^\ast\tau_{n!}X\ar[r]&\tau_r^\ast B^{n+1}_{\tau_{r!}X}(\tau_{r!}X)_{S^n}
\end{tikzcd}.\end{center}
for any $r \leq n$. This gives rise to the following useful decomposition of mapping spaces in $\Model_\pcat$:

\begin{prop}\label{cor:mappingspacedecomposition}
Given $X,Y\in \Model_\pcat$, there is a decomposition
\[
\map_\pcat(Y,X)\simeq\lim_{n\to\infty}\map_{\h_n\pcat}(\tau_{n!}Y,\tau_{n!}X),
\]
with layers fitting into cartesian squares
\begin{center}\begin{tikzcd}
\map_{\h_{n+r}\pcat}(\tau_{(n+r)!}Y,\tau_{(n+r)!}X)\ar[r]\ar[d]&\map_{\h_r\pcat}(\tau_{r!}Y,\tau_{r!}X)\ar[d]\ar[dr,equals]\\
\map_{\h_n\pcat}(\tau_{n!}Y,\tau_{n!}X)\ar[r]&\map_{\h_r\pcat}(\tau_{r!}Y,B^{n+1}_{\tau_{r!}X}(\tau_{r!}X)_{S^n})\ar[r]&\map_{\h_r\pcat}(\tau_{r!}Y,\tau_{r!}X)
\end{tikzcd}\end{center}
for all $1 \leq r \leq n < \infty$.
\end{prop}
\begin{proof}
Map $Y$ into the derived Postnikov tower of $X$ and apply the adjunction $\pi_{F!}\dashv \pi_F^\ast$.
\end{proof}

\begin{cor}
Fix $X,Y\in \Model_\pcat$ and a map $\phi\colon \tau_{r!}Y\to \tau_{r!}X$. For $n\geq r$, write
\[
\map_{\h_n\pcat}(\tau_{n!}Y,\tau_{n!}X)_\phi = \{\phi\}\times_{\map_{\h_r\pcat}(\tau_{r!}Y,\tau_{r!}X)}\map_{\h_n\pcat}(\tau_{n!}Y,\tau_{n!}X)
\]
for the space of maps $\tau_{n!}Y\to\tau_{n!}X$ lifting $\phi$. Then there is a decomposition
\[
\map_\pcat(Y,X)_\phi\simeq\lim_{r\leq n\to\infty}\map_{\h_n\pcat}(\tau_{n!}Y,\tau_{n!}X)_\phi,
\]
with layers fitting into cartesian squares
\begin{center}\begin{tikzcd}
\map_{\h_{n+r}\pcat}(\tau_{(n+r)!}Y,\tau_{(n+r)!}X) \ar[r]\ar[d]&\{0\}\ar[d]\\
\map_{\h_n\pcat}(\tau_{n!}Y,\tau_{n!}X) \ar[r]& \map_{\h_r\pcat/\tau_{r!}X}(\tau_{r!}Y,B^{n+1}_{\tau_{r!}X}(\tau_{r!}X)_{S^n})
\end{tikzcd}.\end{center}
\end{cor}
\begin{proof}
This follows from \cref{cor:mappingspacedecomposition} by pulling back along $\{\phi\} \to \map_{\h_r\pcat}(\tau_{r!}Y,\tau_{r!}X)$.
\end{proof}

\subsection{Truncated and connected morphisms}\label{ssec:truncatedconnected}

Our goal in this subsection is to collect some technical properties we will need about connectivity. These will be used in the next section to prove \cref{thm:animatesquaregeneral}. 
Recall from \cite[\S3.3]{usd1} that if $\pcat$ is Malcev theory, then the $(n+1)$-connective and $n$-truncated morphisms form a factorization system on $\Model_\pcat$, with connectivity and factorizations reflected by the embedding $\Model_\pcat\subset\Fun(\pcat^\op,\spaces)$, i.e.\ computed levelwise. General connectivity properties of spaces therefore imply the corresponding properties for models of Malcev theories by working levelwise, and we begin by recording some of these.

\begin{lemma}\label{lem:retractconnect}
Let $X \to Y$ be an $n$-connective morphism. 
\begin{enumerate}
\item $X$ is $n$-connective if and only if $Y$ is $n$-connective.
\item If $X$ is $(n+1)$-connective, then $Y$ is $(n+1)$-connective,
\end{enumerate}
\end{lemma}
\begin{proof}
(1)~~This is \cite[Proposition 6.5.1.16(5)]{lurie2017highertopos}.

(2)~~Consider the diagram
\begin{center}\begin{tikzcd}
X\ar[r]\ar[d]&Y\ar[r]\ar[d]&\ast\ar[d]\\
X_{\leq n}\ar[r]&Y_{\leq n}\ar[r]&\ast
\end{tikzcd}.\end{center}
As $Y \to Y_{\leq n}$ is $(n+1)$-connective, $Y$ is $(n+1)$-connective if and only if $Y_{\leq n}$ is $n$-connective. As $X$ is $(n+1)$-connective, the composite $X_{\leq n} \to Y_{\leq n} \to \ast$ is an equivalence, and therefore \cite[Proposition 6.5.1.20]{lurie2017highertopos} implies that $Y_{\leq n}$ is $(n+1)$-connective if and only if $X_{\leq n} \to Y_{\leq n}$ is $n$-connective. This is a composite of the $n$-connective maps $X \to Y \to Y_{\leq n}$ and is therefore $n$-connective as claimed.
\end{proof}

\begin{defn}
A square
\begin{center}\begin{tikzcd}
A\ar[r,"f"]\ar[d,"p"]&B\ar[d,"q"]\\
C\ar[r,"g"]&D
\end{tikzcd}\end{center}
is said to be \emph{$n$-cartesian} if any of the following equivalent conditions hold:
\begin{enumerate}
\item The comparison map $A \to B\times_D C$ is $n$-connective;
\item The induced map $\Fib(f)\to \Fib(g)$ between fibres at each basepoint of $B$ is $n$-connective;
\item The induced map $\Fib(p) \to \Fib(q)$ between fibres at each basepoint of $C$ is $n$-connective.
\end{enumerate}
\end{defn}

\begin{lemma}
\label{lem:connectiveretract}
Consider a diagram
\begin{center}\begin{tikzcd}
A\ar[r,"f"]\ar[d]&B\ar[d]\\
C\ar[r,"g"]\ar[d]&D\ar[d]\\
A\ar[r,"f"]&B
\end{tikzcd}\end{center}
in which the vertical composites are the identity. If the bottom square is $(n+1)$-cartesian then the top square is $n$-cartesian, and the converse holds if $B \to D$ is an effective epimorphism.
\end{lemma}
\begin{proof}
Over every basepoint of $B$, this diagram induces a retraction
\[
\Fib(f) \to \Fib(g) \to \Fib(f).
\]
By \cite[Proposition 6.5.1.20]{lurie2017highertopos}, we find that $\Fib(f)\to \Fib(g)$ is $n$-connective if and only if $\Fib(g)\to \Fib(f)$ is $(n+1)$-connective. The lemma follows.
\end{proof}

\begin{lemma}\label{lem:cubelemmas}
Consider a cube
\begin{equation}\label{eq:cubelemma}\begin{tikzcd}
A\ar[rr,"q"]\ar[dr,"a"]\ar[dd,"p"]&&B\ar[dd,"i", near start]\ar[dr,"b"]\\
&A'\ar[rr,"q'" near start]\ar[dd,"p'" near start]&&B'\ar[dd,"i'"]\\
C\ar[rr,"k" near start]\ar[dr,"c"]&&D\ar[dr,"d"]\\
&C'\ar[rr,"k'"]&&D'
\end{tikzcd}\end{equation}
in which the front and back faces are cartesian, i.e.\ $A \simeq B\times_C D$ and $A'\simeq B'\times_{C'}D'$.
\begin{enumerate}
\item If the right face is $n$-cartesian, then the left face is $n$-cartesian, and the converse holds if $k$ is an effective epimorphism.
\item If the right face is $n$-cartesian and $c$ is $n$-connective, then $a$ is $n$-connective.
\item If the right face is $n$-cartesian, $a$ is $(n+1)$-connective, and $p'$ is an effective epimorphism, then $c$ is $(n+1)$-connective.
\end{enumerate}
\end{lemma}
\begin{proof}
(1)~~As the right face is $n$-cartesian and the front and back faces are cartesian, it follows that for any basepoint of $C$ the map
\[
\Fib(A \to C) \simeq \Fib(B \to D) \to \Fib(B' \to D') \simeq \Fib(A' \to C')
\]
is $n$-connective, and thus the left face is $n$-cartesian. Conversely, if $k$ is an effective epimorphism then any basepoint of $D$ may be lifted to a basepoint of $C$, after which the same argument applies.

(2)~~Suppose that $c$ is $n$-connective. The map $a\colon A \to A'$ is $n$-connective if and only if $\Fib (a)$ is $n$-connective at all basepoints of $A'$. These fibres fit into cartesian squares
\begin{equation}\label{eq:cubelemma1} \begin{tikzcd}
\Fib (a) \ar[r]\ar[d]&\Fib (b)\ar[d]\\
\Fib (c) \ar[r]&\Fib (d)
\end{tikzcd}.\end{equation}
As the rightmost square of (\ref{eq:cubelemma}) is $n$-cartesian, the map $\Fib (b) \to \Fib (d)$ is $n$-connective. As $n$-connective morphisms are stable under pullback, it follows that $\Fib(a) \to \Fib(c)$ is $n$-connective. As $\Fib(c)$ is $n$-connective, it follows from \cref{lem:retractconnect}.(1) that $\Fib(a)$ is $n$-connective.

(3)~~Suppose that $a$ is $(n+1)$-connective and that $p'$ is an effective epimorphism. The map $c$ is $(n+1)$-connective if and only if $\Fib(c)$ is $(n+1)$-connective at all basepoints of $C'$. As $p'$ is an effective epimorphism, any basepoint of $C'$ lifts to a basepoint of $A'$, and so $\Fib(c)$ fits into a commutative diagram as in (\ref{eq:cubelemma1}). As the rightmost square of (\ref{eq:cubelemma}) is $n$-cartesian, $\Fib(b) \to \Fib(d)$ is $n$-connective, and therefore $\Fib(a)\to \Fib(c)$ is $n$-connective. As $\Fib(a)$ is $(n+1)$-connective, it follows from \cref{lem:retractconnect}.(2) that $\Fib(c)$ is $(n+1)$-connective.
\end{proof}

We next establish a key technical property of derived functors, \cref{prop:connectivefibres} below, that is special to working with Malcev theories.

\begin{lemma}\label{lem:hsection}
Consider a diagram
\begin{center}\begin{tikzcd}
B\ar[d,"s"]\ar[r,"f"]\ar[dd,"\id_B"',bend right=15mm]&\ar[d,"s'"]\ar[dd,"\id_{B'}",bend left=15mm]B'\\
E\ar[r,"f'"]\ar[d,"p"]&E'\ar[d,"p'"]\\
B\ar[r,"f"]&B'
\end{tikzcd}.\end{center}
Suppose that the bottom square admits a Malcev operation. If $f'$ is $n$-connective, then the bottom square is $n$-cartesian.
\end{lemma}
\begin{proof}
This is an easy consequence of \cite[Proposition 2.2.14]{usd1}.
\end{proof}

\begin{prop}\label{prop:connectivefibres}
Suppose that $m\geq -1$ and that $f\colon \pcat\to\qcat$ is a locally $m$-connective homomorphism between Malcev theories. If $X \to Y$ is an $n$-connective map in $\Model_\pcat$, then
\begin{equation}\label{eq:connectivefibredef}
\begin{tikzcd}
X\ar[r]\ar[d]&Y\ar[d]\\
f^\ast f_! X\ar[r]&f^\ast f_! Y
\end{tikzcd}\end{equation}
is $(n+m)$-cartesian.
\end{prop}

\begin{proof}
When $m=-1$, \cref{lem:localconnectedunit} implies that the vertical arrows of (\ref{eq:connectivefibredef}) are $n$-connective. As any map between $n$-connective objects is at least $(n-1)$-connective, it follows that the square is $(n-1)$-cartesian.

Now suppose that $m\geq 0$, and form the diagram
\begin{center}\begin{tikzcd}
\pcat\ar[drr,"\id",bend left]\ar[ddr,"\id"',bend right]\ar[dr,"j"]\\
&\pcat\times_\qcat\pcat\ar[r]\ar[d]&\pcat\ar[d,"f"]\\
&\pcat\ar[r,"f"]&\qcat
\end{tikzcd}\end{center}
with inner square cartesian. Here, $\pcat\times_\qcat\pcat$ is again a Malcev theory by \cite[Remark 3.1.3, Lemma 4.1.5.(3)]{usd1}
as the condition that $f$ is full ensures that the diagonal $j$ is essentially surjective.

Fix an $n$-connective morphism $X \to Y$ in $\Model_\pcat$. By applying \cref{prop:weakpb} to $j_! X$ and pulling back along $j^\ast$, this may be extended to a diagram in $\Model_\pcat$ of the form
\begin{center}\begin{tikzcd}
X\ar[dr]\ar[dd]\\
&Y\ar[dd]\\
j^\ast j_! X\ar[rr]\ar[dr]\ar[dd]&&X\ar[dr]\ar[dd]\\
&j^\ast j_! Y\ar[rr]\ar[dd]&&Y\ar[dd]\\
X\ar[rr]\ar[dr]&&f^\ast f_! X\ar[dr]\\
&Y\ar[rr]&&f^\ast f_! Y
\end{tikzcd},\end{center}
with front and back facees cartesian and left vertical composites the identity. We must show that the right face is $(n+m)$-cartesian. As $f$ is full, $X \to f^\ast f_! X$ is an effective epimorphism by \cref{lem:localconnectedunit}, and therefore by \cref{lem:cubelemmas}.(1) it suffices to prove that the bottom left face is $(n+m)$-cartesian. 

We now proceed by induction on $m$. First consider the base case $m=0$. After evaluating on an object $P \in \pcat$, the left face is a rectangle of the form
\begin{center}\begin{tikzcd}
X(P)\ar[r]\ar[d]&Y(P)\ar[d]\\
j^\ast j_! X(P) \ar[r]\ar[d]&j^\ast j_! Y(P)\ar[d]\\
X(P)\ar[r]&Y(P)
\end{tikzcd}.\end{center}
This diagram admits a Malcev operation, and therefore \cref{lem:hsection} implies that the square is $n$-cartesian provided $j^\ast j_! X \to j^\ast j_! Y$ is $n$-connective, which holds if $X\to Y$ is $n$-connective by \cite[Proposition 6.2.1]{usd1}.

Next consider the inductive step $m > 0$. As $f$ is locally $m$-connective, $j$ is locally $(m-1)$-connective, and as $m > 0$ this implies that $Y \to j^\ast j_! Y$ is an effective epimorphism. By induction, the top square in the diagram
\begin{center}\begin{tikzcd}
X\ar[r]\ar[d]&Y\ar[d]\\
j^\ast j_! X\ar[r]\ar[d]&j^\ast j_! Y\ar[d]\\
X\ar[r]&Y
\end{tikzcd}\end{center}
is $(n+m-1)$-cartesian, and therefore \cref{lem:connectiveretract} implies that the bottom square is $(n+m)$-cartesian.
\end{proof}

\begin{ex}
A map $f\colon R \to S$ between connective ring spectra induces a derived functor
\[
S\otimes_R(\bs)\colon \LMod_R^{\geq 0} \to \LMod_S^{\geq 0}.
\]
In this context, \cref{prop:connectivefibres} amounts to saying that if $M$ is an $n$-connective $R$-module and $f$ is $m$-connective, then the unit $M \to S\otimes_R M$ is $(n+m)$-connective.
\end{ex}

\subsection{Clutching for models of Malcev theories}
\label{subsection:clutching_for_malcev_theories}

In this section, we extend \cref{prop:weakpb} by giving conditions under which the comparison functor into the pullback of $\infty$-categories of models is an equivalence. 

\begin{defn}
We say that a derived functor $f\colon \pcat\pto\qcat$ \emph{reflects connectivity} if for every $n\geq 0$ and every morphism $\alpha\colon X \to Y$ in $\Model_\pcat$, if $f_! \alpha$ is $n$-connective then $\alpha$ is $n$-connective.
\end{defn}

\begin{rmk}
As a map $\alpha\colon X \to Y$ is an equivalence if and only if it is $n$-connective for all $n$, if $f\colon \pcat\pto\qcat$ reflects connectivity then $f_!\colon \Model_\pcat\to\Model_\qcat$ is conservative.
\end{rmk}

\begin{ex}
Let $f\colon \pcat\to\qcat$ be a homomorphism of Malcev theories. If $f$ is essentially surjective, then the restriction $f^\ast\colon \qcat\pto\pcat$ reflects connectivity.
\end{ex}

\begin{theorem}
\label{thm:animatesquaregeneral}
Consider a cartesian square
\begin{center}\begin{tikzcd}
\pcat'\ar[r,"g'"]\ar[d,"f'"]&\pcat\ar[d,"f"]\\
\qcat'\ar[r,"g"]&\qcat
\end{tikzcd}
$\qquad\qquad$
\begin{tikzcd}
\Model_{\pcat'}\ar[r,"g'_!"]\ar[d,"f'_!"]&\Model_\pcat\ar[d,"f_!"]\\
\Model_{\qcat'}\ar[r,"g_!"]&\Model_\qcat
\end{tikzcd}\end{center}
of Malcev theories and homomorphisms, and associated commutative square of $\infty$-categories of models. Suppose that $f$ is full and $f'$ is essentially surjective and that at least one of the following conditions is satisfied:
\begin{enumerate}
\item $f$ reflects connectivity;
\item $g$ is full;
\item $\pcat$, $\qcat$, and $\qcat'$ are additive,
\end{enumerate}
Then the associated square of $\infty$-categories of models is cartesian.
\end{theorem}
\begin{proof}
Abbreviate
\[
\dcat = \Model_{\qcat'}\times_{\Model_{\qcat}}\Model_{\pcat}.
\]
As $f$ is full, \cref{prop:weakpb} implies that the comparison map
\[
L\colon \Model_\pcat \to \dcat
\]
is fully faithful. To prove that $L$ is an equivalence, we must prove that its right adjoint $R$ is conservative. As with \cref{prop:weakpb}, an object of $\dcat$ may be represented by a triple $\langle Y,X,\alpha\rangle$ where $Y \in \Model_{\qcat'}$, $X \in \Model_{\pcat}$ and $\alpha\colon g_! Y \simeq f_! X$, and if we abbreviate $d = fg' = gf'$ then the right adjoint $R$ may be described by
\[
R(\langle Y,X,\alpha\rangle) = f'{}^\ast Y \times_{d^\ast f_! X} g'{}^\ast X.
\]
A morphism $\langle Y,X,\alpha\rangle \to \langle Y',X',\alpha'\rangle$ in $\dcat$ is represented by two morphisms $y\colon Y \to Y'$ and $x\colon X \to X'$ together with a commutative diagram
\begin{equation}\label{eq:morphisminpbcat}\begin{tikzcd}
g_! Y\ar[r,"f'_! y"]\ar[d,"\alpha","\simeq"']&g_! Y'\ar[d,"\alpha'","\simeq"']\\
f_!X\ar[r,"g'_! x"]&f_!X'
\end{tikzcd}\end{equation}
identifying $f'_!y\simeq g'_! x$. Fix such a morphism, and suppose that it is sent to an equivalence by $R$. Abbreviate
\[
Z = R(\langle Y,X,\alpha\rangle) \simeq R(\langle Y',X',\alpha'\rangle),
\]
so that we have a cube
\begin{equation}\label{eq:cartesiantheoremcube}\begin{tikzcd}
Z\ar[rr]\ar[dd]\ar[dr,equals]&&g'{}^\ast X\ar[dd]\ar[dr,"g'{}^\ast x"]\\
&Z\ar[rr]\ar[dd]&&g'{}^\ast X'\ar[dd]\\
f'{}^\ast Y\ar[rr]\ar[dr,"f'{}^\ast y"]&&d^\ast f_!X\ar[dr]\\
&f'{}^\ast Y'\ar[rr]&&d^\ast f_! X'
\end{tikzcd}\end{equation}
in which the front and back faces are cartesian. We claim that if $Y \to Y'$ is $n$-connective then $X \to X'$ is $n$-connective, and that if $X \to X'$ is $n$-connective then $Y \to Y'$ is $(n+1)$-connective. As every morphism is $(-1)$-connective, it follows by induction that $X \to X'$ and $Y \to Y'$ are equivalences, and thus $R$ is conservative.

First suppose that $X\to X'$ is $n$-connective. As $f$ is full, \cref{prop:connectivefibres} implies that the right face of (\ref{eq:cartesiantheoremcube}) is $n$-cartesian, and \cref{lem:localconnectedunit} implies that the map $g'{}^\ast X' \to d^\ast f_! X'$ is an effective epimorphism. Therefore $Z \to f'{}^\ast Y'$ is an effective epimorphism, and thus $f'{}^\ast Y \to f'{}^\ast Y'$ is $(n+1)$-connective by \cref{lem:cubelemmas}.(3). As $f'$ is essentially surjective, it follows that $Y \to Y'$ is $(n+1)$-connective.

Next suppose that $y\colon Y \to Y'$ is $n$-connective. By \cite[Proposition 6.2.1]{usd1}, it follows that $g_!y$ is $n$-connective. As $g_! y \simeq f_! x$, this implies that $f_!$ is $n$-connective. We now split into cases based on the additional hypothesies (1--3).

(1)~~Suppose that $f$ reflects connectivity. As $f_! x$ is $n$-connective, it follows that $x$ is $n$-connective.

(2)~~Suppose that $g$ is full. Consider the following rotation of (\ref{eq:cartesiantheoremcube}):
\begin{center}\begin{tikzcd}
Z\ar[rr]\ar[dd]\ar[dr,equals]&&f'{}^\ast Y\ar[dr,"f'{}^\ast y"]\ar[dd]\\
&Z\ar[rr]\ar[dd,two heads]&&f'{}^\ast Y'\ar[dd,two heads]\\
g'{}^\ast X\ar[dr,"g'{}^\ast x"]\ar[rr]&&d^\ast f_! X\ar[dr]\\
&g'{}^\ast X'\ar[rr]&&d^\ast f_! X'
\end{tikzcd}.\end{center}
As $g$ is full, the unit $Y' \to g^\ast f'_! Y'$ is an effective epimorphism by \cref{lem:localconnectedunit}. As $f'{}^\ast$ preserves effective epimorphisms, it follows that
\[
f'{}^\ast Y' \to f'{}^\ast g^\ast f'_! Y'\simeq d^\ast f_! X'
\]
is an effective epimorphism. As effective epimorphisms are stable under pullback, this implies that $Z \to g'{}^\ast X'$ is an effective epimorphism. As $f'{}^\ast Y \to f'{}^\ast Y'$ and $d^\ast f_! X \to f^\ast f_! X'$ are $n$-connective, the right face is $(n-1)$-cartesian. Therefore \cref{lem:cubelemmas}.(3) implies that $g'{}^\ast x$ is $n$-connective. As $g'$ is essentially surjective, it follows that $g'$ is $n$-connective.

(3)~~Suppose $\pcat$, $\qcat$, and $\qcat'$ are additive. By embedding these $\infty$-categories of models in their stabilization, we are free to take cofibres and must show that the cofibre of $x$ is $(n+1)$-connective. By passing to cofibres, we have a cartesian square
\begin{center}\begin{tikzcd}
0\ar[r]\ar[d]&g'{}^\ast C(x)\ar[d]\\
f'{}^\ast C(y)\ar[r]&d^\ast C(f_! x)
\end{tikzcd},\end{center}
implying that
\[
\pi_k f'{}^\ast C(y)\oplus \pi_k g'{}^\ast C(x) \cong \pi_k d^\ast C(f_! x)
\]
for all $k\in\integers$. As $y$ and $f_! x$ are $n$-connective, their cofibres are $(n+1)$-connective. Taking $k \leq n$ in the above isomorphism, this implies that $\pi_k g'{}^\ast C(x) = 0$ for $k \leq n$, and therefore $g'{}^\ast C(x)$ is $(n+1)$-connective. As $g'$ is essentially surjective, it follows that $C(x)$ is $(n+1)$-connective as claimed.
\end{proof}

This allows us to give the following.

\begin{definition}
\label{definition:animatedpostnikov}
Let $\pcat$ be a Malcev theory and $1 \leq r \leq n < \infty$ be integers. The \emph{spiral squares} of $\pcat$ are the diagrams of $\infty$-categories 
\begin{center}\begin{tikzcd}
\Model_{\h_{n+r}\pcat}\ar[r]\ar[d]&\Model_{\h_r\pcat}\ar[d,"0_!"]\ar[dr,equals]\\
\Model_{\h_n\pcat}\ar[r,"k_!"]&\Model_{\kinv_{n,r}^{n+1}\pcat}\ar[r]&\Model_{\h_r\pcat}
\end{tikzcd}\end{center}
obtained from the categorical Postnikov square of \cref{constr:categoricalpostnikovsquare} by passing to $\infty$-categories of models.
\end{definition}

\begin{lemma}\label{lem:composethickening}
Suppose we are given derived functors $f\colon \pcat\pto\qcat$ and $g\colon \qcat\pto\rcat$ between Malcev theories, and that $g$ reflects connectivity. Then $f$ reflects connectivity if and only if $g\circ f$ reflects connectivity.
\end{lemma}
\begin{proof}
This is clear in light of the fact that all derived functors also preserve connectivity \cite[Proposition 6.2.1]{usd1}.
\end{proof}

\begin{ex}\label{ex:retractionthickening}
Let $f\colon \pcat\to\qcat$ be a homomorphism which admits a retraction:
\begin{center}\begin{tikzcd}
\pcat\ar[d,"f"]\ar[dr,equals]\\
\qcat\ar[r]&\pcat
\end{tikzcd}.\end{center}
As the identity clearly reflects connectivity, \cref{lem:composethickening} implies that $f$ reflects connectivity.
\end{ex}

\begin{prop}
\label{proposition:spiral_squares_of_a_malcev_theory_are_cartesian}
The spiral squares of a Malcev theory are cartesian.
\end{prop}
\begin{proof}
By \cref{thm:animatesquaregeneral}, it suffices to show that the zero-section $0\colon \h_r\pcat\to\kinv_{n,r}^{n+1}\pcat$ is full and reflects connectivity. This is clear: it is full by construction, and reflects connectivity by \cref{ex:retractionthickening} as it is a section.
\end{proof}

\section{The spiral system of a Malcev theory}

Associated to a Malcev theory $\pcat$ is its categorical Postnikov tower
\[
\pcat\to\cdots\to\h_{n+1}\pcat\to\h_n\pcat\to\cdots\to\h_1\pcat.
\]
Passage to categories of models produces the \emph{spiral tower}
\[
\Model_\pcat\to\cdots\to\Model_{\h_{n+1}\pcat}\to\Model_{\h_n\pcat}\to\cdots\to\Model_{\h_1\pcat}.
\]
In \cref{proposition:spiral_squares_of_a_malcev_theory_are_cartesian}, we say that the functors appearing in this tower can be completed to certain cartesian spiral squares. The main goal of this section is to prove \cref{thm:modelspiral}, which shows that this tower in fact naturally extends to a convergent spiral system in the sense of \cref{definition:spiral_system}.

\subsection{A homology Whitehead theorem}
\label{ssec:homologywhitehead}

The main technical work of this section is to establish the following result. 

\begin{theorem}
\label{thm:homologywhitehead}
Let $\alpha\colon F \to G$ be a natural transformation between elementary modifiers. Then for any Malcev theory $\pcat$, the induced homomorphism $\pi_{\alpha} \colon \pcat_{F} \rightarrow \pcat_{G}$ reflects connectivity.
\end{theorem}

Here, we say that a homomorphism of Malcev theories reflects connectivity if the induced derived functor between $\infty$-categories of models does. 
The proof of \cref{thm:homologywhitehead} occupies this subsection and the next. In this subsection, we consider the special case of the truncation $\id \to \tau_{\leq 0}$, which asserts that $\tau\colon \pcat\to\h\pcat$ reflects connectivity for any Malcev theory $\pcat$. In view of the following example, this can be thought of as a version of the \emph{homology Whitehead theorem} in the context of Malcev theories. 

\begin{ex}
Let $\pcat = \langle \bigvee S^2\rangle \subset\spaces_\ast^{\geq 2}$ be the full subcategory of pointed spaces spanned by wedges of the $2$-sphere. Then $\h\pcat\simeq\lfrees_0(\integers)$ is equivalent to the theory of abelian groups, and the derived functor
\[
\tau_!\colon \Model_{\pcat} \to \Model_{\h\pcat}
\]
may be identified with the functor
\[
\Sigma^{-2}\widetilde{C}_\ast(X;\integers) \colon \spaces_\ast^{\geq 2} \to \Mod_{H\integers}^{\geq 0}
\]
of twice desuspended reduced integral chains on simply connected spaces. \cref{thm:homologywhitehead} now asserts that if $f\colon X \to Y$ is a map of simply connected spaces, then $f$ is $n$-connective if and only if the induced map $\widetilde{C}_\ast(X;\integers) \to \widetilde{C}_\ast(Y;\integers)$ is $n$-connective, recovering the classical homology Whitehead theorem for simply connected spaces \cite{whitehead1949combinatorial}.
\end{ex}

The key technical input we need for this case is the following.

\begin{prop}
\label{prop:thickeningcartesian}
Suppose we are given a cartesian square
\begin{center}\begin{tikzcd}
\pcat'\ar[r,"g'"]\ar[d,"f'"]&\pcat\ar[d,"f"]\\
\qcat'\ar[r,"g"]&\qcat
\end{tikzcd}\end{center}
of Malcev theories and homomorphisms. If $f$ is full and reflects connectivity, then the same is true of $f'$.
\end{prop}
\begin{proof}
As full functors are closed under pullbacks, we must only prove that $f'$ reflects $n$-connectivity. Fix a morphism $X \to Y$ in $\Model_{\pcat'}$. By \cref{prop:weakpb}, we may form a cube
\begin{center}\begin{tikzcd}
X\ar[dr]\ar[rr]\ar[dd]&&g'{}^\ast g'_! X\ar[dd]\ar[dr]\\
&Y\ar[rr]\ar[dd]&&g'{}^\ast g'_! Y\ar[dd]\\
f'{}^\ast f'_! X\ar[rr]\ar[dr]&&g'{}^\ast f^\ast f_! g'_! X\ar[dr]\\
&f'{}^\ast f'_! Y\ar[rr]&&g'{}^\ast f^\ast f_! g'_! Y
\end{tikzcd}\end{center}
in which the front and back faces are cartesian. Suppose that $f'_! X \to f'_! Y$ is $n$-connective. By \cite[Proposition 6.2.1]{usd1}, it follows that $g_! f'_! X \to g_! f'_! Y$ is $n$-connective. This is equivalent to the map $f_! g'_! X \to f_! g'_! Y$, and therefore as $f_!$ reflects $n$-connective maps it follows that $g'_! X \to g'_! Y$ is $n$-connective. As $f$ is full it follows from \cref{prop:connectivefibres} that the right face of this cube is $n$-cartesian, and thus from \cref{lem:cubelemmas}.(2) that $X \to Y$ is $n$-connective.
\end{proof}

We can now give the following.

\begin{prop}\label{prop:basichomologywhitehead}
Let $\pcat$ be a Malcev theory. Then $\tau\colon \pcat\to\h\pcat$ reflects connectivity.
\end{prop}
\begin{proof}
It suffices to fix $n \geq 0$ and prove that $\tau_!$ reflects $n$-connective morphisms. As $\h\pcat\simeq\h(\h_m\pcat)$ and $X \to \tau_m^\ast\tau_{m!} X$ is $(m-1)$-connective for all $m > 0$, we are then free to replace $\pcat$ with $\h_m\pcat$ for some sufficiently large $m$. As derived functors which reflect connectivity are closed under composition by \cref{lem:composethickening}, by induction this reduces us to proving that if $\pcat$ is a Malcev theory then $\h_{m+1}\pcat\to\h_m\pcat$ reflects connectivity. Indeed, $\h_{m+1}\pcat\to\h_m\pcat$ sits in the categorical Postnikov square of \cref{constr:categoricalpostnikovsquare}:
\begin{center}\begin{tikzcd}
\h_{m+1}\pcat\ar[r]\ar[d]&\h\pcat\ar[d,"0"]\ar[dr,equals]\\
\h_m\pcat\ar[r]&\kinv_{m,1}^{m+1}\pcat\ar[r]&\pcat
\end{tikzcd}.\end{center}
The map $0\colon \h\pcat\to\kinv_{m,1}^{m+1}\pcat$ reflects connectivity by \cref{ex:retractionthickening}, and thus $\h_{m+1}\pcat\to\h_m\pcat$ reflects connectivity by \cref{prop:thickeningcartesian}.
\end{proof}

\begin{cor}\label{cor:hthickening}
Let $f\colon \pcat\to\qcat$ be a homomorphism between Malcev theories. If $f_1\colon \h\pcat\to\h\qcat$ reflects connectivity, then $f$ reflects connectivity.
\end{cor}
\begin{proof}
Consider the diagram
\begin{center}\begin{tikzcd}
\pcat\ar[r,"f"]\ar[d,"\tau"]&\qcat\ar[d,"\tau"]\\
\h\pcat\ar[r,"f_1"]&\h\qcat
\end{tikzcd}.\end{center}
If $f_1$ reflects connectivity, then as $\pcat\to\h\pcat$ reflects connectivity by \cref{thm:homologywhitehead}, the composite $\pcat\to\h\qcat$ reflects connectivity by \cref{lem:composethickening}, and therefore $\pcat\to\qcat$ reflects connectivity by \cref{lem:composethickening}.
\end{proof}

\begin{cor}\label{cor:homologywhiteheadhalf}
To establish \cref{thm:homologywhitehead}, it suffices to prove that if $\pcat$ is a Malcev theory and $F$ is an elementary modifier, then the unique natural transformation $\tau_{\leq 0} F \to \tau_{\leq 0}$ induces a homomorphism $\epsilon_F\colon \pcat_{\tau_{\leq 0} F}\simeq \h(\pcat_F) \to \h\pcat$ which reflects connectivity.
\end{cor}
\begin{proof}
Let $\alpha\colon F \to G$ be a natural transformation between elementary modifiers and $\pcat$ be a Malcev theory. By \cref{cor:hthickening}, to prove that $\pi_\alpha\colon \pcat_F \to \pcat_G$ reflects connectivity it suffices to prove that $\pi_{\ol{\alpha}}\colon \h(\pcat_F) \to \h(\pcat_G)$ reflects connectivity. As $\h(\pcat_F)\simeq \pcat_{\tau_{\leq 0} F}$ for any modifier $F$, this homomorphism sits in a diagram
\begin{center}\begin{tikzcd}
\pcat_{\tau_{\leq 0} F}\ar[rr,"\pi_{\ol{\alpha}}"]\ar[dr,"\epsilon_F"']&&\pcat_{\tau_{\leq 0} G}\ar[dl,"\epsilon_G"]\\
&\h\pcat
\end{tikzcd}.\end{center}
By \cref{lem:composethickening}, it therefore suffices to prove that $\epsilon_F$ and $\epsilon_G$ reflect connectivity.
\end{proof}

\subsection{Path components of elementary modifiers}
\label{subsection:path_components_of_elementary_modifiers}

By \cref{cor:homologywhiteheadhalf}, the proof of \cref{thm:homologywhitehead} will be completed once we have established the following.

\begin{prop}\label{prop:discretereflect}
Let $F$ be an elementary modifier and $\pcat$ be a Malcev theory. Then the unique natural transformation $\tau_{\leq 0}F \to \tau_{\leq 0}$ induces a homomorphism $\epsilon_F\colon \pcat_{\tau_{\leq 0}F} \to \h\pcat$ which reflects connectivity.
\end{prop}

The proof will occupy the rest of this subsection. 

\begin{rmk}
\cref{prop:discretereflect} is one of the more technical results of this paper. In practice, the elementary modifiers $F$ that one considers frequently satisfy $\tau_{\leq 0}F\simeq \tau_{\leq 0}$, where there is nothing to show. In particular, the work of this subsection is not needed for our applications to decompositions of moduli spaces in \S\ref{sec:modulispaces}.
\end{rmk}

\begin{ex}\label{ex:pireflect}
Let $\pcat$ be a Malcev theory. We claim that $\pcat_{\tau_{\leq 0}\map(S^n,\bs)}\to \h\pcat$ reflects connectivity for all $n\geq 1$. By \cref{prop:homotopymoduleherd}, we may identify $\tau_{\leq 0}\map(S^n,\bs)\cong \Pi_n(\bs)$. As a consequence, the $\infty$-category $\pcat_{\tau_{\leq 0}\map(S^n,\bs)}\simeq\pcat_{\Pi_n}$ sits in a cartesian square
\begin{center}\begin{tikzcd}
\pcat_{\tau_{\leq 0}\map(S^n,\bs)}\ar[r]\ar[d]&\h\pcat\ar[d,"i"]\ar[dr,equals]\\
\h\pcat\ar[r,"i"]&\pcat_{B_{\tau_{\leq 0}}\Pi_n}\ar[r]&\h\pcat
\end{tikzcd}.\end{center}
As $i$ admits a retraction, it reflects connectivity. As $i$ is full, it follows from \cref{prop:thickeningcartesian} that $\pcat_{\tau_{\leq 0}\map(S^n,\bs)} \to \h\pcat$ reflects connectivity.
\end{ex}

Our basic strategy for proving \cref{prop:discretereflect} is inspired by \cref{ex:pireflect}. For inductive purposes it will be convenient to introduce the following definitions, which will not be used anywhere else in the text.

\begin{defn}
A discrete coproduct-preserving modifier $F$ is \emph{admissible} if for every Malcev theory $\pcat$, the unique natural transformation $F \to \tau_{\leq 0}$ induces a homomorphism $\pcat_F\to\h\pcat$ which reflects connectivity.

A connected space $K$ is \emph{admissible} if every discrete coproduct-preserving modifier $F$ that admits an effective epimorphism $\map(K,\bs) \to F$ is admissible.
\end{defn}

If $F$ is an elementary modifier, then there exists some $T \in \Sph$ for which there exists an effective epimorphism $\map(T,\bs) \to F$. Thus \cref{prop:discretereflect} will be proved once we have shown that every $T\in \Sph$ is admissible. We will prove this by induction up a cellular filtration of $T$. To carry out this argument, we need some information about the structure of cell attachments.

\begin{recollection}[{The action of a loop space on fibre, \cite{safronov2017free-loop-space}}]
\label{recollection:action_of_loop_space_on_fibre}
If $\ast \rightarrow A$ is a pointed object in an $\infty$-topos, then by \cite[6.1.2.11]{lurie2017highertopos} its C\v{e}ch nerve is a groupoid which exhibits an $\mathbf{E}_{1}$-group structure on $\Omega A \colonequals \ast \times_{A} \ast$. If $B \rightarrow A$ is a morphism with fibre $F$, then taking \v{C}ech nerves horizontally in the cartesian square
\[
\begin{tikzcd}
	F & B \\
	\ast & A
	\arrow[from=1-1, to=1-2]
	\arrow[from=1-1, to=2-1]
	\arrow[from=1-2, to=2-2]
	\arrow[from=2-1, to=2-2]
\end{tikzcd}
\]
yields a morphism of simplicial objects which is a left action object in the sense of \cite[{Definition 4.2.2.2}]{higher_algebra}. That is, the \v{C}ech nerve of $F \rightarrow B$ can be identified with an action groupoid 
\begin{center}\begin{tikzcd}
\cdots\ar[r,shift left=1mm]\ar[r]\ar[r,shift right=1mm]&\Omega A \times F \ar[r,shift left=0.5mm]\ar[r,shift right=0.5mm]&F
\end{tikzcd}\end{center}
exhibiting $\Omega A$ as acting on $F$.
\end{recollection}

\begin{remark}
\label{remark:homotopy_quotient_of_action_of_loop_space_on_fibre}
In the context of \cref{recollection:action_of_loop_space_on_fibre}, as the geometric realization of an action groupoid can be identified with the corresponding homotopy orbits, we may identify
\[
F_{h \Omega A} \simeq | F^{\times_{B} \bullet} | \simeq |\check{C}(F\to B)| \simeq \mathrm{im}(F \rightarrow B).
\]
This is the union of those path components of $B$ which live over the basepoint component of $A$. 
\end{remark}

\begin{construction}\label{constr:coact}
Suppose we are given a cocartesian square 
\begin{center}\begin{tikzcd}
S^{n-1}\ar[r]\ar[d]&K\ar[d]\\
\ast\ar[r]&K'
\end{tikzcd}\end{center}
for some $n\geq 1$, realizing $K'$ as obtained by attaching a positive-dimensional cell for $K$. Mapping into a space $X$ we obtain a cartesian diagram 
\[
\begin{tikzcd}
	{\map(K', X)} & {\map(K, X)} \\
	X & {\map(S^{n-1},X)}
	\arrow[from=1-1, to=1-2]
	\arrow[from=1-1, to=2-1]
	\arrow[from=1-2, to=2-2]
	\arrow[from=2-1, to=2-2]
\end{tikzcd}
\]
of spaces over $X$. Applying \cref{recollection:action_of_loop_space_on_fibre} in the $\infty$-topos $\spaces_{/X}$, we see that 
\begin{enumerate}
    \item $\map(S^{n}, X) \simeq \Omega_{X} \map(S^{n-1}, X)$ is an $\mathbf{E}_{1}$-group object in spaces over $X$;
    \item $\map(S^n,X)$ acts on $\map(K', X)$;
    \item The homotopy orbits for this action can be identified
    \[
\map(K', X)_{\h \map(S^{n}, X)} \simeq \map(K, X)_{\mathrm{ext}}
    \]
    with the union of those path components of $\map(K, X)$ which correspond to maps $K \to X$ which extend to $K'$; equivalently, for which the composite $S^{n-1} \rightarrow K \rightarrow X$ is null-homotopic. 
\end{enumerate}
Here, the last part is \cref{remark:homotopy_quotient_of_action_of_loop_space_on_fibre}. 
\end{construction}

We now turn our attention to the cell structures on the objects of $\spheresandmore$.

\begin{defn}
We say that a map $f\colon A \to B$ of connected spaces is \emph{$\Sigma$-null} if $\Sigma f$ is nullhomotopic.
\end{defn}

The $\Sigma$-null maps play a role in the study of supersimple spaces by way of the following.

\begin{lemma}\label{lem:sigmanull}
A map $f\colon A \to B$ of connected spaces is $\Sigma$-null if and only if the following condition is satisfied:
\begin{enumerate}
\item[$(\ast)$] For every supersimple space $X\in \ukanspaces$ and every map $B \to X$, the composite $A \to B \to X$ is nullhomotopic, i.e.\ homotopic to a constant map.
\end{enumerate}
\end{lemma}
\begin{proof}
After choosing a basepoint of $A$, we may suppose without loss of generality that $f$ is a map of pointed connected spaces.

Suppose first that $(\ast)$ is satisfied. As $\eta\colon B \to \Omega \Sigma B$ is a map into a supersimple space, the assumption guarantees that the composite $\eta\circ f\colon A \to B \to \Omega\Sigma B$ is nullhomotopic. As $\Sigma f$ is adjoint to $f\circ \eta$, it follows that $\Sigma f$ is nullhomotopic.

Suppose conversely that $f$ is $\Sigma$-null. Fix a supersimple space $X$ and map $g\colon B \to X$. As $B$ is connected, this map lands in a single path component of $X$. By replacing $X$ with this path component and pointing $X$ by the image of the basepoint of $B$, we may as well suppose that $X$ is pointed and connected. Thus $X$ admits the structure of an $H$-space \cite[Theorem 2.2.13]{usd1}, and by a classical result of James \cite[{Theorem 1.8}]{james1955reduced} it follows that $\eta\colon X \to \Omega\Sigma X$ admits a retraction. Thus to show that $g\circ f\colon A \to B \to X$ nullhomotopic it suffices to show that $\eta\circ g \circ f\colon A \to B \to X \to \Omega\Sigma X$ is nullhomotopic, which follows as this composite is adjoint to $\Sigma g\circ\Sigma f\colon \Sigma A \to \Sigma B \to \Sigma X$ and $\Sigma f$ is nullhomotopic by assumption.
\end{proof}

\begin{lemma}
\label{lem:ssnull}
Let $T\in\spheresandmore$. Then $T$ admits a finite filtration
\[
\ast = T_0 \subset \cdots \subset T_m = T
\]
with the following property:
\begin{enumerate}
\item[($\ast$)] For all $0 \leq k < m$, the space $T_{k+1}$ is obtained by attaching a positive-dimensional cell to $T_k$ via a $\Sigma$-null attaching map, i.e.\ there are cofibre sequences
\[
S^{n-1}\xrightarrow{\alpha} T_k \to T_{k+1}
\]
with $n \geq 1$ and for which $\alpha$ is $\Sigma$-null.
\end{enumerate}
\end{lemma}

\begin{proof}
For the purposes of this proof, we say that a connected space $T$ is \emph{strongly $\Sigma$-null} if it admits a filtration as above; this implies that $\Sigma T$ splits as a wedge of spheres, but is at least a priori a more stringent condition.

The class of strongly $\Sigma$-null spaces clearly contains $S^1$ and $\ast$, so we must show that if $T \in \spheresandmore$ is strongly $\Sigma$-null then so too are $S^1\vee T$, $S^1\wedge T$, and $S^1\times T$. For $S^1\vee T$ this is clear, as $S^1\vee T$ is obtained by attaching a $1$-dimensional cell to $T$ along the constant map $S^0 \to T$. For $S^1 \wedge T$, an easy induction based on the general equivalences
\[
S^1\wedge (A\times B)\simeq (S^1\wedge A) \vee (S^1\wedge A \wedge B)\vee (S^1\wedge B)
\]
shows that if $T \in \spheresandmore$ then $S^1\wedge T$ splits as a wedge of positive-dimensional spheres, which is clearly strongly $\Sigma$-null.

It remains to show that if $T \in \spheresandmore$ is strongly $\Sigma$-null then $S^1\times T$ is strongly $\Sigma$-null. By induction up an $\Sigma$-null filtration for $T$, we may suppose inductively that we are given a cocartesian square
\begin{center}\begin{tikzcd}
S^{n-1}\ar[r,"\alpha"]\ar[d]&T\ar[d]\\
\ast\ar[r]&T'
\end{tikzcd}\end{center}
with the properties that $S^1\times T$ is strongly $\Sigma$-null and $\alpha$ is $\Sigma$-null, and must show that $S^1\times T'$ is strongly $\Sigma$-null. The space $S^1\times T'$ may be obtained from $S^1\times T$ by attaching two additional cells, through two cocartesian diagrams: the outer rectangle in
\begin{center}\begin{tikzcd}
S^{n-1}\ar[r,"\alpha"]\ar[d]&T\ar[r,"i"]\ar[d]&S^1\times T\ar[d]\\
\ast\ar[r]&T'\ar[r]&(S^1\times T) \cup_T T'
\end{tikzcd},\end{center}
where $i\colon T \to S^1\times T$ is induced by the basepoint of $S^1$, and
\begin{center}\begin{tikzcd}
S^{n-1}\ar[r,"\beta"]\ar[d]&(S^1\times T)\cup_T T'\ar[d]\\
\ast\ar[r]&S^1\times T'
\end{tikzcd}.\end{center}
As $\alpha$ is $\Sigma$-null, so is $i\circ \alpha$, so we must verify that $\beta$ is $\Sigma$-null. Indeed, the standard splitting
\[
\Sigma(S^1\times T)\simeq \Sigma S^1\vee \Sigma(S^1\wedge T)\vee \Sigma T
\]
allows us to identify the suspension of $(S^1\times T)\cup_T T' \to S^1\times T'$ with the map
\[
\Sigma S^1\vee \Sigma(S^1\wedge T) \vee \Sigma T' \to \Sigma S^1 \vee \Sigma(S^1\wedge T') \vee \Sigma T'.
\]
As $\alpha$ is $\Sigma$-null, $S^1\wedge T \to S^1\wedge T'$ is the inclusion of a wedge summand, hence so too is this map, implying that $\Sigma\beta$ is nullhomotopic and thus that $\beta$ is $\Sigma$-null as claimed.
\end{proof}

The preceding lemmas could be used to prove \cref{prop:discretereflect} in the special case where $F = \map(T,\bs)$ for some $T \in \Sph$, by an argument directly analogous to \cref{ex:pireflect}. To handle more general elementary modifiers, we will need some information about how group actions interact with Malcev operations.

\begin{lemma}\label{lem:eckmannhilton}
Let $G$ be a group object in the $1$-category of herds. Then $G$ is abelian and its group operation is determined by the unit $e\in G$ and Malcev operation by $x\cdot y = t(x,e,y)$.
\end{lemma}
\begin{proof}
The operation $x\ast y = t(x,e,y)$ defines a unital operation on $G$ which commutes with the group operation:
\[
(w\ast x)\cdot (y\ast z) = t(w,e,x)\cdot t(y,e,z) = t(w\cdot y,e,x\cdot z) = (w\cdot y)\ast (x\cdot z).
\]
The lemma then follows from the Eckmann--Hilton argument.
\end{proof}

\begin{lemma}\label{lem:descendaction}
Let $\pcat$ be a Malcev pretheory, and suppose we are given a group object
\[
\Pi \in \Ab(\Model_\pcat^{\nb,\heartsuit})
\]
acting on a pointed object 
\[
(X,x_0) \in (\Model_\pcat^{\nb,\heartsuit})_{\ast/}.
\]
Write $i\colon \Pi \to X$ for the map determined the action of $\Pi$ on $x_0$. Then for any effective epimorphism $\alpha\colon X \to Y$ of nonbounded models, the image
\[
\Pi' = \operatorname{Im}\left(\Pi \xrightarrow{i} X \xrightarrow{\alpha} Y\right)
\]
inherits the structure of a group object acting on $Y$. This action is free provided that $Y$ satisfies the following grouplike condition:
\begin{enumerate}
\item[($\ast$)] Every $P \in \pcat$ admits a Malcev cooperation with the property that if $y_1,y_2\in Y(P)$, then
\[
t(\bs,y_1,y_2)\colon Y(P) \to Y(P)
\]
is an injection.
\end{enumerate}
\end{lemma}
\begin{proof}
Abbreviate $\beta = \alpha \circ i$. First we show that $\Pi'$ is a quotient group of $\Pi$. This is a pointwise statement: we must show that if $P \in \pcat$ and $a,a',b,b'\in \Pi(P)$ satisfy $\beta (a) =\beta (a')$ and $\beta (b) = \beta (b')$, then $\beta (ab) = \beta (a'b')$. Indeed, if we choose a Malcev cooperation on $P$, then by \cref{lem:eckmannhilton} we may compute
\[
\beta (ab) = \beta (t(a,e,b)) = t(\beta(a),\beta (e),\beta (b)) = t(\beta (a'),\beta (e),\beta (b')) = \beta (a'b').
\]

Next we show that the action of $\Pi$ on $X$ descends to an action of $\Pi'$ on $Y$. This is again a pointwise statement: we must show that if $P \in \pcat$ and $a,a' \in X(P)$ and $x,x' \in Y(P)$, then $\alpha(a\cdot x) = \alpha(a'\cdot x')$. Indeed, observe that in general
\[
a\cdot x = t(a,e,e)\cdot t(x_0,x_0,x) = t(a\cdot x_0,e\cdot x_0,e\cdot x) = t(a\cdot x_0,x_0,x).
\]
It follows that
\[
\alpha(a\cdot x) = \alpha(t(i(a),x_0,x)) = t(\beta(a),\alpha(x_0),\alpha(x)) = t(\beta(a'),\alpha(x_0),\alpha(x')) = \alpha(a'\cdot x').
\]

Finally we show that if $(\ast)$ is satisfied, then the action of $\Pi'$ on $Y$ is free. We must show that if $P\in \pcat$ and $\ol{a} \in \Pi'(P)$ and $\ol{x} \in Y(P)$ satisfy $\ol{a}\cdot \ol{x} = \ol{x}$, then $\ol{a} = e$. Lift $\ol{a}$ and $\ol{x}$ to elements $a\in \Pi(P)$ and $x \in X(P)$. Then we want to show $\beta(a) = \beta(e)$. As
\[
t(\beta(a),\alpha(x_0),\alpha(x)) = \beta(a\cdot x) = \beta(x) = t(\beta(e),\alpha(x_0),\alpha(x)),
\]
this follows from the assumption that we can arrange for
\[
t(\bs,\alpha(x_0),\alpha(x))\colon Y(P) \to Y(P)
\]
to be injective.
\end{proof}

\begin{prop}\label{prop:nullattach}
Suppose we are given a cocartesian square
\begin{center}\begin{tikzcd}
S^{n-1}\ar[r,"\alpha"]\ar[d]&K\ar[d]\\
\ast\ar[r]&K'
\end{tikzcd}\end{center}
for some $n\geq 1$. Suppose that $K$ is admissible and that $\alpha$ is $\Sigma$-null. Then $K'$ is admissible.
\end{prop}
\begin{proof}
For a space $T$, let us abbreviate $\tau_{\leq 0}\map(T,\bs) = [T,\bs]$. By \cite[Corollary 2.2.15.(2)]{usd1}, the natural map
\[
\tau_{\leq 0} \left(\map(S^n,X)\times_X \map(K',X)\right) \to [S^n,X]\times_{\tau_{\leq 0} X}[K',X]
\]
is a bijection for any supersimple space $X$. Therefore the action of $\map(S^n,X)$ on $\map(K',X)$ over $X$ passes to an action of $\Pi_n X = [S^n,X]$ on $[K',X]$ over $\tau_{\leq 0}X$. This action is natural in $X$, i.e.\ defines an action of the modifier $\Pi_n$ on $[K',\bs]$ over $\tau_{\leq 0}$.

Consider a discrete coproduct-preserving modifier $F$ equipped with an effective epimorphism $[K',\bs] \to F$. By definition, the $\infty$-category of modifiers is the $\infty$-category of nonbounded models for a Malcev pretheory:
\[
\malcevmodifiers\simeq \Model_{\ukanspaces^\op}^{\nb}.
\]
By \cite[Proposition 3.2.1]{usd1}, the same is true of the slice category $\malcevmodifiers_{/\tau_{\leq 0}}$. The map $\ast \to K'$ equips $K'$ with a basepoint, and therefore makes $[K',\bs]$ into a pointed object of $\malcevmodifiers_{/\tau_{\leq 0}}^\heartsuit$. Hence by \cref{lem:descendaction} the image
\[
\Pi' = \operatorname{Image}\left(\Pi_n \to [K',\bs] \to F\right)
\]
inherits the structure of a group acting on $F$.

We claim that the actions of $\Pi_n$ on $[K',\bs]$ and of $\Pi'$ on $F$ are free. To that end, we claim that both $[K',\bs]$ and $F$ satisfy condition $(\ast)$ of \cref{lem:descendaction}, and that the map $\Pi_n \to [K',\bs]$ is monic. 

We first show that $F$ satisfies condition $(\ast)$. The free objects of the Malcev pretheory defining the slice category $\End_\sigma(\ukanspaces)_{/\tau_{\leq 0}}$ can be identified as pairs consisting of a supersimple space $X$ and point $x \in \tau_{\leq 0}X$. The value of $F \in \End_\sigma(\ukanspaces)_{/\tau_{\leq 0}}$ on this pair is given by
\[
F(X,x) = F(X) \times_{\tau_{\leq 0} X} \{x\}.
\]
As $F$ preserves coproducts, this is equivalent to the evaluation of $F$ at the path component of $x\in X$. We therefore reduce to showing that every connected supersimple space admits a Malcev operation $t$ with the property that if $y_1,y_2\in F(X)$ then
\[
t(\bs,y_1,y_2)\colon F(X) \to F(X)
\]
is an injection. Indeed, as $X$ is a connected supersimple space, it admits a grouplike $H$-space structure. Hence the Malcev operation on $F(X)$ again comes from a grouplike $H$-space structure, and such Malcev operations have this property. The same argument applies with $[K',\bs]$ in place of $F$, and so both $[K',\bs]$ and $F$ satisfy condition $(\ast)$ of \cref{lem:descendaction}.

We next show that the map $\Pi_n \to [K',\bs]$ is monic. In other words, we must show that if $X$ is any supersimple space, then the map $\Pi_n X \to [K',X]$ is monic. As both sides preserve coproducts in $X$, we may reduce to the case where $X$ is connected, and so admits a necessarily grouplike $H$-space structure by \cite[Theorem 2.2.13]{usd1}. In particular this makes $\Pi_n X \to [K',X]$ into a homomorphism of groups, and to show that it is monic it suffices to show that it has trivial kernel. By the cofiber sequence $K' \to S^n \to \Sigma K$, the kernel of this homomorphism consists of those maps $S^n \to X$ which extend along $S^n \to \Sigma K$. As $S^n \to \Sigma K$ is, up to orientation, the suspension of the $\Sigma$-null map $\alpha\colon S^{n-1}\to K$, it is nullhomotopic by assumption. Therefore any map which extends through it is zero as needed.

As $\Pi'$ acts freely on $F$, the quotient map $F \to F/\Pi'$ is a principle $\Pi'$-bundle over $\tau_{\leq 0}$, and by construction this map is compatible with $[K',\bs] \to [K,\bs]$ which is similarly a principle $\Pi_n$-bundle. In other words, we have a commutative diagram of the form
\begin{center}\begin{tikzcd}
\Pi_n\ar[rr]\ar[dr]\ar[dd]&&{[K',\bs]}\ar[rr]\ar[dd]\ar[dr]&&\tau_{\leq 0}\ar[dr]\ar[dd]\\
&\Pi'\ar[rr]\ar[dd]&&F\ar[rr]\ar[dd]&&\tau_{\leq 0}\ar[dd]\\
\tau_{\leq 0}\ar[rr]\ar[dr]&&{[K,\bs]}\ar[rr]\ar[dr]&&B_{\tau_{\leq 0}}\Pi_n\ar[dr]\\
&\tau_{\leq 0}\ar[rr]&&F/\Pi'\ar[rr]&&B_{\tau_{\leq 0}}\Pi'
\end{tikzcd}\end{center}
in which the front and back faces are cartesian and the diagonal maps are effective epimorphisms. By \cref{prop:thickeningcartesian}, to show that $F \to \tau_{\leq 0}$ reflects connectivity it suffices to show that $F/\Pi' \to B_{\tau_{\leq 0}}\Pi'$ reflects connectivity. By \cref{cor:hthickening}, it suffices to show that $F/\Pi' \to \tau_{\leq 0}$ reflects connectivity, i.e.\ that $F/\Pi'$ is admissible. As $[K,\bs] \to F/\Pi'$ is an effective epimorphism, this holds as $K$ is admissible by assumption.
\end{proof}

We can now give the following.

\begin{proof}[Proof of \cref{prop:discretereflect}]
It suffices to show that every $T\in\spheresandmore$ is admissible. As $T = \ast$ is clearly admissible, this follows by combining \cref{prop:nullattach} with \cref{lem:ssnull}.
\end{proof}

\subsection{Convergence of the spiral tower}

We can now give the following.

\begin{theorem}
\label{thm:modelspiral}
Let $\pcat$ be a Malcev theory. Then
\[
\elementarymodifiers\to\catinfty,\qquad F \mapsto \Model_{\pcat_F}
\]
defines a convergent spiral system of $\infty$-categories. In particular, the spiral tower of $\pcat$ converges:
\[
\Model_\pcat\simeq\lim_{n\to\infty}\Model_{\h_n\pcat}.
\]
\end{theorem}
\begin{proof}
We first verify that this defines a spiral system. As in the proof of \cref{thm:spiralmodel}, a levelwise pullback of elementary modifiers along a levelwise effective epimorphism induces a pullback of Malcev theories along a full homomorphism. This homomorphism reflects connectivity by \cref{thm:homologywhitehead}, and therefore the corresponding square of $\infty$-categories of models is cartesian by \cref{thm:animatesquaregeneral}.

We next establish convergence. We must prove that the comparison functor
\[
L \colon \Model_{\pcat} \rightarrow \lim_{n\to\infty}\Model_{\h_n\pcat}
\]
is an equivalence. This is a diagram of left adjoints, and so admits a right adjoint $R$. We first claim that the unit map $X \to RL X$ is an equivalence for any $X \in \Model_\pcat$. This unit map can be identified with the comparison map $X \rightarrow \lim_{n \to \infty} \tau_{n}^{*} \tau_{n!} X$, which is an equivalence by \cref{thm:spiralmodel}.

To prove that $L$ is an equivalence it is now enough to verify that $R$ is conservative. An object $X$ of the limit $\infty$-category can be identified with a family of models $X_{n} \in \Model_{h_{n} \pcat}$ together with equivalences $\tau_{(n+1, n)!} X_{n+1} \simeq X_{n}$, in which case 
\[
R(X) = \lim_{n\to\infty} \tau_n^\ast X_n.
\]
For any $m\geq n$, the map 
\[
X_m \to \tau_{(m,n)}^\ast \tau_{(m,n)!}X_m \simeq \tau_{(m,n)}^\ast X_n
\]
is $(n-1)$-connective by \cref{lem:localconnectedunit}. Therefore
\[
\tau_m^\ast X_m \to \tau_n^\ast X_n
\]
is $(n-1)$-connective, and hence in the limit $R(X) \to \tau_n^\ast X_n$ is at least $(n-2)$-connective. An arrow $f$ in the limit $\infty$-category can be identified with a compatible family of arrows $f_{n} \colon X_{n} \rightarrow Y_{n}$. If $R(f)$ is an equivalence, then it follows that $X_n \to Y_n$ is $(n-2)$-connective for each $n$. As $f_1 = \tau_{(n,1)!}f_n$, it follows that $f_1\colon X_1 \to Y_1$ is $(n-2)$-connective. As $n$ was arbitrary, it follows that $X_1 \to Y_1$ is an equivalence. By \cref{prop:basichomologywhitehead} applied to $\h_n\pcat$, it follows that $X_n \to Y_n$ is an equivalence, and as $n$ was arbitrary it follows that $f$ is an equivalence.
\end{proof}

\section{Deformation theory}\label{sec:modulispaces}

Following \cref{thm:modelspiral} and \cref{prop:spiralsquares}, associated to every Malcev theory $\pcat$ and integers $1 \leq r \leq n < \infty$ are cartesian squares
\begin{center}\begin{tikzcd}
\h_{n+r}\pcat\ar[r,"\tau_{(n+r,r)}"]\ar[d,"\tau_{(n+r,n)}"]&\h_r\pcat\ar[d,"0"]\\
\h_n\pcat\ar[r,"k"]&\kinv_{n,r}^{n+1}\pcat
\end{tikzcd}
$\qquad\leadsto\qquad$
\begin{tikzcd}
\Model_{\h_{n+r}\pcat}\ar[r,"\tau_{(n+r,r)!}"]\ar[d,"\tau_{(n+r,n)!}"]&\Model_{\h_r\pcat}\ar[d,"0_!"]\\
\Model_{\h_n\pcat}\ar[r,"k_!"]&\Model_{\kinv_{n,r}^{n+1}\pcat}
\end{tikzcd},\end{center}
where both $0$ and $0_!$ are the zero-sections of a suitable $\thesphere_{<r}$-module object. Our goal in this section is to situate these squares in the more general context of \emph{linear extensions} of Malcev theories and $\infty$-categories, generalizing the classical theory of square-zero extensions of rings. 

In particular, a linear extension $f\colon \qcat\to\pcat$ of Malcev theories can be thought of as a \emph{first-order} deformation. We explain how this leads to an obstruction theory for lifting along the derived functor $f_!\colon \Model_\qcat\to\Model_\pcat$, obtaining the examples given in the introduction in \S\ref{ssec:applications} as a special case.

\subsection{Loop spaces over categories}

We begin with some general observations about loop space objects in slice categories of $\catinfty$.

\begin{notation}
Given a functor $q_!\colon \kcat\to\hcat$ of $\infty$-categories, objects $X,Y \in \kcat$, and a map $\phi\colon q_!X \to q_!Y$ in $\hcat$, we write
\[
\map_\kcat(X,Y)_\phi = \{\phi\}\times_{\map_\hcat(q_!X,q_!Y)}\map_\kcat(X,Y)
\]
for the space of lifts of $\phi$ to a map $f\colon X \to Y$ with specified homotopy $q_!f\simeq \phi$.
\end{notation}

\begin{ex}\label{rmk:fibreoveradjunction}
Suppose that $q_!\colon \kcat\to\hcat$ admits a section $0_!\colon \hcat\to\kcat$ which admits a right adjoint $0^\ast\colon \kcat\to\hcat$. If $\Lambda \in \hcat$, then the map
\[
0^\ast 0_! \Lambda \simeq q_! 0_! 0^\ast 0_! \xrightarrow{\epsilon}q_! 0_! \Lambda\simeq\Lambda
\]
defines an object $0^\ast 0_! \Lambda \in \hcat_{/\Lambda}$ for which, if $\phi\colon \Gamma\to\Lambda$ is any map in $\hcat$, then
\[
\map_\kcat(0_!\Gamma,0_!\Lambda)_\phi\simeq\map_{\hcat/\Lambda}(\Gamma,0^\ast 0_! \Lambda).
\] 
\end{ex}

\begin{prop}
\label{lem:loopcats}
Fix $\hcat\in\catinfty$ and $\bcat \in (\catinfty)_{\hcat//\hcat}$, and set $\kcat = \Omega_\hcat\bcat \in (\catinfty)_{\hcat//\hcat}$.
\begin{enumerate}
\item Both $0_!\colon \hcat\to\bcat$ and $q_!\colon \kcat\to\hcat$ are conservative.
\item For $\Lambda \in \hcat$, there is an equivalence
\[
\{\Lambda\}\times_{\hcat}\kcat\simeq \{\id_\Lambda\}\times_{\aut_\hcat(\Lambda)}\aut_\bcat(0_!\Lambda).
\]
\item If $p_!\colon \bcat\to\hcat$ is also conservative, such as if $\bcat$ admits a further delooping, then
\[
\{\Lambda\}\times_{\hcat}\kcat \simeq \map_\bcat(0_!\Lambda,0_!\Lambda)_{\id_\Lambda}.
\]
\item For any map $\phi\colon\Gamma\to\Lambda$ in $\hcat$, there is an equivalence
\[
\map_\kcat(0_!\Gamma,0_!\Lambda)_\phi\simeq \{\phi\}\times_{\hcat^{[1]}}\kcat^{[1]}.
\]
\end{enumerate}
\end{prop}
\begin{proof}
(1)~~Consider the cartesian diagram
\begin{center}\begin{tikzcd}
\kcat\ar[r,"q_!"]\ar[d,"q_!"]&\hcat\ar[d,"0_!"]\ar[dr,equals]\\
\hcat\ar[r,"0_!"]&\bcat\ar[r,"p_!"]&\hcat
\end{tikzcd}.\end{center}
As $0_!\colon \hcat\to\bcat$ is a section to $p_!$, it is conservative. As conservative functors are preserved by base change, it follows that $q_!$ is conservative.

(2)~~Consider the diagram
\begin{center}\begin{tikzcd}
\{\Lambda\}\ar[r]\ar[d]&\{\Lambda\}\ar[d]&\{0_!\Lambda\}\ar[d]\ar[l]\\
\{\Lambda\}\ar[r]&\hcat &\bcat\ar[l,"p_!"']\\
\{\Lambda\}\ar[r]\ar[u]&\{\Lambda\}\ar[u]&\{0_!\Lambda\}\ar[u]\ar[l]
\end{tikzcd}.\end{center}
By comparing the limit taken vertically and horizontally, we may compute
\begin{align*}
\{\Lambda\}\times_{\hcat}\kcat &\simeq \Omega\left(\{\Lambda\}\times_{H(\tau_{<r})}\bcat\right)\\
&\simeq \Fib\left(\{0_!\Lambda\}\times_{\bcat}\{0_!\Lambda\} \xrightarrow{q_!} \{\Lambda\}\times_{\hcat}\{\Lambda\}\right)\\
&\simeq \{\id_\Lambda\}\times_{\aut_\hcat(\Lambda)}\aut_\bcat(0_!\Lambda).
\end{align*}

(3)~~If $p_!\colon \bcat\to\hcat$ is conservative, then the inclusion
\[
\{\id_\Lambda\}\times_{\aut_\hcat(\Lambda)}\aut_\bcat(0_!\Lambda) \to \map_\bcat(0_!\Lambda,0_!\Lambda)_{\id_\lambda}
\]
is an equivalence.

(4)~~Applying (2) to $\bcat^{[1]} \in (\catinfty)_{\hcat^{[1]}//\hcat^{[1]}}$ and $\kcat^{[1]}\simeq \Omega_{\hcat^{[1]}}\bcat^{[1]}$, we may compute
\begin{align*}
\{\phi\}\times_{\hcat^{[1]}}\kcat^{[1]}&\simeq\{\id_\phi\}\times_{\aut_{\hcat^{[1]}}(\phi)}\aut_{\bcat^{[1]}}(0_!\phi)\\
&\simeq\{\id_\phi\}\times_{\Omega_\phi\map_\hcat(\Gamma,\Lambda)}\Omega_{0_!\phi}\map_\bcat(\Gamma,\Lambda)\\
&\simeq\Omega\left(\{\phi\}\times_{\map_\hcat(\Gamma,\Lambda)}\map_\bcat(0_!\Gamma,0_!\Lambda)\right)\\
&\simeq\{\phi\}\times_{\map_\hcat(\Gamma,\Lambda)}\map_\kcat(0_!\Gamma,0_!\Lambda)\simeq \map_\kcat(0_!\Gamma,0_!\Lambda)_\phi
\end{align*}
as claimed.
\end{proof}

\begin{cor}
\label{cor:loopcore}
Fix $\hcat\in\catinfty$ and $\bcat \in (\catinfty)_{\hcat//\hcat}$, and set $\kcat = \Omega_\hcat\bcat \in (\catinfty)_{\hcat//\hcat}$. Suppose that $p\colon \bcat\to\hcat$ is conservative. Then the assignment
\[
\hcat \ni \Lambda \mapsto \map_\bcat(0_!\Lambda,0_!\Lambda)_{\id_\Lambda}
\]
determines a bundle of loop spaces over $\hcat^\core$ with total space $\kcat^\core$. In other words,
\[
\kcat^\core\simeq\colim\limits_{\Lambda\in \hcat^\core}\map_\bcat(0_!\Lambda,0_!\Lambda)_{\id_\Lambda}.
\]
\end{cor}
\begin{proof}
This is just a reinterpretation of \cref{lem:loopcats}.(3).
\end{proof}

\begin{rmk}
In the rest of this section, we will be working mostly with \emph{infinite} loop space objects in slice categories of $\catinfty$, as \cref{thm:beck} below allows cleaner statements to be made in this context. However, much of the content of this section generalizes easily to double loop space objects, and partly to loop space objects, allowing for an even more nonabelian deformation theory of Malcev theories; we leave this to the interested reader. 
\end{rmk}

\subsection{Linear extensions of \texorpdfstring{$\infty$}{infty}-categories}\label{ssec:linearextensions}

It is a classical theorem that if $\ccat$ is an ordinary category, then there is an equivalence 
\[
\Ab(\Cat_{/\ccat})\simeq\Fun(\Tw(\ccat),\Ab),
\]
where $\Tw(\ccat)$ is the \emph{twisted arrow category} of $\ccat$, see \cite[{Proposition 1.6}]{jibladzepirashvili2005linear}. Harpaz--Nuiten--Prasma prove the following $\infty$-categorical version of this result:

\begin{theorem}[{\cite[Theorem 1.0.3]{harpaznuitenprasma2018abstract}}]
\label{thm:beck}
There is an equivalence of $\infty$-categories
\[
\Sp((\catinfty)_{/\ccat})\simeq\Fun(\Tw(\ccat),\Sp)
\]
for any $\infty$-category $\ccat$.
\qed
\end{theorem}
Under this equivalence, an object $\dcat^\bullet \in \Sp((\catinfty)_{/\ccat})$ corresponds to a functor of the form
\[
D\colon \Tw(\ccat) \to \Sp,\qquad D(\phi\colon \Gamma\to\Lambda) =  \map_{\dcat^{\bullet+1}}(0_!\Gamma,0_!\Lambda)_\phi.
\]
More precisely, if we encode spectrum objects as functors out of $\calR_-^\omega(\thesphere)^\op$ as in \cref{def:moduleobjects}, then
\[
D(\phi\colon \Gamma\to\Lambda)(F) = \map_{\dcat^\bullet(\Omega F)}(0_!\Gamma,0_!\Lambda)_\phi.
\]
For the rest of this subsection, we fix a a connective $\bfE_1$-ring spectrum $R$. Then more generally
\[
\RMod_R((\catinfty)_{/\ccat})\simeq \Fun(\Tw(\ccat),\RMod_R)
\]
for any $\infty$-category $\ccat$, motivating the following definition.

\begin{defn}
A \emph{natural system of $R$-modules} on an $\infty$-category $\ccat$ is a functor 
\[
D\colon \Tw(\ccat)\to\RMod_R.
\]
\end{defn}

If $D$ is a natural system of $R$-modules, we write 
\[
\dcat^\bullet \in \RMod_R((\catinfty)_{/\ccat})
\]
for the corresponding $R$-module object in the sense of \cref{def:moduleobjects}, and for $n \geq 0$ we abbreviate $\dcat^n \colonequals \dcat(\Omega^n R)$. In particular, $\dcat^n\simeq\Omega_\ccat\dcat^{n+1}$.

\begin{ex}
\label{ex:pinrnaturalsystem}
Let $\ccat$ be an $\infty$-category enriched in spaces with vanishing Whitehead products. Then the $\thesphere_{<r}$-module object
\[
\kinv_{n,r}^\bullet\ccat \in \Mod_{\thesphere_{<r}}((\catinfty)_{/\h_r\ccat})
\]
corresponds to the \emph{suspension} of a natural system of connective $\thesphere_{<r}$-modules that we might denote
\[
H\Pi_{[n,n+r)}\colon \Tw(\h_r\pcat) \to \Mod_{\thesphere_{<r}}^{\geq 0},
\]
determined by
\[
\Omega^\infty H\Pi_{[n,n+r)}(\phi\colon \tau_r P \to \tau_r Q) = \Pi_{[n,n+r)}(\map_\pcat(P,Q),\phi).
\]
That is, if $\dcat^\bullet \in \Mod_{\thesphere_{<r}}((\catinfty)_{/\h_r\ccat})$ is associated to $H\Pi_{[n,n+r)}$ then $\kinv_{n,r}^{\bullet}\ccat\simeq \dcat^{\bullet+1}$.
\end{ex}

\cref{cor:loopcore} in this stable setting specializes to the following.

\begin{lemma}\label{lem:mappingspacenaturalsystem}
Let $D$ be a natural system of $R$-modules on an $\infty$-category $\ccat$. Then
\[
(\dcat^n)^\core\simeq \colim_{\Lambda \in \ccat^\core}\map_{\dcat^{n+1}}(0_!\Lambda,0_!\Lambda)_{\id_\Lambda}\simeq \colim_{\Lambda\in\ccat^\core} \Omega^{\infty}\Sigma^{n}D(\id_\Lambda)
\]
\end{lemma}
\begin{proof}
The first equivalence follows from \cref{cor:loopcore}, and second from the identification $\map_{\dcat^{n+1}}(0_!\Lambda,0_!\Lambda)_{\id_\Lambda}\simeq \Omega^{\infty-n}D(\id_\Lambda)$.
\end{proof}

\begin{cor}\label{cor:essurjzerosection}
Suppose that $D$ is a natural system of \emph{connective} $R$-modules, i.e.\ a functor
\[
D\colon \Tw(\ccat)\to\RMod_R^{\geq 0}.
\]
Then for any $n \geq 1$, the zero-section $0\colon \ccat\to\dcat^{n}$ is essentially surjective.
\qed
\end{cor}

We now introduce the following.

\begin{defn}\label{def:catsqz}
Let $D$ be a natural system of $R$-modules on an $\infty$-category $\ccat$ and $p\colon \bcat\to\ccat$ be an $\infty$-category over $\ccat$. A functor $f\colon \ecat\to\bcat$ is said to be a \emph{linear extension of $\bcat$ by $D$} if we have specified a cartesian square of $\infty$-categories over $\ccat$ of the form
\begin{center}\begin{tikzcd}
\ecat\ar[r,"t"]\ar[d,"f"]&\ccat\ar[d,"0"]\ar[dr,equals]\\
\bcat\ar[r,"k"]\ar[rr,bend right,"p"]&\dcat^2\ar[r,"q"]&\ccat
\end{tikzcd}.\end{center}
\end{defn}

\begin{rmk}\label{rmk:reducetoid}
Clearly linear extensions of $\bcat \in (\catinfty)_{/\ccat}$ by $D$ are equivalent to linear extensions of $\bcat\in (\catinfty)_{/\bcat}$ by $p^\ast D$. That is, in the context of  \cref{def:catsqz}, one would lose no generality by taking $\bcat = \ccat$ and $p\colon \bcat\to\ccat$ the identity.
\end{rmk}

\begin{rmk}\label{rmk:bwcohomology}
The \emph{Baues--Wirsching cohomology} of an $\infty$-category $\ccat$ with coefficients in a natural system $D\colon \Tw(\ccat)\to\RMod_R$ is defined for $n\in\integers$ by any of
\begin{align*}
\Hrm^{n}(\ccat;D) &\cong \pi_0 \Fun_{/\ccat}(\ccat,\dcat^n) \\
&\cong \pi_0 \map_{\Fun(\Tw(\ccat),\Sp)}(\Sigma^{-1}\thesphere,\Sigma^n D)\\
&\cong \pi_{-n-1} \lim D.
\end{align*}
It follows from the definition that linear extensions of $\bcat$ by $D$ are classified by the Baues--Wirsching cohomology group $\Hrm^2(\bcat;p^\ast D)$. Of course, if $\bcat$ is large then some care may be needed to interpret this, as there may not be a small set of linear extensions.
\end{rmk}

\begin{ex}
Taking $k = 0$, the \emph{trivial linear extension} of $\bcat$ by $D$ is $\bcat\times_\ccat\dcat^1$.
\end{ex}

\begin{ex}\label{ex:postnikovlinearextension}
Let $\ccat$ be an $\infty$-category with vanishing Whitehead products. Then the categorical Postnikov square of \cref{constr:categoricalpostnikovsquare} realizes $\h_{n+r}\ccat$ as a linear extension of $\h_n\ccat$ by the natural system $H\Pi_{[n,n+r)}$ of $\thesphere_{<r}$-modules on $\h_r\ccat$ of \cref{ex:pinrnaturalsystem} for all $1\leq r \leq n$.

If $\ccat$ does not have vanishing Whitehead products, then $\h_{n+r}\ccat$ will still be a linear extension of $\h_n\ccat$ by a natural system $H\Pi_{[n,n+r)}^\ns$ on $\h_{r+1}\ccat$ in the more resrictive range $1\leq r < n$, see \cref{variation:nonsimplepostnikov}. 
\end{ex}

\begin{ex}\label{ex:mappinglinext}
Consider a linear extension as in \cref{def:catsqz}. If $X,Y \in \ecat$, then as mapping spaces in a limit of $\infty$-categories are computed as a limit of mapping spaces, we have a cartesian square 
\begin{center}\begin{tikzcd}
\map_\ecat(Y,X)\ar[r]\ar[d]&\map_\ccat(tY,tX)\ar[d,"0"]\\
\map_\bcat(fY,fX)\ar[r]&\map_{\dcat^2}(0tY,0tX)
\end{tikzcd}\end{center}
with 
\[
\map_{\dcat^{\bullet}}(0tY,0tX) \in \Sp((\catinfty)_{/\map_\ccat(tY,tX)})\simeq\Sp(\spaces_{/\map_\ccat(tY,tX)}),
\]
realizing $\map_\ecat(Y,X)\to\map_\bcat(fY,fX)$ as a linear extension of spaces. 
\end{ex}

\begin{ex}
\label{ex:funlinext}
Consider a linear extension as in \cref{def:catsqz}. If $\jcat$ is an $\infty$-category, then as $\Fun(\jcat,\bs)$ preserves limits of $\infty$-categories, $\Fun(\jcat,\ecat)\to\Fun(\jcat,\bcat)$ sits in a cartesian square
\begin{center}\begin{tikzcd}
\Fun(\jcat,\ecat)\ar[d,"f_\ast"]\ar[r,"t_\ast"]&\Fun(\jcat,\ccat)\ar[d,"0_\ast"]\\
\Fun(\jcat,\bcat)\ar[r,"k_\ast"]&\Fun(\jcat,\dcat^2)
\end{tikzcd}\end{center}
with
\[
\Fun(\jcat,\dcat^\bullet) \in \Sp((\catinfty)_{/\Fun(\jcat,\ccat)}),
\]
realizing $\Fun(\jcat,\ecat)$ as a linear extension of $\Fun(\jcat,\bcat)$. The corresponding natural system
\[
\tilde{D}\colon \Tw(\Fun(\jcat,\ccat)) \to \RMod_R
\]
is of the following form: given a natural transformation $\alpha\colon F \to F'$ of functors $\jcat\to\ccat$, we have
\begin{align*}
\tilde{D}(\alpha) = \map_{\Fun(\jcat,\dcat^{\bullet+1})}(0F,0F')_\alpha
\simeq \lim_{(\phi\colon i \to j)\in \Tw(\jcat)}\map_{\dcat^{\bullet+1}}(F(i),F'(j))_{F'(\phi)\circ \alpha_i\simeq \alpha_j \circ F(\phi)}.
\end{align*}
In particular, 
\[
\tilde{D}(\id_F) \simeq \lim_{\phi\in \Tw(\jcat)}D(F(\phi)) = \lim F^\ast D
\]
encodes the Baues--Wirsching cohomology of $\jcat$ with coefficients in $F^\ast D$.
\end{ex}

Fix a linear extension as in \cref{def:catsqz}. We can use the discussion of the previous section to understand moduli spaces of lifts along $f\colon \ecat\to\bcat$.

\begin{defn}\label{def:spaceoflifts}
Given a conservative functor $f\colon \ecat\to\bcat$ and object $X\in \bcat$, define spaces $\Lifts_\varepsilon^\ecat(X)$ for $\varepsilon\in\{0,1\}$ by the cartesian squares
\begin{center}\begin{tikzcd}
\Lifts_1^\ecat(X)\ar[r]\ar[d]&\Lifts_0^\ecat(X)\ar[r]\ar[d]&\ecat\ar[d,"f"]\\
\ast\ar[r]&B\!\Aut_\bcat(X)\ar[r]&\bcat
\end{tikzcd}.\end{center}
Thus $\Lifts_0^\ecat(X)\subset\ecat^\core$ is the full subspace of objects $\tilde{X}$ for which there exists an equivalence $f(\tilde{X})\simeq X$, and $\Lifts_1^\ecat(X)$ is the space of objects $\tilde{X} \in \ecat$ equipped with an equivalence $f(\tilde{X})\simeq X$.
\end{defn}

\begin{theorem}\label{thm:spaceoflifts}
Let $f\colon \ecat\to\bcat$ be a linear extension of $\infty$-categories by a natural system $D$, as in \cref{def:catsqz}. Fix $X \in \bcat$ and set $\Lambda = p X \in \ccat$. Then the spaces $\Lifts_\varepsilon^\ecat(X)$ for $\varepsilon\in\{0,1\}$ sit in cartesian squares
\begin{center}\begin{tikzcd}
\Lifts_0^\ecat(X)\ar[r]\ar[d]&B\!\Aut_\ccat(\Lambda)\ar[d,"0"]\\
B\!\Aut_\bcat(X)\ar[r,"k"]&(\Omega^{\infty}\Sigma^2 D(\id_\Lambda))_{\h\!\Aut(\Lambda)}
\end{tikzcd}
$\qquad$
\begin{tikzcd}
\Lifts_1^\ecat(X)\ar[r]\ar[d]&\{0\}\ar[d]\\
\{k(X)\}\ar[r]&\Omega^{\infty}\Sigma^2 D(\id_\Lambda)
\end{tikzcd}.\end{center}
In particular, canonically associated to $X$ is an obstruction class
\[
k(X) \in \pi_{-2}D(\id_\Lambda)
\]
to exhibiting a lift of $X$ to $\ecat$, and a choice of nullhomotopy of $k(X)$ provides an equivalence
\[
\Lifts_1^\ecat(X)\simeq \map_{\dcat^2}(0\Lambda,0\Lambda)_{\id_\Lambda}\simeq\Omega^\infty \Sigma D(\id_\Lambda).
\]
\end{theorem}
\begin{proof}
The defining cartesian square of the linear extension induces a cartesian square
\begin{equation}\label{eq:sqzcore}\begin{tikzcd}
\ecat^\core\ar[r,"f"]\ar[d,"f"]&\ccat^\core\ar[d,"0"]\ar[dr,equals]\\
\bcat^\core\ar[r,"k"]\ar[rr,"p",bend right=8mm]&(\dcat^2)^\core\ar[r,"1"]&\ccat^\core
\end{tikzcd}\end{equation}
of underlying $\infty$-groupoids. Here, we may identify
\[
\bcat^\core\simeq\coprod_{X \in \pi_0 \bcat^\core}B\!\Aut_\bcat(X),\qquad
\ecat^\core\simeq\coprod_{X \in \pi_0\bcat^\core}\Lifts_0^\ecat(X),\qquad \ccat^\core \simeq \coprod_{\Lambda \in \pi_0\ccat^\core}B\!\Aut_\ccat(\Lambda)
\]
by definition, and \cref{lem:mappingspacenaturalsystem} identifies
\begin{align*}
(\dcat^2)^\core&\simeq \colim_{\Lambda \in \ccat^\core}\Omega^\infty \Sigma^2 D(\id_\Lambda)\simeq \coprod_{\Lambda \in \pi_0 \ccat^\core} \left(\Omega^\infty\Sigma^2 D(\id_\Lambda)\right)_{\h\!\Aut_\ccat(\Lambda)}.
\end{align*}
The first cartesian square is therefore the restriction of (\ref{eq:sqzcore}) to the path components living over $\Lambda$, and the second then follows by passing to fibers.
\end{proof}

\subsection{Natural systems on theories}

The rest of this section specializes the above discussion to Malcev theories and their categories of models. We start with some general observations. 

\begin{lemma}\label{lem:cartesiansystem}
Let $D$ be a natural system of connective $R$-modules on a pretheory $\pcat$. The following are equivalent:
\begin{enumerate}
\item The corresponding object $\dcat^\bullet\in \RMod_R((\catinfty)_{/\pcat})$ lifts to $\RMod_R((\pretheories)_{/\pcat})$;
\item For all $P \in \pcat$, the composite
\[
\pcat_{/P}^\op \simeq \Tw(\pcat)\times_{\pcat^\op\times\pcat}(\pcat^\op\times\{P\}) \to \Tw(\pcat) \xrightarrow{D} \RMod_R^{\geq 0} \xrightarrow{\Omega^\infty} \spaces
\]
preserves products, i.e.\ defines a model of $\pcat_{/P}$.
\item For any set of maps $\{f_i\colon P_i \to P\}$ with sum $f\colon \coprod_i P_i \to P$, the comparison
\[
D(f\colon \coprod_i P_i \to P) \to \prod_i D(f_i\colon P_i \to P)
\]
is an equivalence.
\end{enumerate}
\end{lemma}
\begin{proof}
Clearly (2) and (3) are equivalent, so we must show that these are equivalent to (1). Given $\dcat^\bullet \in \RMod_R((\catinfty)_{/\pcat})$ and a set of maps $\{f_i\colon P_i \to P\}$ with sum $f\colon \coprod_i P_i \to P$, we may identify the given comparison map as
\[
D(f) \simeq \map_{\dcat^1}(0(\coprod_i P_i),0(P))_{f} \to \prod_i \map_{\dcat^1}(0(P_i),0(P))_{f_i} \simeq \prod_i \map_{\dcat^1}(0(P_i),0(P))_{f_i}.
\]
This is an equivalence for any set of maps $\{f_i\}$ if and only the projections $0(\coprod_i P_i) \to 0(P_i)$ realize $0(\coprod_i P_i)\simeq \coprod_i 0(P_i)$ in $\dcat^1$ for any set of objects $\{P_i\}$ in $\pcat$. As $0\colon \pcat\to\dcat^1$ is essentially surjective, this proves that (2) is equivalent to the condition that $\dcat^1$ is a pretheory and $0\colon \pcat\to\dcat^1$ is a homomorphism of pretheories.

It remains to show that this implies that the entire diagram $\dcat^\bullet\colon \calR_-^\omega(R)^\op \to(\catinfty)_{/\pcat}$ consists of pretheories and homomorphisms. Applying the above with $D$ replaced by a sum of suspensions of $D$, we find that $\dcat^\bullet(\Omega F)$ is a pretheory for any $F\in \calR_-^\omega(R)$, and that $0\colon \pcat\to\dcat^\bullet(\Omega F)$ is a homomorphism of pretheories. The cartesian square
\begin{center}\begin{tikzcd}
\dcat^\bullet(F)\ar[r,"q"]\ar[d,"q"]&\pcat\ar[d,"0"]\\
\pcat\ar[r,"0"]&\dcat^\bullet(\Sigma F)
\end{tikzcd}\end{center}
then implies that $\dcat^\bullet(F)$ is a pretheory for any $F \in \calR_-^\omega(R)$ and that both $0\colon \pcat\to\dcat^\bullet(F)$ and $q\colon \dcat^\bullet(F)\to\pcat$ are homomorphisms of pretheories. Moreover, $q$ is conservative.

It remains to show that if $F' \to F$ is any map in $\calR_-^\omega(R)$, then $f\colon \dcat^\bullet(F) \to \dcat^\bullet(F')$ is a homomorphism. Consider the diagram
\begin{center}\begin{tikzcd}
\dcat^\bullet(F)\ar[rr,"f"]\ar[dr,"q"']&&\dcat^\bullet(F')\ar[dl,"q'"]\\
&\pcat
\end{tikzcd}.\end{center}
As $q'$ is a conservative homomorphism, to prove that $f$ is a homomorphism it suffices to prove that $q'\circ f = q$ is a homomorphism, which we have already done.
\end{proof}

\begin{defn}\label{def:cartesiansystems}
A \emph{cartesian system of connective $R$-modules} on a pretheory $\pcat$ is a natural system of connective $R$-modules satisfying the equivalent characterizations of \cref{lem:cartesiansystem}.
\end{defn}

Combining \cref{cor:essurjzerosection} with \cite[Remark 3.1.3, Lemma 4.1.5]{usd1} shows that if $D$ is a cartesian system of connective $R$-modules on a (Malcev) theory $\pcat$, then $\dcat^{\bullet+1}$ again consists of (Malcev) theories. Our work in this paper now shows the following.

\begin{prop}\label{prop:animatenaturalsystem}
Let $D$ be a cartesian system of connective $R$-modules on a Malcev theory $\pcat$. Then $\dcat^{\bullet+1} \in \RMod_R(\malcevtheories_{/\pcat})$ induces
\[
\Model_{\dcat^{\bullet+1}}\in \RMod_R((\catinfty)_{/\Model_\pcat}).
\]
The corresponding natural system is of the form
\[
D_!\colon \Tw(\Model_\pcat)\to\RMod_R,\qquad D_!(\phi\colon \Gamma\to\Lambda) = \map_{\pcat/\Lambda}(\Gamma,B^\bullet_\Lambda\Lambda_D),
\]
where $\Lambda_D \in \RMod_R^{\geq 0}((\Model_\pcat)_{/\Lambda})$ is a connective $R$-module object given by
\[
\Lambda_D = 0^\ast 0_! \Lambda,\qquad \text{where}\qquad 0\colon \pcat\to\dcat^{\bullet+1}.
\]
\end{prop}
\begin{proof}
We first show that $\Model_{\dcat^{\bullet+1}}\in \RMod_R((\catinfty)_{/\Model_\pcat})$. Consider the composite
\[
\calR_-^\omega(R)^\op \to \catinfty,\qquad P \mapsto \Model_{\dcat^{\bullet+1}(P)} = \Model_{\dcat^\bullet(\Omega P)}.
\]
This sends $0 \mapsto \Model_\pcat$, and to show that the induced functor
\[
\calR_-^\omega(R)^\op \to (\catinfty)_{/\Model_\pcat}
\]
defines an $R$-module object it suffices to show that if $P,P'\in \calR_+^\omega(R)$ then both maps
\begin{align*}
\Model_{\dcat(\Omega P \oplus \Omega P')} &\to \Model_{\dcat(\Omega P)}\times_{\Model_\pcat}\Model_{\dcat(\Omega P')},\\
\Model_{\dcat(\Omega^2 P)} &\to \Model_\pcat\times_{\Model_\dcat(\Omega P)}\Model_\pcat
\end{align*}
are equivalences. As the homomorphisms $\dcat(\Omega P) \to \pcat \leftarrow \dcat(\Omega P')$ and $\pcat \to \dcat(\Omega P)$ are full, this follows from \cref{thm:animatesquaregeneral}.
The associated natural system $D_!$ satisfies
\[
D_!(\phi) = \map_{\dcat^{\bullet+1}}(0_!\Gamma,0_!\Lambda)_\phi = \map_{\pcat/\Lambda}(\Gamma,0^\ast 0_!\Lambda)
\]
by construction and the adjunction $0_! \dashv 0^\ast$, proving the final statement.
\end{proof}

\begin{rmk}
If $\Lambda = \nu P$ for $P\in \pcat$, then
\[
(\nu P)_D \in \RMod_R^{\geq 0}((\Model_\pcat)_{/\nu P})\simeq \RMod_R^{\geq 0} \otimes \Model_{\pcat_{/P}}
\]
is encoded by the product-preserving functor
\[
\pcat_{/P}^\op \to \Tw(\pcat)\xrightarrow{D} \RMod_R^{\geq 0}
\]
of \cref{lem:cartesiansystem}.(2). 
\end{rmk}

\begin{ex}\label{ex:lambdahpn}
Let $\pcat$ be a Malcev theory, and consider the cartesian system $H\Pi_{[n,n+r)}$ of connective $\thesphere_{<r}$-modules on $\h_r\pcat$ described in \cref{ex:pinrnaturalsystem}. By \cref{lem:pseudocotensor}, we can identify $\Lambda_{H\Pi_{[n,n+r)}} = B^n_\Lambda\Lambda_{S^n}$.
\end{ex}

\begin{notation}
\label{notation:andre_quillen_cohomology}
Given a map $\phi\colon \Gamma\to\Lambda$, one can think of 
\[
\Omega^{\infty-n} D_!(\phi) \simeq \map_{\dcat^{n+1}}(\Gamma,\Lambda)_\phi \simeq \map_{\pcat/\Lambda}(\Gamma,B^n_\Lambda\Lambda_D)
\]
as an \emph{André-Quillen cohomology} space. We write  
\[
\Hrm^n_{\pcat/\Lambda}(\Gamma;\Lambda_D) \colonequals \pi_{-n}D_!(\phi)\simeq \pi_0 \map_{\pcat/\Lambda}(\Gamma,B^n_\Lambda\Lambda_D).
\]
\end{notation}

\subsection{Linear extensions of Malcev theories}

The classical square-zero extension theory of rings admits a nonabelian generalization to the linear extensions theories, studied for ordinary Lawvere theories by Jibladze--Pirashvili \cite{jibladzepirashvili1991cohomology,jibladzepirashvili2005linear}. Fix a cartesian system $D$ of connective $R$-modules on a Malcev theory $\tcat$.

\begin{prop}\label{prop:linearextensionsaremalcev}
Fix a homomorphism $p\colon \pcat\to\tcat$, and let $f\colon\qcat\to\pcat$ be a linear extension of the underlying $\infty$-category of $\pcat$ by $D$. Then $\qcat$ is a Malcev theory and $f$ is a homomorphism.
\end{prop}
\begin{proof}
Consider the defining cartesian diagram of the linear extension:
\begin{center}\begin{tikzcd}
\qcat\ar[r,"t"]\ar[d,"f"]&\tcat\ar[d,"0"]\ar[dr,equals]\\
\pcat\ar[r,"k"]\ar[rr,"p",bend right]&\dcat^2\ar[r,"q"]&\tcat
\end{tikzcd}.\end{center}
We first claim that $\qcat$ is a pretheory and $f$ is a homomorphism of pretheories. It suffices to prove that $k$ is a homomorphism of pretheories. As $q$ is a conservative homomorphism of pretheories, it suffices to check that $q\circ k \simeq p$ is a homomorphism of pretheories, which holds by assumption.

We next claim that $\qcat$ is Malcev. As $D$ is connective, $0\colon \tcat\to\dcat^2$ is full, so this follows from \cite[Proposition 4.1.6.(2)]{usd1}.

We finally claim that $\qcat$ is a theory. As $D$ is connective, $f$ is essentially surjective. As $f$ is conservative, it therefore induces a bijection between equivalence classes of objects. Thus any set of generating objects for $\pcat$ lifts to a set of generating objects for $\qcat$.
\end{proof}

\begin{ex}
Let $B$ be a connective ring spectrum and $I$ be a connective $B$-bimodule. Then square-zero extensions of $B$ by $I$ are equivalent to linear extensions of $\cfrees(B)$ by $\cfrees(B\oplus I)$, where $B \oplus I \in \Sp(\Alg_{\bfE_1/B})$ is the trivial square-zero extension associated to $I$.
\end{ex}

\begin{ex}
Let $\pcat$ be a Malcev theory and $F$ be an elementary modifier. Then the proof of \cref{prop:discretereflect} shows that the homomorphism $\epsilon_F \colon \h(\pcat_F)\to\h\pcat$ factors as a finite tower of linear extensions by quotients of the natural system $H\Pi_n = \Omega^n H\Pi_{[n,n+1)}$ of abelian groups on $\h\pcat$.
\end{ex}

We can now state the following.

\begin{theorem}\label{thm:linearextension}
Let $f\colon \qcat\to\pcat$ be a linear extension of $\pcat$ by a cartesian system $D$ of connective $R$-modules. Then $f_!\colon \Model_\qcat\to\Model_\pcat$ is a linear extension of $\Model_\pcat$ by the natural system $D_!$ of \cref{prop:animatenaturalsystem}.
\end{theorem}
\begin{proof}
As $0\colon \tcat\to\dcat^2$ induces an equivalence $\h(\tcat)\simeq\h(\dcat^2)$, this is an immediate application of \cref{thm:animatesquaregeneral} to the defining cartesian square of the linear extension.
\end{proof}

\begin{ex}
Let $\pcat$ be a Malcev theory. Then 
\[
\tau_{(n+r,n)!}\colon \Model_{\h_{n+r}\pcat}\to\Model_{\h_n\pcat}
\]
is a linear extension of $\Model_{\h_n\pcat}$ by the natural system $H\Pi_{[n,n+r)!}$ satisfying
\[
H\Pi_{[n,n+r)!}(\phi\colon \Gamma\to\Lambda) = \map_{\h_r\pcat/\Lambda}(\Gamma,B^{n+\bullet}_\Lambda\Lambda_{S^n}).
\]
This follows by combining \cref{thm:linearextension} with \cref{ex:postnikovlinearextension} and \cref{ex:lambdahpn}.
\end{ex}

\subsection{Decomposing moduli spaces of lifts}

All of the above formalism comes together to address the following fundamental question: given a linear extension $f\colon \qcat\to \pcat$, when can structure in $\Model_\pcat$ be lifted to structure in $\Model_\qcat$? To set notation, fix for the rest of this section a diagram
\begin{center}\begin{tikzcd}
\qcat\ar[r,"t"]\ar[d,"f"]&\tcat\ar[dr,equals]\ar[d,"0"]\\
\pcat\ar[r,"k"]\ar[rr,"p"',bend right]&\dcat^2\ar[r,"q"]&\tcat
\end{tikzcd}\end{center}
of Malcev theories, with $\dcat^\bullet$ associated to a cartesian system $D$ of connective $R$-modules on $\tcat$, realizing $\qcat$ as a linear extension of $\pcat$ by $D$.

\begin{theorem}\label{thm:mainlift}
Fix $X \in \Model_\pcat$,  and define spaces
$\Lifts_\varepsilon^\qcat(X)$ for $\varepsilon\in\{0,1\}$ by the cartesian squares
\begin{center}\begin{tikzcd}
\Lifts_1^\qcat(X)\ar[r]\ar[d]&\Lifts_0^\qcat(X)\ar[r]\ar[d]&\Model_\qcat\ar[d,"f_!"]\\
\{X\}\ar[r]&B\!\Aut_\pcat(X)\ar[r]&\Model_\pcat
\end{tikzcd}.\end{center}
Set $\Lambda = p_! X \in \Model_\tcat$. Then these spaces sit in canonical cartesian squares
\begin{center}\begin{tikzcd}[column sep=tiny]
\Lifts_0^\qcat(X)\ar[r]\ar[d]&B\!\Aut_\tcat(\Lambda)\ar[d,"0"]\\
B\!\Aut_\pcat(X)\ar[r,"k"]&\map_{\tcat/\Lambda}(p_!\Lambda,B^2_{p_!\Lambda}(p_!\Lambda)_D)_{\h\!\Aut(p_!\Lambda)}
\end{tikzcd},
\begin{tikzcd}[column sep=tiny]
\Lifts_1^\qcat(X)\ar[r]\ar[d]&\{0\}\ar[d]\\
\{k_!(X)\}\ar[r]&\map_{\tcat/p_!\Lambda}(p_!\Lambda,B^2_{p_!\Lambda}(p_!\Lambda)_D)
\end{tikzcd}.\end{center}
In particular, there is a canonical obstruction class
\[
k_!(X) \in \Hrm^2_{\tcat/p_!\Lambda}(p_!\Lambda;(p_!\Lambda)_D)
\]
to finding a lift of $X$ to a model of $\qcat$, and a choice of nullhomotopy of $k_!(X)$ induces an equivalence
\[
\Lifts_1^\qcat(X) \simeq \map_{\tcat/p_!\Lambda}(p_!\Lambda,B_{p_!\Lambda}(p_!\Lambda)_D),
\]
and therefore a bijection $\pi_0 \Lifts_1^\qcat(X)\simeq \Hrm^1_{\tcat/p_!\Lambda}(p_!\Lambda;(p_!\Lambda)_D)$.
\end{theorem}
\begin{proof}
By \cref{thm:linearextension}, $f_!\colon \Model_\qcat\to\Model_\pcat$ is a linear extension by a natural system $D_!$ on $\Model_\tcat$ satisfying
\[
D_!(\id_\Lambda) = \map_{\tcat/p_!\Lambda}(p_!\Lambda,B_{p_!\Lambda}^\bullet(p_!\Lambda)_D).
\]
The theorem then follows from \cref{thm:spaceoflifts}.
\end{proof}

\begin{rmk}
One may also identify
\[
\map_{\tcat/p_!\Lambda}(p_!\Gamma,B_{p_!\Lambda}^\bullet(p_!\Lambda)_D)\simeq \map_{\pcat/\Lambda}(\Gamma,B_{\Lambda}^\bullet \Lambda_{p^\ast D})
\]
for any morphism $\Gamma\to\Lambda$ in $\Model_\pcat$; compare \cref{rmk:reducetoid}.
\end{rmk}

The same proof as \cref{thm:mainlift}, applied to the linear extensions derived from $f_!\colon \Model_\qcat\to\Model_\pcat$ via \cref{ex:mappinglinext}  and \cref{ex:funlinext}, yields the following two theorems:

\begin{theorem}
\label{thm:linextspaceoflifts}
Fix $Y,X\in \Model_\qcat$. Then there is canonical cartesian square
\begin{center}\begin{tikzcd}
\map_\qcat(Y,X)\ar[r]\ar[d]&\map_\tcat(t_!Y,t_!X)\ar[d,"0"]\\
\map_\pcat(f_!Y,f_!X)\ar[r]&\map_\tcat(t_!Y,B_{t_!X}(t_!X)_D)
\end{tikzcd}.\end{center}
If we fix a map $\phi\colon t_! Y \to t_! X$, then this restricts to
\begin{center}\begin{tikzcd}
\map_\qcat(Y,X)_\phi\ar[r]\ar[d]&\{0\}\ar[d]\\
\map_\pcat(f_!Y,f_!X)_\phi\ar[r,"k_!"]&\map_{\tcat/t_!X}(t_!Y,B_{t_!X}(t_!X)_D)
\end{tikzcd}.\end{center}
In particular, given a map $\alpha\colon f_! Y \to f_! X$ lifting $\phi$, there is an obstruction 
\[
k_!(\alpha) \in \Hrm^1_{\tcat/t_!X}(t_!Y,(t_!X)_D)
\]
 to exhibiting a further lift to $Y \to X$.
\qed
\end{theorem}

\begin{remark}
As it only concerns a decomposition of mapping spaces, rather than the space of lifts of objects, \cref{thm:linextspaceoflifts} can be alternatively deduced from the much more simple \cref{prop:weakpb}, as with \cref{cor:mappingspacedecomposition}.
\end{remark}

\begin{theorem}\label{thm:liftdiagrams}
Let $F\colon \jcat\to\Model_\pcat$ be a diagram in $\Model_\pcat$, and let $\Lifts^\qcat_1(F)$ be the space of lifts to $\Model_\qcat$:
\[
\Lifts^\qcat_1(F) = \Fun_{/\Model_\pcat}(\jcat,\Model_\qcat) = \left\{\begin{tikzcd}
&\Model_\qcat\ar[d,"f_!"]\\
\jcat\ar[r,"F"]\ar[ur,dashed]&\Model_\pcat
\end{tikzcd}\right\}
\]
Then there is a cartesian diagram of the form
\begin{center}\begin{tikzcd}
\Lifts_1^\qcat(F)\ar[r]\ar[d]&\{0\}\ar[d]\\
\{k_!(F)\}\ar[r]&\map_{\Fun(\jcat,\Model_\tcat)_{/F}}(F,B^2_FF_D)
\end{tikzcd}.\end{center}
Here, one may also identify
\[
\map_{\Fun(\jcat,\Model_\tcat)/F}(F,B^2_FF_D)\simeq \lim\limits_{(\phi\colon i \to j)\in \Tw(\jcat)}\map_{\tcat/F(j)}(F(i),B^2_{F(j)}F(j)_D).
\]
In particular, such a lift exists if and only if $k_!(F)\sim 0$.
\qed
\end{theorem}

\cref{introthm:modulidecomposition} as stated in the introduction is a minor variant of \cref{thm:mainlift} in the case of the linear extension $\h_{n+r}\pcat\to\h_n\pcat$, which we spell out in detail: 

\begin{theorem}\label{thm:modulidecomposition}
Let $\pcat$ be a Malcev theory. Fix $\Lambda \in \Model_{\h_r\pcat}$, and let
\[
\calM(\Lambda) \subset \Model_\pcat^\core
\]
be the full subspace of $X \in \Model_\pcat$ for which $\tau_{r!} X \simeq \Lambda$. Then there is a decomposition
\[
\calM(\Lambda)\simeq\lim_{r\leq n\to\infty}\calM_n(\Lambda),
\]
with layers fitting into cartesian squares
\begin{equation}\label{eq:modulidecomposition}\begin{tikzcd}
\calM_{n+r}(\Lambda)\ar[r]\ar[d]&B\!\Aut(\Lambda)\ar[d,"0_!"]\\
\calM_n(\Lambda)\ar[r,"k_!"]&\map_{\h_r\pcat/\Lambda}(\Lambda,B^{n+2}_\Lambda\Lambda_{S^n})_{\h\!\Aut(\Lambda)}
\end{tikzcd}\end{equation}
for all $r\leq n$.
\end{theorem}
\begin{proof}
By construction,
\[
\calM(\Lambda) \simeq B \!\Aut(\Lambda)\times_{\Model_{\h_r\pcat}}\Model_\pcat
\]
and
\[
\calM_n(\Lambda) \simeq B\!\Aut(\Lambda)\times_{\Model_{\h_r\pcat}}\Model_{\h_n\pcat}.
\]
Convergence of the spiral tower of $\pcat$, proved in \cref{thm:modelspiral}, therefore implies
\begin{align*}
\calM(\Lambda)&\simeq B \!\Aut(\Lambda)\times_{\Model_{\h_r\pcat}}\Model_\pcat \simeq B\!\Aut(\Lambda)\times_{\Model_{\h_r\pcat}}\lim_{r\leq n \to\infty}\Model_{\h_n\pcat}\\
&\simeq \lim_{r\leq n \to\infty}B\!\Aut(\Lambda)\times_{\Model_{\h_r\pcat}}\Model_{\h_n\pcat} \simeq \lim_{r\leq n \to\infty}\calM_n(\Lambda).
\end{align*}
By \cref{thm:linearextension}, the linear extension $\h_{n+r}\pcat\to\h_n\pcat$ induces a linear extension of $\infty$-categories of models. Passing to underlying $\infty$-groupoids, we obtain a cartesian square
\begin{center}\begin{tikzcd}
\Model_{\h_{n+r}\pcat}^\core\ar[r,"\tau_{(n+r,r)!}"]\ar[d,"\tau_{(n+r,n)!}"]&\Model_{\h_r\pcat}^\core\ar[d,"0_!"]\ar[dr,equals]\\
\Model_{\h_n\pcat}^\core\ar[r,"k_!"]&\Model_{\kinv_{n,r}^{n+1}\pcat}^\core\ar[r,"q_!"]&\Model_{\h_r\pcat}^\core
\end{tikzcd}\end{center}
of spaces. Restricting this square to those path components living over $B\!\Aut(\Lambda)\subset \Model_{\h_r\pcat}^\core$ provides the desired cartesian square.
\end{proof}

\begingroup
\raggedright
\bibliographystyle{amsalpha}
\bibliography{bibliography}
\endgroup

\end{document}

%% file: preamble.tex
\usepackage[OT2,T1]{fontenc}
\usepackage[utf8]{inputenc}
\usepackage{amsmath, amsthm, amsfonts, amssymb}

\usepackage[dvipsnames]{xcolor} % colourfultext
\usepackage{url}
\usepackage{nicefrac}
\usepackage[all]{xy}
\usepackage{placeins} % barriers for floats
\usepackage{colonequals}

\usepackage{euscript} % Lurie-type Euler script
\usepackage{bbm} % mathbb numbers
\usepackage{enumitem} % fancy enumerate

\usepackage{hyperref}
\hypersetup{
    colorlinks, linkcolor=NavyBlue,
    citecolor=OliveGreen, urlcolor=RedOrange
}
\usepackage[nameinlink,capitalise,noabbrev]{cleveref} % \Cref{prop:1.4} prints ``Proposition 1.4''
\usepackage{todonotes}
\setuptodonotes{fancyline, color=green!30}
\usepackage{faktor}

\usepackage{tikz-cd}
\usepackage{tikz}
\usetikzlibrary{arrows,backgrounds,
    decorations.pathreplacing,
    decorations.pathmorphing
}
\usetikzlibrary{matrix,arrows}
\usepackage{colonequals}
\theoremstyle{plain}

\numberwithin{equation}{subsection}

\theoremstyle{definition} % I've moved this up here since I have trouble reading large blocks of italic text. 
\newtheorem{theorem}[equation]{Theorem}
\newtheorem{lemma}[equation]{Lemma}

\newtheorem{prop}[equation]{Proposition}

\newtheorem{cor}[equation]{Corollary}

\newtheorem{construction}[equation]{Construction}
\newtheorem{variation}[equation]{Variation}
\newtheorem{definition}[equation]{Definition}
\newtheorem{defn}[equation]{Definition}

\newtheorem{ex}[equation]{Example}
\newtheorem{notation}[equation]{Notation}

\newtheorem{recollection}[equation]{Recollection}
\newtheorem{remark}[equation]{Remark}
\newtheorem{rmk}[equation]{Remark}
\newtheorem{warning}[equation]{Warning}

\newtheorem*{remark*}{Remark}
\newtheorem*{terminology*}{Terminology}
\newtheorem*{interpretation*}{Interpretation}
\newtheorem*{definition*}{Definition}
\newtheorem*{conjecture*}{Conjecture}
\newtheorem*{notation*}{Notation}
\newtheorem*{convention*}{Convention}

\input{commands}

\makeatletter
  \def\subsection{\@startsection{subsection}{1}%
  \z@{.7\linespacing\@plus\linespacing}{.5\linespacing}%
  {\normalfont\bfseries\centering}}% NEW
\makeatother

\let\oldtocsection=\tocsection
\let\oldtocsubsection=\tocsubsection
\let\oldtocsubsubsection=\tocsubsubsection
\renewcommand{\tocsection}[2]{\hspace{0em}\oldtocsection{#1}{#2}}
\renewcommand{\tocsubsection}[2]{\hspace{1em}\oldtocsubsection{#1}{#2}}
\renewcommand{\tocsubsubsection}[2]{\hspace{2em}\oldtocsubsubsection{#1}{#2}}

\calclayout

%% file: commands.tex
\newcommand{\Hrm}{\mathrm{H}}

\newcommand{\euscr}[1]{\EuScript{#1}} % Euler script
\newcommand{\bcat}{\euscr{B}} % category A in Euler script
\newcommand{\ccat}{\euscr{C}} % category C in Euler script 
\newcommand{\dcat}{\euscr{D}} % category D in Euler script
\newcommand{\ecat}{\euscr{E}} % category E in Euler script 
\newcommand{\pcat}{\euscr{P}} % category P in Euler script
\newcommand{\vcat}{\euscr{V}} % category V in Euler script
\newcommand{\kcat}{\euscr{K}} % category E in Euler script
\newcommand{\qcat}{\euscr{Q}} % category Q in Euler script
\newcommand{\rcat}{\euscr{R}}
\newcommand{\jcat}{\euscr{J}}
\newcommand{\tcat}{\euscr{T}}

\newcommand{\hcat}{\euscr{H}}

\newcommand{\xcat}{\euscr{X}}

\newcommand{\ctwocat}{\mathsf{C}}

\newcommand{\Fss}{F_\bullet}

\newcommand{\calM}{\euscr{M}}

\newcommand{\calO}{\euscr{O}}
\newcommand{\calV}{\euscr{V}}
\newcommand{\calJ}{\euscr{J}}
\newcommand{\calU}{\euscr{U}}

\newcommand{\calL}{\euscr{L}}
\newcommand{\calR}{\euscr{R}}

\newcommand{\sfO}{\mathsf{O}}
\newcommand{\sfC}{\mathsf{C}}
\newcommand{\sfD}{\mathsf{D}}

\newcommand{\bbZ}{\mathbb{Z}}

\newcommand{\bfE}{\mathbf{E}}

\DeclareMathOperator{\Mod}{Mod}
\DeclareMathOperator{\map}{map}

\DeclareMathOperator{\Fun}{Fun}

\DeclareMathOperator{\Alg}{Alg}
\DeclareMathOperator{\CAlg}{CAlg}

\DeclareMathOperator{\End}{End}
\DeclareMathOperator{\Aut}{Aut}
\DeclareMathOperator{\Spc}{Spc} % motivic spaces
\DeclareMathOperator{\Spiral}{Spiral}

\DeclareMathOperator{\Ext}{Ext}

\DeclareMathOperator{\Tw}{Tw}
\DeclareMathOperator{\aut}{aut}
\DeclareMathOperator{\iso}{iso}
\DeclareMathOperator{\Fib}{Fib}
\DeclareMathOperator{\tFib}{tFib}

\DeclareMathOperator*{\colim}{colim}
\newcommand{\spaces}{\euscr{S}\mathrm{pc}} % the category of spaces
\newcommand{\catinfty}{\euscr{C}\mathrm{at}_{\infty}} % the category of infty categories
\newcommand{\pretheories}{\widehat{\euscr{T}\mathrm{hry}}}
\newcommand{\spectra}{\euscr{S}\mathrm{p}} % the category of spectra
\newcommand{\malcevtheories}{\euscr{M}\mathrm{alc}} % Category of Malcev theories and homomorphisms
\newcommand{\Malc}{\euscr{M}\mathrm{alc}} % Category of Malcev theories and functors
\newcommand{\ukanspaces}{\euscr{S}\mathrm{ps}} % category of universally Kan spaces
\newcommand{\malcevmodifiers}{{\End_\sigma(\ukanspaces)}} % category of modifiers on ukan spaces
\newcommand{\elementarymodifiers}{\euscr{E}\mathrm{lm}} % category of elementary modifiers
\newcommand{\spheresandmore}{\euscr{S}\mathrm{ph}}

\newcommand{\cfrees}{\euscr{L}_0}

\newcommand{\lfrees}{\euscr{L}}

\newcommand{\presheaves}{\euscr{P}\mathrm{sh}}

\newcommand{\LMod}{\mathrm{L}\euscr{M}\mathrm{od}}
\newcommand{\RMod}{\mathrm{R}\euscr{M}\mathrm{od}}

\newcommand{\Model}{\euscr{M}\mathrm{odel}}

\newcommand{\Cat}{{\euscr{C}\mathrm{at}}}
\newcommand{\Ab}{\euscr{A}\mathrm{b}}
\newcommand{\Sph}{\euscr{S}\mathrm{ph}}

\newcommand{\Sp}{{\euscr{S}\mathrm{p}}}

\newcommand{\Lifts}{\mathrm{Lifts}}

\newcommand{\Ob}{\mathrm{Ob}} 
 
\newcommand{\kinv}{\mathrm{k}}

\usepackage{bbm}

\newcommand{\integers}{\mathbb{Z}}
\newcommand{\thesphere}{\mathbf{S}}
\newcommand{\pt}{\mathrm{pt}}
\newcommand{\id}{\mathrm{id}}

\newcommand{\h}{\mathrm{h}}

\newcommand{\op}{\mathrm{op}}

\newcommand{\cn}{\mathrm{cn}}

\newcommand{\core}{{\text{core}}}

\newcommand{\ns}{\mathrm{ns}}

\newcommand{\Rezk}{\mathrm{Rezk}}

\newcommand{\nb}{\mathrm{nb}}

\newcommand{\bs}{{-}}
\newcommand{\ol}[1]{\overline{#1}}

\usepackage{todonotes}

\makeatletter
\newcommand{\pto}{}% just for safety
\newcommand{\pgets}{}% just for safety

\DeclareRobustCommand{\pto}{\mathrel{\mathpalette\p@to@gets\to}}
\DeclareRobustCommand{\pgets}{\mathrel{\mathpalette\p@to@gets\gets}}

\newcommand{\p@to@gets}[2]{%
  \ooalign{\hidewidth$\m@th#1\mapstochar\mkern5mu$\hidewidth\cr$\m@th#1\to$\cr}%
}
\makeatother

\usetikzlibrary{decorations.markings}
\tikzset{mid vert/.style={/utils/exec=\tikzset{every node/.append style={outer sep=0.8ex}},
postaction=decorate,decoration={markings,
mark=at position 0.5 with {\draw[-] (0,#1) -- (0,-#1);}}},
mid vert/.default=0.75ex}